\documentclass[12pt]{amsart}
\usepackage{amsmath,amssymb,amsthm,amscd}

\usepackage{graphicx}

\newtheorem{formula}{}[section]
\newtheorem{proposition}[formula]{Proposition}
\newtheorem{corollary}[formula]{Corollary}
\newtheorem{lemma}[formula]{Lemma}
\newtheorem{theorem}[formula]{Theorem}

\theoremstyle{definition}
\newtheorem{definition}[formula]{Definition}

\newtheorem{notation}[formula]{Notation}

\theoremstyle{remark}
\newtheorem{remark}{Remark}

\begin{document}

\title[$B$-rigidity of the property to be an almost Pogorelov polytope]{$B$-rigidity of the property to be an almost Pogorelov polytope}
\author[N.Yu.~Erokhovets]{Nikolai~Erokhovets}
\address{Department of Mechanics and Mathematics, Moscow State University, 119992 Moscow, Russia}
\email{erochovetsn@hotmail.com}

\def\sgn{\mathrm{sgn}\,}
\def\bideg{\mathrm{bideg}\,}
\def\tdeg{\mathrm{tdeg}\,}
\def\sdeg{\mathrm{sdeg}\,}
\def\grad{\mathrm{grad}\,}
\def\ch{\mathrm{ch}\,}
\def\sh{\mathrm{sh}\,}
\def\th{\mathrm{th}\,}

\def\mod{\mathrm{mod}\,}
\def\In{\mathrm{In}\,}
\def\Im{\mathrm{Im}\,}
\def\Ker{\mathrm{Ker}\,}
\def\Hom{\mathrm{Hom}\,}
\def\Tor{\mathrm{Tor}\,}
\def\rk{\mathrm{rk}\,}
\def\codim{\mathrm{codim}\,}

\def\ko{{\mathbf k}}
\def\sk{\mathrm{sk}\,}
\def\RC{\mathrm{RC}\,}
\def\gr{\mathrm{gr}\,}

\def\R{{\mathbb R}}
\def\C{{\mathbb C}}
\def\Z{{\mathbb Z}}
\def\A{{\mathcal A}}
\def\B{{\mathcal B}}
\def\K{{\mathcal K}}
\def\M{{\mathcal M}}
\def\N{{\mathcal N}}
\def\E{{\mathcal E}}
\def\G{{\mathcal G}}
\def\D{{\mathcal D}}
\def\F{{\mathcal F}}
\def\L{{\mathcal L}}
\def\V{{\mathcal V}}
\def\H{{\mathcal H}}



\thanks{The research was supported by the RFBR grant No 18-51-50005}



\begin{abstract}
Toric topology assigns to each $n$-dimensional combinatorial simple convex polytope $P$ with $m$ facets an $(m+n)$-dimensional moment-angle manifold $\mathcal{Z}_P$ with an action of a compact torus $T^m$ such that $\mathcal{Z}_P/T^m$ is a convex polytope of combinatorial type $P$. We study the notion of $B$-rigidity. A property of a polytope $P$ is called $B$-rigid, if any isomorphism of graded rings $H^*(\mathcal{Z}_P,\mathbb Z)= H^*(\mathcal{Z}_Q,\mathbb Z)$ for a simple $n$-polytope $Q$ implies that it also has this property. We study families of $3$-dimensional polytopes defined by their cyclic $k$-edge-connectivity. These families include flag polytopes and Pogorelov polytopes, that is polytopes realizable as bounded right-angled polytopes in Lobachevsky space $\mathbb L^3$. Pogorelov polytopes include fullerenes -- simple polytopes with only pentagonal and hexagonal faces. It is known that the properties to be a flag $3$-dimensional polytope or a Pogorelov polytope are $B$-rigid. We focus on almost Pogorelov polytopes, which are strongly cyclically $4$-edge-connected polytopes. They correspond to right-angled polytopes of finite volume in $\mathbb L^3$. There is a subfamily of ideal almost Pogorelov polytopes corresponding to ideal right-angled polytopes. We prove that the properties to be an almost Pogorelov polytope and an ideal almost Pogorelov polytope are $B$-rigid. As a corollary we obtain that $3$-dimensional associahedron $As^3$ and permutohedron $Pe^3$ are $B$-rigid.
We generalize methods known for Pogorelov polytopes. We obtain results on $B$-rigidity of subsets in $H^*(\mathcal{Z}_P,\mathbb Z)$ and prove an analog of the so-called separable circuit condition (SCC). As an example we consider the ring $H^*(\mathcal{Z}_{As^3},\mathbb Z)$.
\end{abstract}

\maketitle

\setcounter{section}{0}

\section*{Introduction}
Consider a unit circle $S^1\subset\mathbb C$ and a unit disk $D^2\subset\mathbb C$.
Toric topology (see \cite{BP15}) assigns to each simple convex $n$-polytope $P$ with $m$ faces $F_1,\dots,F_m$ an $(m+n)$-dimensional {\it moment-angle manifold}  $\mathcal{Z}_P$ with an action of a compact $m$-dimensional torus $T^m=(S^1)^m$ such that $\mathcal{Z}_P/T^m=P$. One of the ways to define the space $\mathcal{Z}_P$ is 
$$
\mathcal{Z}_P= T^m\times P^n/\sim,\text{ where }(t_1,p_1)\sim(t_2,p_2)\text{ if and only if }p_1=p_2,\text{ and }t_1t_2^{-1}\in T^{G(p_1)},
$$ 
where $T^{G(p)}=\{(t_1,\dots,t_m)\in T^m\colon t_i=1\text{ for }F_i\not\ni p\}$.

It can be shown that topological type of the space $\mathcal{Z}_P$ depends only on combinatorial type of $P$ and $\mathcal{Z}_P$ has a smooth structure such that the action of $T^m$ obtained from its action on the first factor, is smooth.

\begin{definition}
A simple $n$-polytope $P$ is called {\it $B$-rigid}, if for any simple $n$-polytope $Q$ any isomorphism of graded rings 
$H^*(\mathcal{Z}_P,\mathbb Z)\simeq H^*(\mathcal{Z}_Q,\mathbb Z)$ implies that $P$ and $Q$ are combinatorially equivalent.
\end{definition}


We study families of combinatorial simple $3$-polytopes defined by the condition of a {\it cyclic edge $k$-connectivity} ({\it $ck$-connectivity} for short). For $3$-polytopes by {\it faces} we mean facets. Two faces are {\it adjacent}, if they have a common edge.
\begin{definition}
A {\it $k$-belt} is a cyclic sequence of $k$ faces with the property that faces are adjacent if and only if they follow each other, and no three faces have a common vertex.  A $k$-belt is {\it trivial}, if it surrounds a face. 
\end{definition}
A simple $3$-polytope different from the {\it simplex} $\Delta^3$ is $ck$-connected, if it has no  $l$-belts for $l<k$, and is  {\it strongly $ck$-connected ($c^*k$-connected)}, if in addition any its $k$-belt is trivial. By definition $\Delta^3$ is $c^*3$-connected but not $c4$-connected. 

All simple $3$-polytopes (family $\mathcal{P}_s$) are $c3$-connected.  We obtain a chain of nested families:
$$
\mathcal{P}_s\supset \mathcal{P}_{aflag}\supset\mathcal{P}_{flag}\supset\mathcal{P}_{aPog}\supset\mathcal{P}_{Pog}\supset\mathcal{P}_{Pog^*}
$$

The family of  $c4$-connected polytopes coincides with the family $\mathcal{P}_{flag}$ of {\it flag} $3$-polytopes defined by the property that any set of pairwise adjacent faces has a nonempty intersection. Each face of a flag polytope is surrounded by a belt. The Euler formula implies that any simple polytope has a $3$-, $4$- or $5$-gonal face. Hence it can not be more than  $c^*5$-connected. 

The family of $c^*3$-connected polytopes we call {\it almost flag} polytopes and denote $\mathcal{P}_{aflag}$. 

Results by  A.V.~Pogorelov  \cite{P67} and E.M.~Andreev \cite{A70a} imply that $c5$-connected polytopes  (family $\mathcal{P}_{Pog}$) are exactly polytopes realizable in the Lobachevsky space $\mathbb L^3$ as bounded polytopes with right dihedral angles. Moreover, the realization is unique up to isometries. These polytopes were called {\it Pogorelov polytopes}, since the paper \cite{P67} is devoted to exactly this class of polytopes. The paper  \cite{A70a} implies that flag polytopes are exactly polytopes realizable in $\mathbb L^3$ as polytopes with equal nonobtuse dihedral angles. Examples of Pogorelov polytopes are given by {\it $k$-barrels} $B_k$, $k\geqslant 5$, ({\it L\"obell polytopes} in terminology of A.Yu.\,Vesnin~\cite{V17}, or {\it truncated trapezohedra}), see Fig. \ref{7DF}a). Results by T.~D\v{o}sli\'c \cite{D98,D03} imply that the family $\mathcal{P}_{Pog}$ contains  {\it fullerenes}, that is simple $3$-polytopes with only pentagonal and hexagonal faces. Mathematical fullerenes model spherical carbon atoms. In 1996   R.~Curl, H.~Kroto, and R.~Smalley obtained the Nobel Prize in chemistry ``for their discovery of fullerenes''. They  synthesized  {\it Buckminsterfullerene} $C_{60}$ (see Fig. \ref{7DF}b), which has the form of the truncated icosahedron (and also of a soccer ball). W.P.~Thurston \cite{T98} built a parametrisation of the fullerene family, which implies that the number of fullerenes with  $n$ carbon atoms grows like~$n^9$ when $n$ tends to infinity. 

\begin{figure}
  \centering
  \includegraphics[height=4cm]{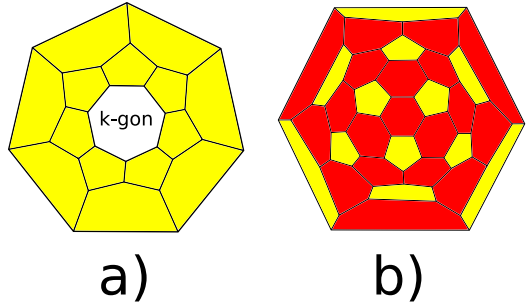}
  \caption{a) $k$-barrel $B_k$; b) fullerene $C_{60}$}\label{7DF}
\end{figure}

The family $\mathcal{P}_{aPog}$ of $c^*4$-connected polytopes we call {\it almost Pogorelov} polytopes, and the family $\mathcal{P}_{Pog^*}$ of $c^*5$-connected polytopes  -- {\it strongly Pogorelov} polytopes.
In \cite[Section 6]{FMW20} the simplicial $3$-polytopes dual to almost Pogorelov polytopes form the family $\mathcal{Q}$.
G.D.~Birkhoff \cite{B1913} reduced the $4$-colour problem to the task to colour in $4$ colours faces of polytopes of the family, which turned out to be the family $\mathcal{P}_{Pog^*}$.

T.E.~Panov remarked that results by E.~M.~Andreev \cite{A70a, A70b} should imply that almost Pogorelov polytopes correspond to right-angled polytopes of finite volume in  $\mathbb L^3$. Such polytopes may have $4$-valent vertices on the absolute, while all proper vertices have valency $3$. It was proved in \cite{E19} that  cutting of $4$-valent vertices defines a bijection between classes of congruence of right-angled polytopes of finite volume in $\mathbb {L}^3$ and almost Pogorelov polytopes different from the cube $I^3$ and the pentagonal prism $M_5\times I$.

In the paper \cite{V87} A.Yu.~Vesnin introduced a construction of a  $3$-dimensional compact hyperbolic manifold 
$R(P,\Lambda_2)$ corresponding to a mapping $\Lambda_2$ of the set of faces of a polytope $P \in\mathcal{P}_{Pog}$ to 
$\mathbb Z_2^3=(\mathbb Z/2\mathbb Z)^3$ such that for any vertex the images of the faces containing it form a basis.
The manifold is obtained by gluing $8$ copies of the polytope along faces. The existence of a mapping follows from the  $4$ colour theorem. In toric topology  (see \cite{BP15, DJ91}) such manifolds are obtained as  {\it small covers} over right-angled polytopes. For a simple $n$-polytope $P$ and a mapping  $\Lambda\colon\{F_1,\dots, F_m\}\to\mathbb Z^n$ such that for any vertex the images of facets containing it form a basis in $\mathbb Z^n$ toric topology associates a $2n$-dimensional smooth {\it quasitoric manifold} $M(P,\Lambda)$ with an action of  $T^n$ such that  $M(P,\Lambda)/T^n=P$.   

In \cite{FMW15} F.~Fan, J.~Ma and X.~Wang proved that any Pogorelov polytope is $B$-rigid.
In \cite{B17} F.~Bosio presented a construction of flag $3$-polytopes, which are not $B$-rigid.

In \cite{BEMPP17} the following result was proved
\begin{theorem}[\cite{BEMPP17}]\label{C6th}For polytopes  $P, P'\in\mathcal{P}_{Pog}$ with functions $\Lambda_2$, $\Lambda$, and $\Lambda_2'$, $\Lambda'$ the manifolds $M(P,\Lambda)$ and $M(P',\Lambda')$ are diffeomorphic if and only if their cohomology rings over $\mathbb Z$ are isomorphic as graded rings,  and if and only if there is a combinatorial equivalence  $\varphi$ between   $P$ and $P'$ and a change of coordinates $A\in Gl_3(\mathbb Z)$ such that for any face $F$ of $P$ we have $\Lambda'(\varphi(F))=\pm A\Lambda(F)$. The manifolds $R(P,\Lambda_2)$ and $R(P',\Lambda'_2)$ are diffeomorphic if and only if their cohomology rings over $\mathbb Z_2$ are isomorphic as graded rings,  and if and only if there is a combinatorial equivalence  $\varphi$ between   $P$ and $P'$ and a change of coordinates $A_2\in Gl_3(\mathbb Z_2)$ such that for any face $F$ of $P$ we have $\Lambda_2'(\varphi(F))=\pm A_2\Lambda_2(F)$.
\end{theorem}
Recent results of this type in the context of toric varieties for $n$-dimensional cubes see in \cite{CLMP20}.

In this paper we develop methods from \cite{FMW15} (see also \cite{BE17S}) to almost Pogorelov polytopes. 

\begin{figure}
\begin{center}
\includegraphics[height=4cm]{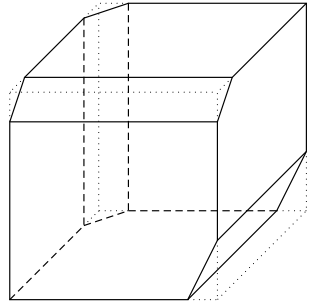}
\caption{A $3$-dimensional associahedron  (Stasheff polytope) $As^3$ as a truncated cube}\label{AsFig}
\end{center}
\end{figure}

The cube and the pentagonal prism belong to $\mathcal{P}_{aPog}$.  All the other polytopes from $P_{aPog}$ have no adjacent quadrangles. The simplest polytope of this type is the $3$-dimensional {\it Stasheff polytope} (or {\it associahedron})  $As^3$ (see Fig. \ref{AsFig} and Fig. \ref{As-sh}), which is the cube with three pairwise disjoint orthogonal edges cut. Result by D.~Barnette \cite{B74} imply that a simple polytope belongs to $\mathcal{P}_{aPog}\setminus\{I^3,M_5\times I\}$ if and only if it can be obtained from  $As^3$ by a sequence of operations of cutting off an edge not lying in quadrangles, and cutting off a pair of adjacent edges of at least hexagonal face by one plane. On Fig. \ref{f*poly-sch} we draw all possible ways to construct almost Pogorelov polytopes with at most $m=11$ faces. It should be mentioned that it follows from \cite{K69} (see also \cite{V15,BE15}) that a simple $3$-polytope is flag if and only if it can be obtained from the cube $I^3$ by a sequence of operations of cutting off an edge and cutting off a pair of adjacent edges of at least hexagonal face.

\begin{figure}
\begin{center}
\includegraphics[height=8cm]{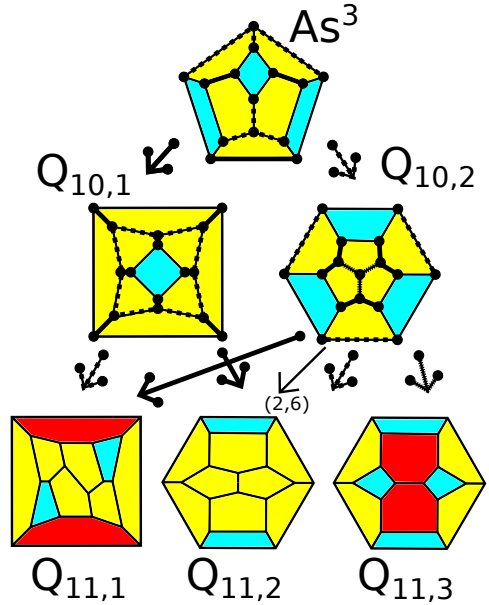}\\
\caption{Almost Pogorelov polytopes with at most $11$ faces}\label{f*poly-sch}
\end{center}
\end{figure}

A polytope in $\mathbb{L}^3$ is called {\it ideal}, if all its vertices lie on the absolute. An ideal polytope has a finite volume. 
Its congruence class is uniquely defined by the combinatorial type (see \cite{R96}).
We will call almost Pogorelov polytopes corresponding to ideal right-angled polytopes {\it ideal almost Pogorelov polytopes} 
and denote their family $\mathcal{P}_{IPog}$. 
It is known (see more details in \cite{E19}) that graphs of ideal right-angled $3$-polytopes  are exactly {\it medial graphs} 
of $3$-dimensional (not necessarily simple) polytopes.  A medial graph of the $k$-pyramid is known as $k$-antiprism, see Fig. \ref{k-antiprism} b). 
The boundaries of polytopes dual to the polytopes in $\mathcal{P}_{IPog}$ 
appeared in \cite[Example 6.3(2)]{FMW20} as examples of flag triangulations such that any $4$-circuit 
full subcomplex is simple.

\begin{figure}
  \centering
  \includegraphics[height=4cm]{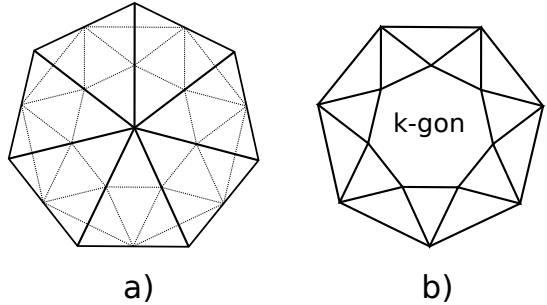}
  \caption{a) Medial graph of the $k$-gonal pyramid; b) The $k$-antiprism}
  \label{k-antiprism}
\end{figure}

An operation of an {\it edge-twist} is drawn on Fig. \ref{E-twist}. Two edges on the left lie in the same face and are disjoint. Let us call an edge-twist {\it restricted}, if both edges are adjacent to an edge of the same face. It follows from \cite{V17,BGGMTW05, E19} that  a polytope is realizable as an ideal right-angled polytope if and only if either it is a $k$-antiprism, $k\geqslant 3$, or it can be obtained from the $4$-antiprism by a sequence of restricted edge-twists.  

\begin{figure}
  \centering
  \includegraphics[width=200pt]{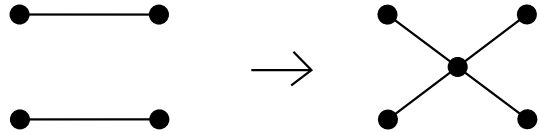}
  \caption{An edge-twist}
  \label{E-twist}
\end{figure}

Thus, the simplest ideal right-angled polytope is a $3$-antiprism, which coincides with the octahedron.
The corresponding ideal almost Pogorelov polytope is known as $3$-dimensional permutohedron $Pe^3$, see Fig. \ref{Pe3}. This is a unique ideal almost Pogorelov polytope with minimal number $m$ of faces ($14$ faces). It is easy to see that 
for ideal almost Pogorelov polytopes $m=2(p_4+1)$, where $p_4$ is the number of quadrangles. It follows from the above construction that for $m=16$ there are no polytopes in $\mathcal{P}_{IPog}$, and for $m=18$ and $20$ there are unique polytopes. First of the them corresponds to the $4$-antiprism, and the second corresponds to a polytope obtained from it by an edge-twist. Recent results on volumes of ideal right-angled polytopes see in \cite{VE20}.

\begin{figure}
\begin{center}
\includegraphics[height=3cm]{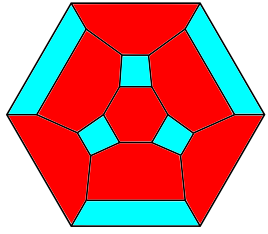}
\caption{Three-dimensional permutohedron $Pe^3$}\label{Pe3}
\end{center}
\end{figure}

\begin{definition}
A property of an $n$-polytope $P$ is called {\it $B$-rigid}, if any isomorphism of graded rings $H^*(\mathcal{Z}_P,\mathbb Z)= H^*(\mathcal{Z}_Q,\mathbb Z)$ for a simple $n$-polytope $Q$ implies that it also has this property. 
\end{definition}
It follows from \cite{FMW15} that the properties to be flag, Pogorelov, and strongly Pogorelov $3$-polytope are $B$-rigid.

Main result of our paper is the fact that the properties to be an almost Pogorelov polytope and an ideal 
almost Pogorelov polytopes are $B$-rigid (Theorem \ref{BaPog} and Corollary \ref{BiaPog}). This implies that 
polytopes of these families with $m\leqslant 11$ and $m\leqslant 20$ respectively, in particular $As^3$ and $Pe^3$, 
are $B$-rigid (Corollaries \ref{11cor} and \ref{18cor}). 

Below in the plan of the paper.

In Section \ref{Cohsec} we give a brief description of the cohomology ring of the moment-angle
manifold of a simple $3$-polytope.

In Section \ref{Bsec} we give a scheme of the proof of $B$-rigidity of any Pogorelov polytope from \cite{FMW15}.
Also in Proposition \ref{newflag} we simplify a criterion for a simple $3$-polytope $P\ne \Delta^3$ to be flag, namely 
$$
{\rm rk }\,H^4(\mathcal{Z}_P)=\frac{(m-2)(m-4)(m-6)}{3}.
$$

In Section \ref{BaPsec} we prove that the property to be an almost Pogorelov polytope and an ideal 
almost Pogorelov polytopes are $B$-rigid (Theorem \ref{BaPog} and Corollary \ref{BiaPog}). This implies that 
polytopes of these families with $m\leqslant 11$ and $m\leqslant 20$ respectively, in particular $As^3$ and $Pe^3$, 
are $B$-rigid (Corollaries \ref{11cor} and \ref{18cor}). 

In Section \ref{HA3sec} we give another characterization of  
almost Pogorelov and ideal almost Pogorelov polytopes in terms of the quotient ring of $H^*(\mathcal{Z}_P)$.

In Section \ref{SCCsec} we prove an analog of the so-called {\it separable circuit condition} (SCC), 
which plays a crucial role in toric topology 
of Pogorelov polytopes. The same result in the dual setting for
simplicial polytopes was earlier proved in \cite[Proposition F.1]{FMW20}. 

Denote by $|\mathcal{B}|$ the union of faces of the belt $\mathcal{B}$.
\begin{lemma}[SCC, \cite{FMW15}] \label{SCClem}
For any Pogorelov polytope $P$ and any three pairwise different faces $\{F_i,F_j,F_k\}$ 
with $F_i\cap F_j=\varnothing$  there exist $l\geqslant 5$ and an $l$-belt $\mathcal{B}_l$ such that 
$F_i,F_j\in \mathcal{B}_l$, $F_k\notin \mathcal{B}_l$, 
and $F_k$ does not intersect at least one of the two connected components of $|\mathcal{B}_l|\setminus(F_i\cup F_j)$. 
\end{lemma}
We prove (Lemma \ref{APb-lemma}) that for an almost Pogorelov polytope $P$ for three pairwise different 
faces $\{F_i,F_j,F_k\}$ such that $F_i\cap F_j=\varnothing$ there exists an $l$-belt ($l\geqslant 4$) $\mathcal{B}_l$ 
such that $F_i,F_j\in \mathcal{B}_l$, $F_k\notin \mathcal{B}_l$, and $F_k$ does not intersect at least one of the two 
connected components of $|\mathcal{B}_l|\setminus(F_i\cup F_j)$ if and only if $F_k$ does not intersect quadrangles 
among the faces $F_i$ and $F_j$. 
Moreover, we prove (Proposition \ref{apbprop}) that this condition is characteristic for a flag $3$-polytope to be an 
almost Pogorelov polytope or the polytope $P_8$ obtained from the cube by cutting off two nonadjacent orthogonal edges.

In Section \ref{Annsec} we prove (Lemma \ref{h3lemma}) an analog of the other crucial tool for Pogorelov polytopes  
-- an {\it annihilator lemma} (which is not valid for almost Pogorelov polytopes, see Proposition \ref{anneq}), 
and analogs of its corollaries (Corollaries \ref{n4cor} and \ref{Bcor}, see also Proposition \ref{BwProp}), 
and study the arising notion of good and bad pairs of disjoint faces. 

In Section \ref{Brssec} 
we prove $B$-rigidity in the class $\mathcal{P}_{aPog}\setminus\{I^3,M_5\times I\}$ (see Definition \ref{Bdef}) 
of 
\begin{itemize}
\item the sets of elements corresponding to $4$-belts (Lemma \ref{4bbr});
\item the cosets corresponding to pairs  of disjoint faces lying in $4$-belts (Lemma \ref{wb4lemma});
\item the sets of elements corresponding to pairs of adjacent quadrangles (Proposition \ref{BwProp}),
\item to $(2k)$-belts containing $k$ quadrangles (Lemma \ref{2k4klemma}), 
\item and to trivial belts of this type (Lemma \ref{2k4ktrivlemma}).
\end{itemize}

In Section \ref{Exas} we consider the ring $H^*(\mathcal{Z}_P,\mathbb Z)$ for $P=As^3$. In particular, we prove (Proposition \ref{anneq}) that  the annihilator lemma is not valid for almost Pogorelov polytopes.

\begin{remark}
In paper \cite{FMW20} the authors generalize Theorem \ref{C6th} (\cite[Theorem 5.1]{FMW20}) and 
result on $B$-rigidity (\cite[Theorems 3.12 and 3.14]{FMW20})  for flag polytopes of dimension of higher than $3$ 
satisfying a generalization of the SCC condition (see Lemma \ref{SCClem}). 
In particular, products of Pogorelov polytopes satisfy these conditions. Namely, they prove that if $P$ is a flag $n$-polytope 
satisfying generalized SCC condition,  $Q$ is a flag $n$-polytope, and there is an isomorphism of graded rings
$H^*(\mathcal{Z}_P)\simeq H^*(\mathcal{Z}_Q)$, then $P$ and $Q$ are combinatorially equivalent.  Note that
at this moment for $n>3$ it is not known if the property to be a flag $n$-polytope is determined by a graded ring 
$H^*(\mathcal{Z}_P)$. Also it is announced without any details, that in the next paper \cite{FMW20b} the authors will prove 
$B$-rigidity of any almost Pogorelov polytope and an analog of Theorem \ref{C6th} for them. 
But their Definition 2.16 of $B$-rigidity assumes isomorphism of bigraded rings, which is a more restrictive condition.
\end{remark}
\begin{remark}
Results of our paper imply (see Remark \ref{BIR-rem}) that any ideal almost Pogorelov is $B$-rigid. This 
gives cohomologically rigid families of $3$-dimensional and $6$-dimensional manifolds corresponding
to a unique (up to a permutation of colours) colouring of faces of an ideal almost Pogorelov polytope into 
$3$ colours. They are "pullbacks from the linear model" from \cite[Example 1.15(1)]{DJ91}.
Details see in \cite{E20b}. 
\end{remark}
\section{Cohomology ring of a moment-angle manifold of a simple $3$-polytope}\label{Cohsec}
Details on cohomology of a moment-angle manifolds see in \cite{BP15, BE17S}.
If not specified, we study cohomology over $\mathbb Z$ and omit the coefficient ring.

The ring $H^*(\mathcal{Z}_P)$  has a multigraded structure:
$$
H^*(\mathcal{Z}_P)=\bigoplus\limits_{i\geqslant0,\omega\subset[m]}H^{-i,2\omega}(\mathcal{Z}_P),
\text{ where }H^{-i,2\omega} (\mathcal{Z}_P)\subset H^{2|\omega|-i} (\mathcal{Z}_P),
$$
and $[m]=\{1,\dots,m\}$. There is a canonical isomorphism 
$H^{-i,2\omega}(\mathcal{Z}_P)\simeq \widetilde{H}^{|\omega|-i-1}(P_{\omega})$,
where $P_{\omega}=\bigcup_{i\in \omega}F_i$, and $\widetilde{H}^{-1}(\varnothing)=\mathbb Z$.

The multiplication between components is nonzero, only if $\omega_1\cap \omega_2=\varnothing$. The mapping
$$
H^{-i,2\omega_1}(\mathcal{Z}_P)\otimes H^{-j,2\omega_2}(\mathcal{Z}_P)\to 
H^{-(i+j),2(\omega_1\sqcup\omega_2)}(\mathcal{Z}_P)
$$
via the Poincare-Lefschets duality $H^i(P_{\omega})\simeq H_{n-1-i}(P_{\omega},\partial P_{\omega})$ 
up to signs is induced by intersection of faces of $P$. 
Denote by $\widehat{H}_{n-i-1}(P_{\omega},\partial P_{\omega})$ the subgroup corresponding to 
$\widetilde{H}^i(P_{\omega})$.

For any simple $3$-polytope $P$ and $\omega\ne\varnothing$ the set 
$P_{\omega}$ is a $2$-dimensional manifold, perhaps with a boundary. Each connected component of $P_{\omega}$ 
is a sphere with holes, and its boundary is a disjoint union of simple edge-cycles. Then $\widetilde{H}^k(P_{\omega})$
is nonzero only for $(k,\omega)\in \{(-1,\varnothing),(0,*),(1,*),(2,[m])\}$. In particular, $P$ has no torsion. There is a multigraded Poincare duality, which means that the bilinear form  
$$
H^{-i,2\omega}(\mathcal{Z}_P)\times H^{-(m-3-i),2([m]\setminus\omega)}(\mathcal{Z}_P)\to H^{-(m-3),2[m]}(\mathcal{Z}_P)=\mathbb Z
$$
has in some basis matrix with determinant $\pm1$. In particular, for a basis $\{e_{\alpha}\}$ in 
$H^{-i,2\omega}(\mathcal{Z}_P)$ there is a dual basis $\{e_{\beta}^*\}$ in 
$H^{-(m-3-i),2([m]\setminus\omega)}(\mathcal{Z}_P)$ such that 
$e_{\alpha}\cdot e_{\beta}^*=\delta_{\alpha,\beta}[\mathcal{Z}_P]$, 
where $[\mathcal{Z}_P]$ is a fundamental class in cohomology. We have 
\begin{gather*}
H^0(\mathcal{Z}_P)=\widetilde{H}^{-1}(\varnothing)=\mathbb Z=\widetilde{H}^2(\partial P)=H^{m+3}(\mathcal{Z}_P);\\
H^1(\mathcal{Z}_P)=H^2(\mathcal{Z}_P)=0=H^{m+1}(\mathcal{Z}_P)=H^{m+2}(\mathcal{Z}_P);\\
H^k(\mathcal{Z}_P)=\bigoplus\limits_{|\omega|=k-1}\widetilde{H}^0(P_{\omega})\oplus\bigoplus\limits_{|\omega|=k-2}\widetilde{H}^1(P_{\omega}),\quad 3\leqslant k\leqslant m.
\end{gather*}

Nontrivial multiplication occurs only for
\begin{enumerate}
\item $\widetilde{H}^{-1}(P_{\varnothing})\otimes \widetilde{H}^{k}(P_{\omega})\to \widetilde{H}^{k}(P_{\omega})$. This 
corresponds to multiplication by $1$ in $H^*(\mathcal{Z}_P)$;
\item  $\widetilde{H}^{0}(P_{\omega})\otimes \widetilde{H}^{1}(P_{[m]\setminus\omega})\to \widetilde{H}^{2}(\partial P)$.
This corresponds to the Poincare duality pairing.
\item $\widetilde{H}^{0}(P_{\omega_1})\otimes \widetilde{H}^{0}(P_{\omega_2})\to \widetilde{H}^{1}(P_{\omega_1\sqcup\omega_2})$. This corresponds to the mapping 
$$
\widehat{H}_2(P_{\omega_1},\partial P_{\omega_1})\otimes \widehat{H}_2(P_{\omega_2},\partial P_{\omega_2})\to H_1(P_{\omega_1\sqcup\omega_2},\partial P_{\omega_1\sqcup\omega_2}).
$$
\end{enumerate}
Each group on the left has a basis with elements corresponding to consistently 
oriented connected components of $P_{\omega_i}$  with a unique relation that the sum of these elements is zero. 
For connected components in $P_{\omega_1}$ and $P_{\omega_2}$ 
the product up to a sign is the intersection of them, 
which is a sum of classes $[\gamma_1]$, $\dots$, $[\gamma_r]$ of disjoint oriented edge paths on 
$P_{\omega_1\sqcup \omega_2}$ connecting components of 
$\partial P_{\omega_1\sqcup\omega_2}$ (if the intersection is not a boundary cycle of each component), see Fig. \ref{PO12}. 
Moreover, for the product to be nonzero at least one component of $\partial P_{\omega_1\sqcup\omega_2}$ 
should have one common point with some path $\gamma_i$.
For each connected component on the right a basis in the first homology group corresponds to any set of oriented 
edge paths connecting one boundary cycle to all the other. 

\begin{figure}
\begin{center}
\includegraphics[height=6cm]{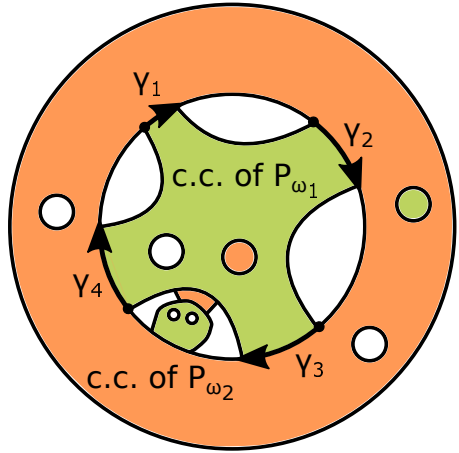}
\caption{Intersection of two connected components of $P_{\omega_1}$ and $P_{\omega_2}$}\label{PO12}
\end{center}
\end{figure}

\section{$B$-rigidity of any Pogorelov polytope}\label{Bsec}
\begin{definition}\label{Bdef}
Let $\mathfrak{P}$ be some set of $3$-polytopes.

We call a number $n(P)$, a set $\mathfrak{S}_P\subset H^*(\mathcal{Z}_P)$ or a collection of such sets defined for any polytope $P\in\mathfrak{P}$ {\em $B$-rigid in the class} $\mathfrak{P}$ if for any isomorphism $\varphi$ of graded rings $H^*(\mathcal{Z}_P)\simeq H^*(\mathcal{Z}_Q)$, $P,Q\in \mathfrak{P}$, we have $n(P)=n(Q)$,  $\varphi(\mathfrak{S}_P)=\mathfrak{S}_Q$, or each set from the collection for $P$ is mapped bijectively to some set from the collection for $Q$ respectively.

To be short, $B$-rigidity in the class of all simple $3$-polytopes we call {\it $B$-rigidity}.
\end{definition}

In \cite{FW15} it was proved that the property to be a flag polytope is $B$-rigid. Namely, 
it is equivalent to the fact that the ring $\widetilde{H}^*(\mathcal{Z}_P)/([\mathcal{Z}_P])$ is a (nonzero) indecomposable ring. Also it was proved that for flag $3$-polytopes 
$$
\widetilde{H}^1(P_{\omega})=\bigoplus_{\omega_1\sqcup\omega_2}\widetilde{H}^0(P_{\omega_1})\cdot\widetilde{H}^0(P_{\omega_2}).
$$
This result was based on the fact that for any flag $3$-polytope and {\it any three pairwise different faces $\{F_i,F_j,F_k\}$ with $F_i\cap F_j=\varnothing$ there exist $l\geqslant 4$ and an $l$-belt $\mathcal{B}_l$ such that $F_i,F_j\in \mathcal{B}_l$ and $F_k\notin \mathcal{B}_l$}. This fact is characteristic for a polytope $P\ne\Delta^3$ to be flag. 
Indeed, if $P$ has a $3$-belt, then for the faces $F_i$ and $F_j$ lying in different connected components of the complement to this belt there are no belts containing $F_i$ and $F_j$.   

On the base of these facts in \cite{BE17S} and \cite{BEMPP17} a simple characterization of flag polytopes was given: a simple 
$3$-polytope $P\ne\Delta^3$ is flag if and only if 
$$
H^{m-2}(\mathcal{Z}_P)\subset(\widetilde{H}^*(\mathcal{Z}_P))^2.
$$

Moreover, it turned out that there is another criterion when a $3$-simple polytope is flag.

\begin{proposition}\label{newflag}
A simple $3$-polytope $P\ne \Delta^3$ with $m$ faces is flag if and only if 
$$
{\rm rk }\,H^4(\mathcal{Z}_P)=\frac{(m-2)(m-4)(m-6)}{3}.
$$
\end{proposition}
\begin{proof}
Indeed, \cite[Theorem 4.6.2]{BP15} states that for a simple $n$-polytope
$$
(1-t^2)^{m-n}(h_0+h_1t^2+\dots+h_n t^{2n})=\sum\limits_{-i,2j}(-1)^i\beta^{-i,2j}t^{2j},
$$
where $h_0+h_1t+\dots+h_nt^n=(t-1)^n+f_{n-1}(t-1)^{n-1}+\dots+f_0$.

In particular, for $3$-polytopes we have
\begin{multline*}
(1-t^2)^h(1+ht^2+ht^4+t^6)=\\
1-\beta^{-1,4}t^4+\sum\limits_{j=3}^{h}(-1)^{j-1}(\beta^{-(j-1),2j}-\beta^{-(j-2),2j})t^{2j}+\\
(-1)^{h-1}\beta^{-(h-1),2(h+1)}t^{2(h+1)}+(-1)^ht^{2(h+3)},
\end{multline*}
where $h=m-3$.

Then 
\begin{multline*}
(1-ht^2+\frac{h(h-1)}{2}t^4-\frac{h(h-1)(h-2)}{6}t^6+\dots)(1+ht^2+ht^4+t^6)=\\
1-\beta^{-1,4}t^4+(\beta^{-2,6}-\beta^{-1,6})t^6+\dots.
\end{multline*}

For any simple $3$-polytope $P$ we have:
\begin{itemize}
\item ${\rm rk }\,H^3(\mathcal{Z}_P)=\beta^{-1,4}=|N_2(P)|=\frac{h(h-1)}{2}$;
\item $\beta^{-1,6}$ = \#\{$3$-belts\};
\item $\beta^{-2,6}-\beta^{-1,6}=\frac{(h^2-1)(h-3)}{3}$.
\end{itemize}
Then 
$$
{\rm rk }\,H^4(\mathcal{Z}_P)=\beta^4=\beta^{-2,6}=\beta^{-1,6}+\frac{(h^2-1)(h-3)}{3}=\text{\#\{$3$-belts\}}+\frac{(h^2-1)(h-3)}{3}
$$
\end{proof}

It is easy to see that a polytope has no $4$-belts if and only if the multiplication
$$
H^3(\mathcal{Z}_P)\otimes H^3(\mathcal{Z}_P)\to H^6(\mathcal{Z}_P)
$$
is trivial.
Thus, the property to be a Pogorelov polytope is also $B$-rigid.

In \cite{FMW15} (details see also in \cite{BE17S}) it was proved that any Pogorelov polytope is 
$B$-rigid.

The proof consisted of several steps. Denote 
$$
N_2(P)=\{\{i,j\}\subset [m]\colon F_i\cap F_j=\varnothing\}.
$$
For any $\omega\in N_2(P)$ we have $\widetilde H^0(P_{\omega})=\mathbb Z$. Denote by $\widetilde{\omega}$ 
a generator of this group (it is defined up to a sign, we choose one of them arbitrarily). Then for any $3$-polytope $H^3(\mathcal{Z}_P)$ is a free abelian
group with the basis $\{\widetilde\omega\colon\omega\in N_2(P)\}$. 

First is was proved that the set of elements 
$$
\{\pm \widetilde{\omega}\colon \omega\in N_2(P)\}\subset H^3(\mathcal{Z}_P)
$$ 
corresponding to sets $\omega\in N_2(P)$ is $B$-rigid in the class of Pogorelov polytopes. 
The proof was based on the so-called {\it separable circuit condition} (SCC): {\it for any  
Pogorelov polytope $P$ and any three different faces $\{F_i,F_j,F_k\}$ with $F_i\cap F_j=\varnothing$ 
there exist $l\geqslant 5$ and an $l$-belt $\mathcal{B}_l$ such that $F_i,F_j\in \mathcal{B}_l$, $F_k\notin \mathcal{B}_l$, 
and $F_k$ does not intersect at least one of the two connected components of $|\mathcal{B}_l|\setminus(F_i\cup F_j)$}.
In fact,  this condition is characteristic for a polytope $P\ne\Delta^3$ to be Pogorelov. 
Indeed, if $P$ has a $4$-belt, then for the faces $F_i$ and $F_j$ lying in different connected components of the complement to this belt and $F_k$ lying on the belt any belt containing $F_i$ and $F_j$ intersects $F_k$ by both components.   
Later, in Lemma \ref{APb-lemma} we will prove an analog of SCC for almost Pogorelov polytopes. It will be also characteristic to be an almost Pogorelov polytope with one exception, see Proposition \ref{apbprop}.

Now let us give some details of the proof, which we will need in what follows.

\begin{definition}
An {\em annihilator}\index{annihilator} of an element $r$ in a ring $R$ is defined as 
$$
{\rm Ann}_R(r)=\{s\in R\colon rs=0\}
$$
\end{definition}

Since the group $H^*(\mathcal{Z}_P)$ has no torsion, we have the isomorphism $H^*(\mathcal{Z}_P,\mathbb Q)\simeq H^*(\mathcal{Z}_P)\otimes\mathbb Q$ and the embedding $H^*(\mathcal{Z}_P)\subset H^*(\mathcal{Z}_P)\otimes\mathbb Q$. For polytopes $P$ and $Q$ the isomorphism $H^*(\mathcal{Z}_P)\simeq H^*(\mathcal{Z}_Q)$ implies the isomorphism over $\mathbb Q$. For the cohomology over $\mathbb Q$ all  theorems about structure of $H^*(\mathcal{Z}_P,\mathbb Q)$ are still valid.

It was proved that for a Pogorelov polytope $P$ and an element  $\alpha$ in $H=H^*(\mathcal{Z}_P,\mathbb Q)$: 
$$
\alpha=\sum\limits_{\omega\in N_2(P)}r_{\omega}\widetilde\omega\quad\text{with }|\{\omega\colon r_{\omega}\ne 0\}|\geqslant 2
$$ 
we have 
$$
\dim {\rm Ann}_H (\alpha)<\dim {\rm Ann}_H (\widetilde \omega),\text { if }r_{\omega}\ne 0.
$$
We call this result the {\it annihilator lemma}.

For almost Pogorelov polytopes the annihilator lemma 
is not valid, see Proposition \ref{anneq}. In Lemma \ref{h3lemma} we will modify it 
to be valid for all simple $3$-polytopes.  Corollary \ref{n4cor} shows how this lemma can be applied to almost Pogorelov polytopes.
Namely, under any isomorphism of graded rings $H^*(\mathcal{Z}_P)\to H^*(\mathcal{Z}_Q)$ 
the element $\widetilde\omega$ for $\omega\in N_2(P)$ is mapped to $\pm\widetilde{\omega'}$ for 
$\omega'\in N_2(Q)$, if $\omega=\{p,q\}$, and both $F_p$ and $F_q$ are not adjacent to quadrangles. 
In Definition \ref{good-bad} for $\omega\in N_2(P)$ we introduce the notion of {\it good} and {\it bad} elements in $N_2(P)$, and in
Corollary \ref{gbcor} we prove that for any simple 
$3$-polytopes $P$ and $Q$ 
the element $\widetilde\omega, \omega\in N_2(P)$ is mapped to a linear combination of an element
$\widetilde{\omega'}$, $\omega'\in N_2(Q)$ and elements $\widetilde\omega''$, $\omega''\in N_2(Q)$, where $\omega''$ 
are bad for $\omega'$.

On the second step in \cite{FMW15} it was proved that the set
$$
\{\pm \widetilde{\mathcal{B}_k}\colon \mathcal{B}_k - \text{ a $k$-belt}\}\subset H^{k+2}(\mathcal{Z}_P)
$$ 
is $B$-rigid in the class of Pogorelov polytopes. 
Here $\widetilde{\mathcal{B}_k}$ is the element in $H^{k+2}(\mathcal{Z}_P)$ corresponding to a generator in 
$\widetilde{H}^1(P_{\omega(\mathcal{B}_k)})\simeq\mathbb Z$, where 
$\omega(\mathcal{B}_k)=\{i\in[m]\colon F_i\in\mathcal{B}_k\}$. 
In Corollary \ref{4bbr}  we will prove that the set of elements corresponding to $4$-belts is 
$B$-rigid in the class of almost Pogorelov polytopes. Also in Corollary \ref{Bcor} we will prove that if 
$P\in\mathcal{P}_{aPog}\setminus\{I^3,M_5\times I\}$ and all the faces of a $k$-belt $\mathcal{B}_k$ do not have 
common points with quadrangles, then under any isomorphism of graded rings $H^*(\mathcal{Z}_P)\to H^*(\mathcal{Z}_Q)$ 
for a simple $3$-polytope $Q$ the element $\widetilde{\mathcal{B}_k}$ is mapped to $\pm\widetilde{\mathcal{B}_k'}$ for 
some $k$-belt $\mathcal{B}_k'$ of $Q$.

Also from the arguments in \cite{FMW15} it follows (see the detailed proof in \cite{BE17S}) that for any simple $3$-polytope $P$ the 
subgroup $\boldsymbol{B}_k\subset H^{k+2}(\mathcal{Z}_P)$, $4\leqslant k\leqslant m-2$, with the basis $\{\widetilde{\mathcal{B}_k}\colon \mathcal{B}_k - \text{ a $k$-belt}\}$ is $B$-rigid. In particular, the number of $k$-belts, $k\geqslant 4$, is $B$-rigid.

Then it was proved that the set of elements 
$$
\{\pm \widetilde{\mathcal{B}_k}\colon \mathcal{B}_k - \text{ a $k$-belt around a face}\}\subset H^{k+2}(\mathcal{Z}_P)
$$ 
is $B$-rigid in the class of Pogorelov polytopes. 
\begin{corollary}\label{BrPog*}
The property to be a strongly Pogorelov polytope is $B$-rigid. 
\end{corollary}

This induces a bijection between the sets of faces of polytopes.

Then it was proved that images of adjacent faces are adjacent. This finishes the proof of $B$-rigidity of any Pogorelov polytope.

\section{$B$-rigidity of the property to be an almost Pogorelov polytope}\label{BaPsec}
Our aim is to prove that the property to be an almost Pogorelov polytope is $B$-rigid.

First the cube and the pentagonal prism are the only flag polytopes with $m=6$ and $m=7$ respectively. Hence, they are 
$B$-rigid.

Now consider all the other almost Pogorelov polytopes. They have no adjacent quadrangles, since a pair of adjacent quadrangles
should be surrounded by a nontrivial $4$-belt.

The image of the mapping $H^3(\mathcal{Z}_P)\otimes H^3(\mathcal{Z}_P)\to H^6(\mathcal{Z}_P)$ is 
the subgroup $\boldsymbol{B}_4$ generated  by the elements $\widetilde{\mathcal{B}_4}$ corresponding to $4$-belts.

There is a $B$-rigid subgroup $A_3\subset H^3(\mathcal{Z}_P)$:
$$
A_3=\{x\in H^3(\mathcal{Z}_P)\colon x\cdot y=0\text{ for all } y\in H^3(\mathcal{Z}_P)\}.
$$
It is generated by elements $\widetilde{\omega}$ such that the set $\omega\in N_2(P)$ does not 
belong to $\omega(\mathcal{B}_4)$ for any $4$-belt $\mathcal{B}_4$.  Denote this set $N_2^0(P)$.

The subgroup $A_3$ is a direct summand in $H^3(\mathcal{Z}_P)$, therefore 
the quotient group $\boldsymbol{H}_3=H^3(\mathcal{Z}_P)/A_3$  is also a free abelian group.
We have the mapping
$\boldsymbol{H}_3\otimes \boldsymbol{H}_3\to H^6(\mathcal{Z}_P)$. Also
$\boldsymbol{H}_3\otimes\mathbb {Q}\simeq H^3(\mathcal{Z}_P,\mathbb Q)/(A_3\otimes \mathbb Q)$
and there is an embedding $\boldsymbol{H}_3\subset\boldsymbol{H}_3\otimes\mathbb {Q}$.

\begin{lemma}\label{2unlemma}
We have $2\,\rk \boldsymbol{B}_4\geqslant \rk\boldsymbol{H}_3$, and the equality holds
if and only if each set $\omega\in N_2(P)$
can be included in at most one set $\omega(\mathcal{B}_4)$.
\end{lemma} 
\begin{proof}
Each element $\widetilde{\mathcal{B}_4}$  can be represented uniquely up to a change of order
as a product of two elements $\widetilde{\omega}_1\cdot\widetilde{\omega}_2$, $\omega_i\in N_2(P)$.
On the other hand, an element $\widetilde{\omega}$ can be a factor of several elements, corresponding
to $4$-belts, if the faces $F_p,F_q$ can be included into several $4$-belts, where $\omega=\{p,q\}$.
Let us calculate the number of pairs $(\omega, \mathcal{B}_4)$, where $\omega\in N_2(P)$ and 
$\omega\subset \omega(\mathcal{B}_4)$. On the one hand, it is equal to $2\,\rk \boldsymbol{B}_4$, on the other hand,
it is at least the number of elements in $N_2(P)$ that can be included into a $4$-belt, and the equality holds if and only if 
each this element can be included in exactly one set $\omega(\mathcal{B}_4)$.
\end{proof}

\begin{lemma}\label{44lemma}
Let $P$ be a flag $3$-polytope. If $2\,\rk \boldsymbol{B}_4= \rk\boldsymbol{H}_3$, then $P$ has no adjacent quadrangles.
\end{lemma}
\begin{proof}
If $P$ is the cube $I^3$, then $\rk \boldsymbol{B}_4= \rk\boldsymbol{H}_3=3$. Else if $P$ has two adjacent quadrangles, then 
these quadrangles
are surrounded by a $4$-belt such that the two faces intersecting both these quadrangles are included into two 
$4$-belts surrounding the quadrangles. 
\end{proof}

\begin{figure}
\begin{center}
\includegraphics[height=4cm]{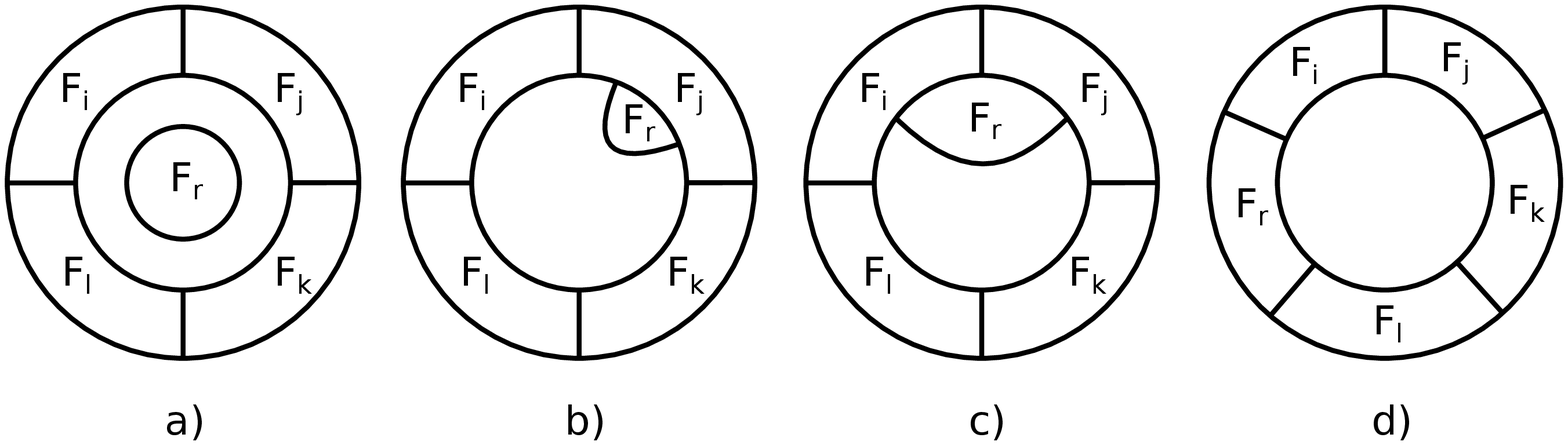}
\caption{a) $4$-belt $(F_i,F_j,F_k,F_l)$ disjoint from $F_r$; b) $4$-belt $(F_i,F_j,F_k,F_l)$, and $F_r$ intersects only $F_j$; c) $4$-belt $(F_i,F_j,F_k,F_l)$ and $F_r$ intersects only $F_i$ and $F_j$; d) $5$-belt $(F_i,F_j,F_k,F_l,F_r)$}\label{SetsH7}
\end{center}
\end{figure}

\begin{lemma}\label{H7lemma}
Let $P$ be a flag $3$-polytope. 

If $\widetilde{H}_1(P_{\omega})\ne 0$, then $|\omega|\geqslant 4$. Moreover, if $|\omega|=4$, then $\omega=\omega(\mathcal{B}_4)$ for some $4$-belt.

If $2\,\rk \boldsymbol{B}_4= \rk\boldsymbol{H}_3$, and $|\omega|=5$, then $\widetilde{H}_1(P_{\omega})\ne 0$ if and only if $P_{\omega}$ has one of the $4$ types drawn on Fig. \ref{SetsH7}.
\end{lemma}
\begin{proof}
As we mentioned above, $P_{\omega}$ is a disjoint union of spheres with holes.  
If $\widetilde{H}_1(P_{\omega})\ne 0$, then one sphere $P_{\omega'}$, 
$\omega'\subset\omega$, has at least two holes. 
This is impossible for $|\omega'|=1\text{ or }2$. Hence, $|\omega'|\geqslant 3$.
Consider the boundary component $\gamma$ of $P_{\omega'}$. 
Walking along this simple edge-cycle we obtain a cyclic sequence of faces 
$\mathcal{L}=(F_{i_1},\dots,F_{i_p})$, where the index $i_l$ lies in $\mathbb Z_p=\mathbb Z/p\mathbb Z$. 
If $p=1$, then $P_{\omega'}$ consists of one face and is contractible. 
If $p=2$, then $P_{\omega'}$ consists of two adjacent faces and also is contractible. 
If $p=3$, then $F_{i_1}$, $F_{i_2}$ and $F_{i_3}$ are different 
pairwise adjacent faces. Since $P$ has no $3$-belts, they have a common vertex  and $P_{\omega'}$ consists of these three 
faces and is contractible. Thus, $p\geqslant 4$. 

We have $F_{i_l}\ne F_{i_{l+1}}$ for any $l$. If $F_{i_l}=F_{i_{l+2}}$ for some $l$, then we have the configuration on Fig. \ref{Cyclered}a), and $P_{\omega'}$ is homeomorphic to $P_{\omega'\setminus\{i_{l+1}\}}$. 
If $F_{i_l}\cap F_{i_{l+1}}$ is an edge, then we have the configuration on Fig. \ref{Cyclered}b), 
and $P_{\omega'}$ is homeomorphic to $P_{\omega'\setminus\{i_{l+1}\}}$. If $F_{i_l}= F_{i_{l+3}}$, then $F_{i_l}$, $F_{i_{l+1}}$, $F_{i_{l+2}}$ are three different pairwise adjacent faces.
Since $P$ has no $3$-belts, the edge $F_{i_{l+1}}\cap F_{i_{l+2}}$ intersects $F_{i_l}$, and we have the configuration on Fig. \ref{Cyclered}c). Then $P_{\omega'}$ is homeomorphic to $P_{\omega'\setminus\{i_{l+1},i_{l+2}\}}$.

\begin{figure}
\begin{center}
\includegraphics[height=4cm]{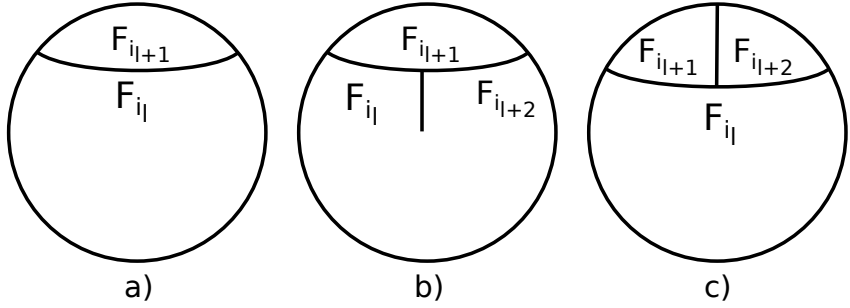}
\caption{a) $F_{i_l}=F_{i_{l+2}}$; b) $F_{i_l}\cap F_{i_{l+2}}$ is an edge, c) $F_{i_l}=F_{i_{l+3}}$}\label{Cyclered}
\end{center}
\end{figure}

If  $|\omega'|=3$, then $F_{i_l}\ne F_{i_{l+1}}$, $F_{i_l}\ne F_{i_{l+2}}$, $F_{i_l}\ne F_{i_{l+3}}$ for all $l$ by the previous argument. A contradiction. Hence, $|\omega|\geqslant |\omega'|\geqslant 4$.

If $|\omega'|=4$, then $F_{i_l}\ne F_{i_{l+1}}$, $F_{i_l}\ne F_{i_{l+2}}$, 
$F_{i_l}\cap F_{i_{l+2}}=\varnothing$, $F_{i_l}\ne F_{i_{l+3}}$ for all $l$ by the previous argument. If $p=4$, then $\mathcal{L}$
is a $4$-belt and $\omega=\omega'=\omega(\mathcal{L})$. Let $p\geqslant 5$. Then $F_{i_1}=F_{i_5}\ne F_{i_6}$, in particular, 
$p\geqslant 6$. But $F_{i_6}$ is
separated from $F_{i_2}$, $F_{i_3}$, $F_{i_4}$ by $F_{i_1}=F_{i_5}$, which is a contradiction. Thus, $p=4$. 

Let $|\omega|=5$ and $2\,\rk \boldsymbol{B}_4= \rk\boldsymbol{H}_3$. 
If $\omega'\ne\omega$, then $\omega'=\omega(\mathcal{B}_4)$ for some $4$-belt, and $P_{\omega}$ has
the structure drawn on Fig. \ref{SetsH7}a). Now let $\omega'=\omega$. If $F_{i_l}=F_{i_{l+2}}$ for some $l$, 
then by the previous argument $P_{\omega}$ has the structure drawn on Fig. \ref{SetsH7}b). If $F_{i_l}\cap F_{i_{l+2}}$ 
is an edge for some $l$, then  $P_{\omega}$ has the structure drawn on Fig. \ref{SetsH7}c). If $F_{i_l}= F_{i_{l+3}}$, then 
$P_{\omega}$ is contractible, which is a contradiction. Now assume that  $F_{i_l}\ne F_{i_{l+1}}$, $F_{i_l}\ne F_{i_{l+2}}$, 
$F_{i_l}\cap F_{i_{l+2}}=\varnothing$, $F_{i_l}\ne F_{i_{l+3}}$ for all $l$. If $F_{i_l}\cap F_{i_{l+3}}$ is an edge for some $l$, then
$\mathcal{B}_4=(F_{i_l},F_{i_{l+1}},F_{i_{l+2}},F_{i_{l+3}})$ is a $4$-belt.
Let $i=\omega\setminus\{i_l,i_{l+1},i_{l+2},i_{l+3}\}$. 
The face $F_i$ lies in one of the connected components of the complement to this belt in $\partial P$. 
If $F_i$ intersects two opposite faces of the belt, say $F_{i_l}$ and $F_{i_{l+2}}$, then it should also intersect $F_{i_{l+1}}$ 
and $F_{i_{l+3}}$ for otherwise the pair $\{F_{i_l},F_{i_{l+2}}\}$ can be included into two $4$-belts, which contradicts the condition $2\,\rk \boldsymbol{B}_4= \rk\boldsymbol{H}_3$. Since $P$ is flag, 
this means, that $\mathcal{B}_4$ is a belt around $F_i$ and $P_{\omega}$ is contractible. Thus, $F_i$ intersects at most two faces and they should be successive in the belt. We obtain the cases b) and c) of Fig. \ref{SetsH7} again.
Now consider the case when $F_{i_l}\cap F_{i_{l+3}}=\varnothing$ for all $l$. Then $p\geqslant 5$,  $F_{i_l}\ne F_{i_{l+4}}$ and
$\omega=\{i_l,i_{l+1},i_{l+2},i_{l+3},i_{l+4}\}$ for any $l$. If $p=5$, then $\mathcal{L}$ is a $5$-belt (Fig. \ref{SetsH7}d). Else $F_{i_l}=F_{i_{l+5}}$, in particular, $p\geqslant 7$. Then $F_{i_{l+6}}$ is separated from $F_{i_{l+1}}$, $F_{i_{l+2}}$, $F_{i_{l+3}}$, $F_{i_{l+4}}$ by $F_{i_l}=F_{i_{l+5}}$, which is a contradiction. Thus, $p=5$. Now we have considered all the possible cases and the proof is finished.
\end{proof}

\begin{corollary}
Let $P$ be a flag $3$-polytope with $2\,\rk \boldsymbol{B}_4= \rk\boldsymbol{H}_3$. Then the image $I_7$ of the mapping
$$
H^3(\mathcal{Z}_P)\otimes H^4(\mathcal{Z}_P)\to H^7(\mathcal{Z}_P)
$$
is a free abelian group with a basis corresponding to elements $\widetilde{\mathcal{B}_5}$, where $\mathcal{B}_5$ is a $5$-belt, 
and generators of the groups $\widetilde{H}^1(P_{\omega})\simeq\mathbb Z$ for $P_{\omega}$ drawn 
on Fig. \ref{SetsH7} a)-c). 
\end{corollary}
\begin{corollary}\label{rki7cor}
Let $P$ be a flag $3$-polytope with $2\,\rk \boldsymbol{B}_4= \rk\boldsymbol{H}_3$. 
$$
\rk I_7=\#\{\text{$5$-belts}\}+(m-5)\#\{\text{trivial $4$-belts}\}+(m-4)\#\{\text{nontrivial $4$-belts}\}.
$$
\end{corollary}
\begin{proof}
For any $4$-belt $\mathcal{B}=(F_i,F_j,F_k,F_l)$ and any face $F_r\notin \mathcal{B}$ if  
$F_r$ intersects opposite faces of the belt, say $F_i$ and $F_k$, then $F_r\cap F_j\ne\varnothing$, for otherwise
$(F_i,F_j,F_k,F_r)$ is a $4$-belt, and $F_r\cap F_l\ne\varnothing$ for the same reason. Since $P$ is flag, this implies that 
$\mathcal{B}$ surrounds $F_r$. Otherwise, $F_r$ intersects at most two faces, and they should be successive. 
Then $\mathcal{B}\cup F_r$ has the form drawn on Fig. \ref{SetsH7}a)-c). For a trivial belt we have $m-5$ such faces $F_r$, 
and for a nontrivial belt we have $m-4$ such faces. This finishes the proof.
\end{proof}

\begin{theorem}\label{BaPog}
A flag simple $3$-polytope $P$ belongs to $\mathcal{P}_{aPog}\setminus\{I^3,M_5\times I\}$ if and only if 
$$
2\,\rk \boldsymbol{B}_4= \rk\boldsymbol{H}_3, \text{ and }\,\rk I_7=\rk \boldsymbol{B}_5+(m-5)\rk\,\boldsymbol{B}_4,
$$
where $\rk \boldsymbol{B}_5=\#\{\text{$5$-belts}\}$ is a $B$-rigid number.
In particular, the property to be an almost Pogorelov polytope is $B$-rigid. 
\end{theorem}
\begin{proof}
Corollary \ref{rki7cor} implies that if the conditions of the theorem hold, then every $4$-belt of $P$ is trivial, 
hence $P$ is an almost Pogorelov polytope. Lemma \ref{44lemma} implies that $P$ has no adjacent quadrangles. 
Hence, $P$ is different from the cube and the pentagonal prism.

Now let $P\in\mathcal{P}_{aPog}\setminus\{I^3,M_5\times I\}$. It has no adjacent quadrangles,
for otherwise these quadrangles should be surrounded by a nontrivial $4$-belt.
Consider a pair of nonadjacent faces $F_i$ and $F_k$  that can be included into a $4$-belt $(F_i,F_j,F_k,F_l)$. 
The belt surrounds some quadrangle $F_r$. 
Let $F_i$ and $F_k$ be included into the other $4$-belt $(F_i,F_p,F_k,F_q)$. This belt surrounds some other
quadrangle $F_t\ne F_r$. Then $(F_i,F_r,F_k,F_t)$ is a $4$-belt. It surrounds some quadrangle adjacent to both
$F_t$ and $F_r$. A contradiction. Thus, any pair of nonadjacent faces of $P$ can be included in at most
one $4$-belt, and $2\,\rk \boldsymbol{B}_4= \rk\boldsymbol{H}_3$ by Lemma \ref{2unlemma}. Then 
$\rk I_7=\rk \boldsymbol{B}_5+(m-5)\rk\,\boldsymbol{B}_4$ by Corollary \ref{rki7cor}. This finishes the proof.
\end{proof}
\begin{corollary}\label{11cor}
Any almost Pogorelov polytope with at most $11$ faces is $B$-rigid. 
\end{corollary}
\begin{proof}
Indeed, as we mentioned above for $m=6$ the cube, and for $m=7$ the pentagonal prism is a unique flag $3$-polytope.
For $m=8$ there are no almost Pogorelov polytopes. For $m=9$ the associahedron $As^3$ is a unique almost Pogorelov polytope.
For $m=10$ and $11$ different almost Pogorelov polytopes have different number of quadrangles, which is equal to $ \rk\boldsymbol{B}_4$. 
\end{proof}
\begin{corollary}\label{BiaPog}
A flag simple $3$-polytope $P$ is an ideal almost Pogorelov polytope 
if and only if 
$$
2\,\rk \boldsymbol{B}_4= \rk\boldsymbol{H}_3=m-2, \text{ and }\,\rk I_7=\rk \boldsymbol{B}_5+\frac{(m-5)(m-2)}{2}.
$$
In particular, the property to be an ideal almost Pogorelov polytope is $B$-rigid. 
\end{corollary}
\begin{proof}
Indeed, an almost Pogorelov polytope is ideal if and only if any its vertex lies on a unique quadrangle, 
which is equivalent to the fact that the number of quadrangles $p_4=\rk \boldsymbol{B}_4$ multiplied by $4$ 
is equal to the number of vertices $f_0$, which is $2(m-2)$ for any simple $3$-polytope.
\end{proof}
\begin{corollary}\label{18cor}
The $3$-dimensional permutohedron $Pe^3$ is $B$-rigid.
Moreover, any ideal almost Pogorelov polytope with $m\leqslant 20$ is $B$-rigid.
\end{corollary}
\begin{proof}
Indeed, $Pe^3$ is a unique ideal almost Pogorelov polytope with minimal number of faces  $m=14$. Since $m=2+2p_4$, this number is even. For $m=16$ there are no ideal almost Pogorelov polytopes. For $m=18$ and $m=20$ there are unique ideal almost Pogorelov polytopes. 
\end{proof}

\section{The quotient ring of $H^*(\mathcal{Z}_P)$ modulo the ideal generated by $A_3$.}
\label{HA3sec}
Now we will give equivalent characterizations of almost Pogorelov and ideal almost Pogorelov polytopes
in terms of the quotient ring of $H^*(\mathcal{Z}_P)$ modulo the ideal generated by $A_3$. 

\begin{lemma}\label{52lemma}
Let $P$ be a flag $3$-polytope.Then for $k\geqslant 5$ any its $k$-belt $\mathcal{B}$ 
different from the belt surrounding the base of a $k$-gonal prism, 
contains a pair of nonadjacent faces that can not be included into any $4$-belt. 
\end{lemma}
\begin{proof}
Let $\mathcal{B}=(F_{i_1},F_{i_2},\dots,F_{i_k})$ and $C_1$ and $C_2$ be components of its 
complement in $\partial P$. Let any pair of non-successive faces of $\mathcal{B}$ can be included into a $4$-belt.
Then there is a $4$-belt $(F_{i_1},F_{p_1},F_{i_3},F_{q_1})$. At least one of the faces $F_{p_1}$ and $F_{q_1}$, 
say $F_{p_1}$ is different from $F_{i_2}$, hence lies in $\overline{C_1}$ or $\overline{C_2}$, say $\overline{C_1}$. 
There is a $4$-belt $(F_{i_2},F_{p_2},F_{i_4},F_{q_2})$. 
As before, $F_{p_2}$ lies in $\overline{C_1}$ or $\overline{C_2}$. If $F_{p_2}$ lies in $\overline{C_1}$, 
then $F_{p_1}=F_{p_2}$, since no other faces in $\overline{C_1}$ different from $F_{p_1}$ can intersect both 
$F_{i_2}$ and $F_{i_4}$ (this follows from the Jordan curve theorem, since
the set $F_{i_2}\cap \overline{C_1}$ lies inside the  disk bounded by a simple piecewise linear curve lying inside 
$F_{i_1}\cup F_{i_2} \cup F_{i_3}\cup F_{p_1}$, and $F_{i_4}\cap \overline{C_1}$ 
lies outside this curve (see Fig. \ref{54belts}a)). Thus,
$F_{p_1}$ is adjacent to $F_{i_1}$, $F_{i_2}$, $F_{i_3}$, and $F_{i_4}$. 
Since $P$ is flag, we obtain configuration drawn on
Fig. \ref{54belts}b).  Then $F_{q_1},F_{q_2}\subset\overline{C_2}$. Thus, $F_{q_1}=F_{q_2}$ 
by the same agrument, and we obtain  the same picture as for $F_{p_1}$.  
There is a $4$-belt $(F_{i_3},F_{p_3},F_{i_5}, F_{q_3})$. The face 
$F_{i_3}$ intersects only faces $F_{i_2}$, $F_{i_4}$, $F_{p_1}$, and $F_{q_1}$. 
Since $F_{i_2}\cap F_{i_5}=\varnothing$, we have $\{F_{p_3}, F_{q_3}\}=\{F_{p_1},F_{q_1}\}$, 
and we can assume that $F_{p_3}=F_{p_1}$, and $F_{q_3}=F_{q_1}$.
Repeating this argument, we obtain that $F_{p_i}=F_{p_1}$, $F_{q_i}=F_{q_1}$ for all $i$, and $\mathcal{B}$ is the 
belt around the base of a $k$-gonal prism, which is a contradiction.

Thus, $F_{p_2}\subset \overline{C_2}$, $F_{q_2}\subset F_{i_3}\cup \overline{C_2}$, and 
$F_{q_1}\subset F_{i_2}\cup \overline{C_1}$. Moreover, if $F_{q_1}=F_{i_2}$, then $F_{p_1}\cap F_{i_2}=\varnothing$. 
Else $F_{p_1},F_{q_1}\subset \overline{C_1}$, and at least one of these faces do not intersect $F_{i_2}$. Without
loss of generality we can assume that $F_{p_1}\cap F_{i_2}=\varnothing$. Similarly $F_{p_2}\cap F_{i_3}=\varnothing$. 
Repeating the argument, we obtain $F_{p_3}\subset \overline{C_1}$, $F_{q_3}\subset F_{i_4}\cup \overline{C_1}$, 
$F_{p_3}\cap F_{i_4}=\varnothing$, and so on. If $k$ is odd,
then $F_{p_k}\subset \overline{C_1}$, $F_{q_k}\subset F_{i_{k+1}}\cup \overline{C_1}$, and 
$F_{p_1}\subset F_{i_2}\cup \overline{C_2}$, which is a contradiction. If $k$ is even, then $k\geqslant 6$. Consider the $4$-belt 
$(F_{i_1},F_s,F_{i_4},F_t)$. Both faces $F_s$ and $F_t$ do not belong to $\mathcal{B}$, since they intersect both $F_{i_1}$ and $F_{i_4}$. But in $\overline{C_1}$ there are no faces intersecting both $F_{i_1}$ and $F_{i_4}$, because 
$F_{p_3}\cap F_{i_4}=\varnothing$. Similarly in $\overline{C_2}$ 
there are no faces intersecting both $F_{i_1}$ and $F_{i_4}$, because 
$F_{p_k}\cap F_{i_1}=\varnothing$. A contradiction.
This finishes the proof.

\begin{figure}
\begin{center}
\includegraphics[height=4cm]{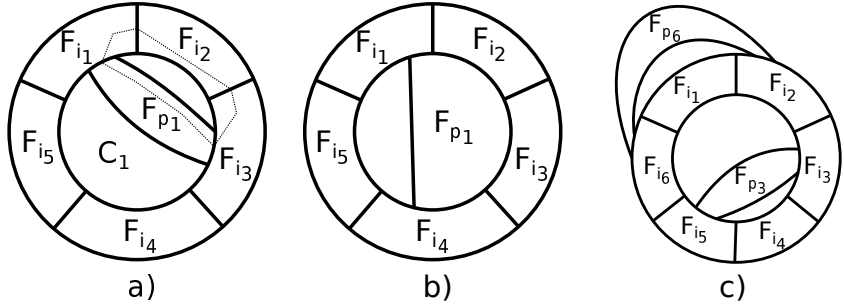}
\caption{Configuration of faces near the belt}\label{54belts}
\end{center}
\end{figure}
\end{proof}

\begin{corollary} In the ring 
$H^*(\mathcal{Z}_P)/\langle A_3\rangle$ all the elements $[\widetilde{\mathcal{B}_k}]$ corresponding
to $k$-belts, $k\geqslant 5$, are zero.
\end{corollary}
\begin{corollary}
Let $P$ be a flag $3$-polytope with $2\,\rk \boldsymbol{B}_4= \rk\boldsymbol{H}_3$. Then the image $A_7$ of the mapping
$$
A_3\otimes H^4(\mathcal{Z}_P)\to H^7(\mathcal{Z}_P)
$$
is a free abelian group $\boldsymbol{B}_5$ with the basis corresponding to elements $\widetilde{\mathcal{B}_5}$,  where $\mathcal{B}_5$ is a $5$-belt. 
\end{corollary}
\begin{proof}
Indeed, Lemma \ref{52lemma} implies, that $\boldsymbol{B}_5\subset A_7$.
Let $x\in A_7$. Then 
$$
x=\left(\sum\limits_{\omega\in N_2^0(P)}\lambda_{\omega}\widetilde{\omega}\right)\cdot \left(\sum\limits_{\tau\subset[m]\colon |\tau|=3}y_{\tau}\right)=\sum\limits_{\eta\subset[m]\colon |\eta|=5}z_{\eta},
$$ 
where each nonzero $y_{\tau}$ corresponds to an element in $\widehat{H}_2(P_{\tau},\partial P_{\tau})\ne 0$, and each nonzero
$z_{\eta}$ corresponds to an element in $\widehat{H}_1(P_{\eta},\partial P_{\eta})\ne 0$ such that
for some $\omega\in N_2^0(P)$ we have $\omega\subset\eta$ and $\widehat{H}_2(P_{\eta\setminus\omega},\partial P_{\eta\setminus\omega})\ne 0$. By Lemma \ref{H7lemma} $P_{\eta}$ should have structure drawn on Fig. \ref{SetsH7}a)-d).
For cases a)-c) for any $\omega\in N_2^0(P)$, $\omega\subset \eta$ the set $P_{\eta\setminus\omega}$ is contractible.
Thus, each $\eta$ has the form $\omega(\mathcal{B}_5)$ for some $5$-belt, and $z_{\eta}=\mu_{\eta}\widetilde{\mathcal{B}_5}$.
Hence, $A_7\subset \boldsymbol{B}_5$, and this finishes the proof.
\end{proof}

Denote $\boldsymbol{I}_7=I_7/A_7$.
\begin{corollary}\label{kappacor}
Let $P$ be a flag simple $3$-polytope with $2\,\rk \boldsymbol{B}_4= \rk\boldsymbol{H}_3$. Then group $\boldsymbol{I}_7$ is a free abelian group of rank 
$$
(m-5)\#\{\text{trivial $4$-belts}\}+(m-4)\#\{\text{nontrivial $4$-belts}\}\geqslant (m-5)\rk\boldsymbol{B}_4.
$$
\end{corollary}
\begin{corollary}\label{BaPog-cor}
A flag simple $3$-polytope $P$ belongs to $\mathcal{P}_{aPog}\setminus\{I^3,M_5\times I\}$ if and only if 
$$
2\,\rk \boldsymbol{B}_4= \rk\boldsymbol{H}_3, \text{ and }\rk \boldsymbol{I}_7=(m-5)\rk\boldsymbol{B}_4.
$$
\end{corollary}

\begin{corollary}\label{BiaPog-cor}
A flag simple $3$-polytope $P$ is an ideal almost Pogorelov polytope 
if and only if 
$$
2\,\rk \boldsymbol{B}_4= \rk\boldsymbol{H}_3=m-2, \text{ and }\,\rk \boldsymbol{I}_7=\frac{(m-5)(m-2)}{2}.
$$
\end{corollary}

\section{Analog of the SCC for almost Pogorelov polytopes}\label{SCCsec}
In this section we generalize the SCC to almost Pogorelov polytopes. 

\begin{definition}
A \textit{$k$-path}, $k\geqslant 1$,
(or a  \textit{thick path}) is a sequence $(F_{i_1},\dots,F_{i_k})$ of faces such 
that successive faces are adjacent. A $k$-path is called \textit{simple}, if all its 
faces are pairwise distinct. By definition a~\textit{length} of a~$k$-path is equal to~$k$. For any two faces 
in~$P$ there is a thick path connecting them. The shortest thick path between 
faces is simple.

A \textit{$1$-loop} is a face. By a \textit{$k$-loop}, $k\geqslant 2$,
we call a closed $(k+1)$-path $(F_{i_1},\dots,F_{i_{k+1}})$:
$F_{i_{k+1}}=F_{i_1}$. We denote the~$k$-loop simply $(F_{i_1},\dots,F_{i_k})$. 
A $k$-loop is called \textit{simple},
if all its faces are distinct. For convenience we assume that the~parameter~$j$
in the~index~$i_j$ of a ~$k$-loop belongs to~$\mathbb Z_k=\mathbb Z/k\mathbb Z$, that is $i_k=i_0$,
$i_{k+1}=i_1$ and so on.
\end{definition}
For a thick path, a $k$-loop or a $k$-belt $\mathcal{L}$ denote by $|\mathcal{L}|$ the union of its faces. The following result generalize
the SCC, which is crucial for Pogorelov polytopes.
\begin{lemma}\label{APb-lemma}
Let $P$ be an almost Pogorelov polytope. Then for three pairwise different faces $\{F_i,F_j,F_k\}$ such that $F_i\cap F_j=\varnothing$ there exists an $l$-belt ($l\geqslant 4$) $\mathcal{B}_l$ such that $F_i,F_j\in \mathcal{B}_l$, $F_k\notin \mathcal{B}_l$, and $F_k$ does not intersect at least one of the two connected components of $|\mathcal{B}_l|\setminus(F_i\cup F_j)$ if and only if $F_k$ does not intersect quadrangles among the faces $F_i$ and $F_j$.
\end{lemma}
\begin{remark}
The same result in the dual setting for simplicial polytopes was earlier proved in \cite[Proposition F.1]{FMW20}.
\end{remark}
\begin{proof}
If $F_i$ is a quadrangle adjacent to $F_k$, then any $l$-belt  $\mathcal{B}_l$ containing $F_i$ and not containing $F_k$ in the closure of each connected component of  $|\mathcal{B}_l|\setminus(F_i\cup F_j)$ contains a face from the $4$-belt around $F_i$ 
adjacent to $F_k$. Hence, $F_k$ intersects both connected components.

Now let $F_k$ do not intersect quadrangles among the faces $F_i$ and $F_j$.

Since $P$ is flag, there is an $s$-belt $\mathcal{B}_1=(F_i,F_{i_1},\dots,F_{i_p},F_j,F_{j_1},\dots,F_{j_q})$, with $F_k\notin \mathcal{B}_1$,  $s=p+q+2$, and $p,q\geqslant 1$. The complement $\partial P\setminus |\mathcal{B}_1|$  consists of two connected components $C_1$ and $C_2$, both homeomorphic to disks. Without loss of generality assume that $F_k\subset\overline{C_1}$. Then either $\overline{C_1}=F_k$, or $\partial C_1\cap F_k$ consists of finite set of disjoint edge-segments $\gamma_1$, $\dots$, $\gamma_d$. 
\begin{figure}
\begin{center}
\includegraphics[height=4cm]{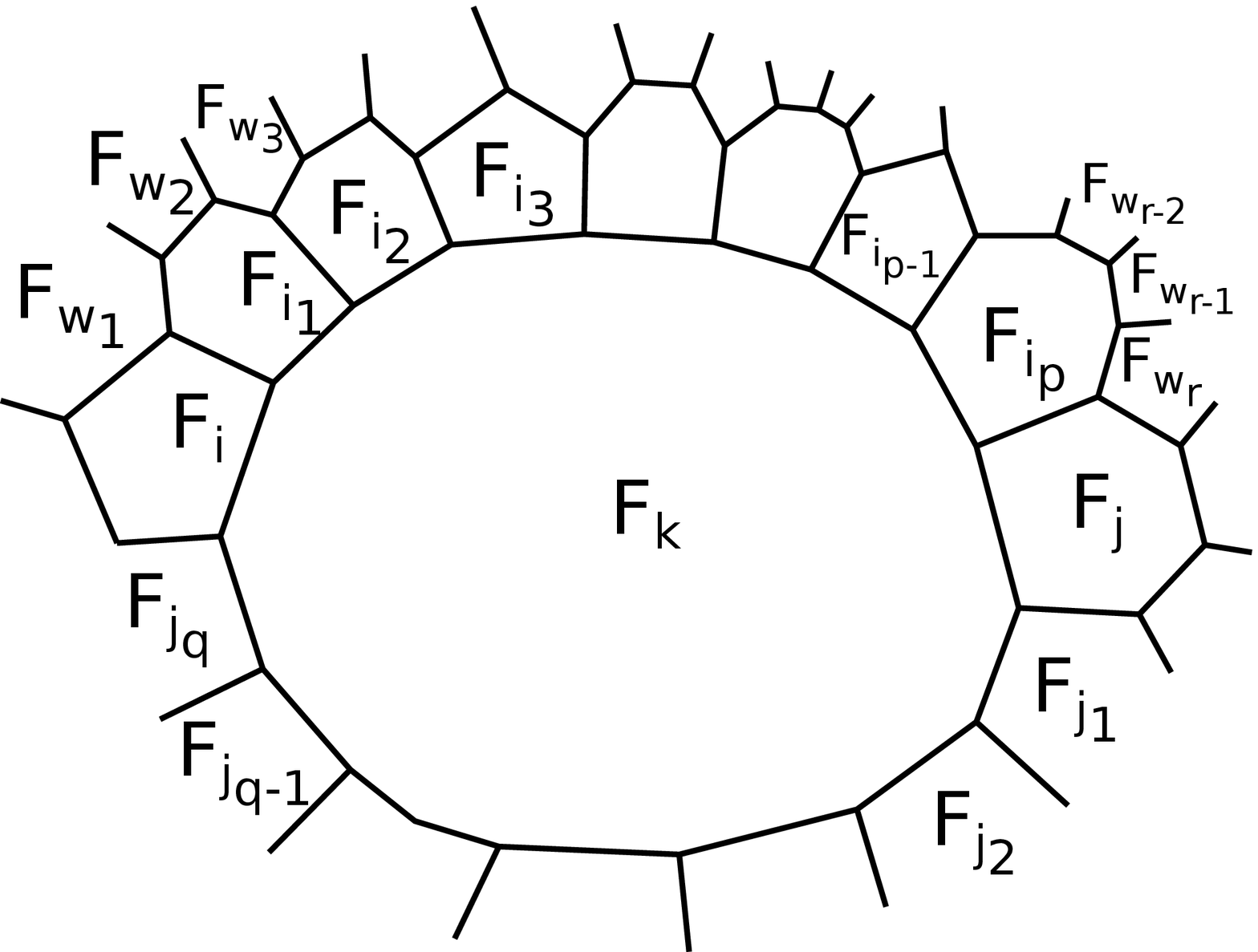}
\end{center}
\caption{Case 1}\label{bijk-1}
\end{figure}

In the first case $\mathcal{B}_1$ surrounds $F_k$, and $F_i$ and $F_j$ are adjacent to $F_k$. 
Consider all faces $\{F_{w_1},\dots, F_{w_r}\}$ in $\overline{C_2}$  adjacent to faces in $\{F_{i_1},\dots, F_{i_p}\}$ 
(see Fig. \ref{bijk-1}), in the order we meet them while walking round $\partial C_2$ from $\overline{C_2}\cap F_i$ to 
$\overline{C_2}\cap F_j$. Then $F_{w_a}\cap F_{j_b}=\varnothing$ for any $a,b$, for otherwise 
$(F_k,F_{j_b},F_{w_a},F_{i_c})$ is a $4$-belt for any $i_c$ with $F_{i_c}\cap F_{w_a}\ne\varnothing$, 
since $F_k\cap F_{w_a}=\varnothing$ and $F_{j_b}\cap F_{i_c}=\varnothing$. 
This belt should surround a quadrangle from $\mathcal{B}_1$ intersecting both $F_{j_b}$ and $F_{i_c}$. 
This quadrangle can be only $F_i$ or $F_j$, which is a contradiction.
We have a thick path $(F_i, F_{w_1},\dots, F_{w_r},F_j)$. 
Consider the shortest thick path of the form $(F_i,F_{w_{s_1}},\dots,F_{w_{s_t}},F_j)$. 
If two faces of this path intersect, then they are successive, else there is a shorter thick path. 
Thus we have the belt $\mathcal{B}_l=(F_i,F_{w_{s_1}},\dots, F_{w_{s_t}},F_j,F_{j_1},\dots, F_{j_q})$ with faces of 
the segment $(F_{w_{s_1}},\dots,F_{w_{s_t}})$ not intersecting $F_k$.

In the second case we can assume that $F_i\cap F_k=\varnothing$ or $F_j\cap F_k=\varnothing$, say $F_i\cap F_k=\varnothing$, for otherwise we can take the belt $\mathcal{B}_1'$ around $F_k$. 
Let $\gamma_a=(F_k\cap F_{u_{a,1}},\dots,F_k\cap F_{u_{a,l_a}})$.   Denote by $\mathcal{U}_a$ 
the thick path $(F_{u_{a,1}},\dots, F_{u_{a,l_a}})$, and by $\mathcal{S}_a$ the segment 
$(F_{s_{a,1}},\dots,F_{s_{a,t_a}})$ of $\mathcal{B}_1$ between $\mathcal{U}_a$ and $\mathcal{U}_{a+1}$. 
Then  $\mathcal{B}_1=(\mathcal{U}_1,\mathcal{S}_1,\mathcal{U}_2,\dots,\mathcal{U}_d,\mathcal{S}_d)$. 

Consider the thick path $\Pi_a=(F_{w_{a,1}},\dots,F_{w_{a,r_a}})\subset \overline{C_2}$ arising while walking round 
$\partial{C_2}$ along edges $\partial{C_2}\cap F_{u_{a,j}}$, $j=1,\dots,l_a$ (see Fig. \ref{bijk-2}). 
Then $\Pi_a\cap \Pi_b=\varnothing$ 
for $a\ne b$, else $(F_w,F_{u_{a,j_1}},F_k,F_{u_{b,j_2}})$ is a $4$-belt for any $F_w\in \Pi_a\cap \Pi_b$ such that 
$F_w\cap F_{u_{a,j_1}}\ne\varnothing$, $F_w\cap F_{u_{b,j_2}}\ne\varnothing$. 
This belt surrounds a quadrangle $F_x$ adjacent to $F_k$. If $F_x\subset \overline{C_1}$, 
then $F_x\cap F_w=\varnothing$, which is a contradiction. 
Otherwise, $F_x\in\mathcal{B}_1$, and the faces $F_{u_{a,j_1}}$ and $F_{u_{b,j_2}}$ are adjacent to it in the belt. 
Then $a=b$ (since $P$ is flag), which is a contradiction. 
Also $F_{w_{a,j_1}}\ne F_{w_{a,j_2}}$ for $j_1\ne j_2$. 
This is true for faces adjacent to the same face $F_{u_{a,i}}$. Let $F_{w_{a,j_1}}= F_{w_{a,j_2}}$. 
If the faces are adjacent to the successive faces $F_{u_{a,i}}$ and $F_{u_{a,i+1}}$, 
then the flagness condition implies that $j_1=j_2$ and $F_{w_{a,j_1}}$ is the face in $\overline{C_2}$ 
intersecting  $F_{u_{a,i}}\cap F_{u_{a,i+1}}$. If the faces are adjacent to non-successive faces 
$F_{u_{a,i}}$ and $F_{u_{a,j}}$, then $(F_{w_{a,j_1}},F_{u_{a,i}},F_k,F_{u_{a,j}})$ is a $4$-belt 
around a quadrangle $F_x$. If $F_x\subset \overline{C_1}$, 
then $F_x\cap F_{w_{a,j_1}}=\varnothing$. Otherwise, $F_x\in\mathcal{B}_1$, 
and the faces $F_{u_{a,i}}$ and $F_{u_{a,j}}$ are adjacent to it in the belt. 
Without loss of generality we can assume that $F_x=F_{u_{a,i+1}}$, and $j=i+2$. Then $j_1=j_2$, 
and $F_{w_{a,j_1}}$ is a unique face in $\overline{C_2}$ 
intersecting  $F_{u_{a,i+1}}$. A contradiction. 

Now consider the thick path $\mathcal{V}_b=(F_{v_{b,1}},\dots,F_{v_{b,c_b}})$ arising while walking round 
$\partial{C_1}$ along edges $\partial{C_1}\cap F_{s_{b,j}}$ (see Fig. \ref{bijk-2}). Then $|\mathcal{V}_a|\cap|\mathcal{V}_b|=\varnothing$ for $a\ne b$, and  $|\Pi_a|\cap |\mathcal{V}_b|=\varnothing$ for any $a,b$, since interiors of the corresponding faces lie in different connected components of $\partial P\setminus(|\mathcal{B}_1|\cup F_k)$. 
\begin{figure}
\begin{center}
\includegraphics[height=9cm]{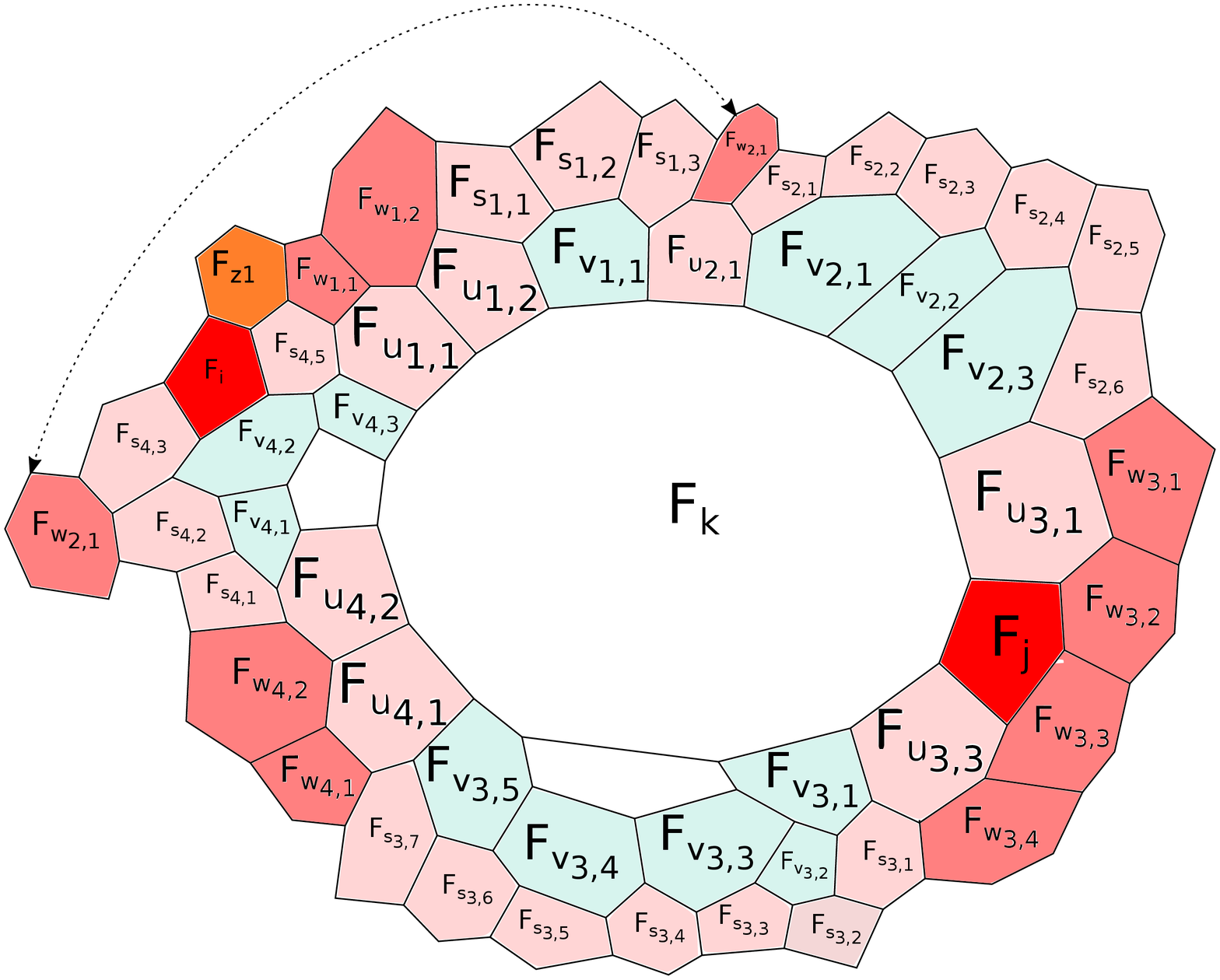}
\end{center}
\caption{Case 2}\label{bijk-2}
\end{figure}

Now we will deform the segments $\mathcal{I}=(F_{i_1},\dots,F_{i_p})$ and $\mathcal{J}=(F_{j_1},\dots,F_{j_q})$ of the belt $\mathcal{B}_1$ to obtain a new belt $(F_i,\mathcal{I}',F_j,\mathcal{J}')$ with $\mathcal{I}'$ not intersecting $F_k$. First substitute the thick path $\Pi_a$ for each segment $\mathcal{U}_a\subset \mathcal{I}$ and the thick path $\mathcal{V}_b$ for each segment $\mathcal{S}_b\subset \mathcal{J}$. Since $F_{s_{a,t_a}}\cap F_{w_{a+1,1}}\ne\varnothing$, $F_{w_{a,r_a}}\cap F_{s_{a,1}}\ne\varnothing$, $F_{v_{a,c_a}}\cap F_{u_{a+1,1}}\ne\varnothing$, and $F_{u_{a,l_a}}\cap F_{v_{a,1}}\ne\varnothing$ for any $a$ and $a+1$ considered $\mod d$, we obtain a loop $\mathcal{L}_1=(F_i,\mathcal{I}_1,F_j,\mathcal{J}_1)$ instead of $\mathcal{B}_1$.  

Since $F_i\cap F_k=\varnothing$, we have $F_i=F_{s_{a_i,f_i}}$ for some $a_i,f_i$. If $F_j=F_{s_{a_j,f_j}}$ for some $a_j,f_j$, then we can assume that $a_i\ne a_j$, else the faces in $\mathcal{I}$ or $\mathcal{J}$ already do not intersect $F_k$, and $\mathcal{B}_1$ is a belt we need. If $F_j=F_{u_{a_j,f_j}}$ for some $a_j$ and some $f_j>1$, then substitute the thick path $(F_{w_{a_j,1}},\dots, F_{w_{a_j,g_j}})$, where $g_j$ -- the first integer with $F_{w_{a_j,g_j}}\cap F_j\ne\varnothing$ (then $F_j\cap F_{u_{a_j,f_j-1}}\cap F_{w_{a_j,g_j}}$ is a vertex), for the segment $(F_{u_{a_j,1}},\dots, F_{u_{a_j,f_j-1}})$ to obtain a loop $\mathcal{L}_2=(F_i,\mathcal{I}_2,F_j,\mathcal{J}_1)$ (else set $\mathcal{I}_2=\mathcal{I}_1$ and $\mathcal{L}_2=\mathcal{L}_1$) with faces in $\mathcal{I}_2$ not intersecting $F_k$.  If $1<f_j<l_{a_j}$, then $F_{w_{a_j,g_j}}\cap F_{u_{a_j,f_j+1}}=\varnothing$, for otherwise $(F_k,F_{u_{a_j,f_j-1}},F_{w_{a_j,g_j}},F_{u_{a_j,f_j+1}})$ is a $4$-belt around the quadrangle $F_j$. Then $F_{w_{a,l}}\cap F_{u_{a_j,r}}=\varnothing$ for any $r\in \{f_j+1,\dots,l_{a_j}\}$ and $a,l$, such that either $a\ne a_j$, or $a=a_j$, and $l\in\{1,\dots,g_j\}$. Hence faces of the segment $(F_{u_{a_j,f_j+1}},\dots, F_{u_{a_j,l_{a_j}}})$ do not intersect faces in $\mathcal{I}_2$. 

Now a face $F_{i_a'}$ of  $\mathcal{I}_2$ can intersect a face $F_{j_b'}$ of $\mathcal{J}_1$ only if $F_{i_a'}=F_{w_{c,h}}$ for some $c,h$, and $F_{j_b'}=F_{s_{a_i,l}}$ for $l<f_i$, or $F_{j_b'}=F_{s_{a_j,l}}$ for $F_j=F_{s_{a_j,f_j}}$ and $l>f_j$.  
In the first case take the smallest  $l$ for all $c,h$, and the correspondent face $F_{w_{c,h}}$. Consider the face $F_{u_{b,g}}=F_{i_e}\in \mathcal{I}$ with $F_{u_{b,g}}\cap F_{w_{c,h}}\ne\varnothing$. Then $\mathcal{L}'=(F_{s_{a_i,l}},F_{s_{a_i,l+1}},\dots,F_i,F_{i_1},\dots,F_{i_e},F_{w_{c,h}})$ is a simple loop. If $f_i<t_{a_i}$, then consider the thick path $\mathcal{Z}_1=(F_{z_{1,1}},\dots,F_{z_{1,y_1}})$ arising while walking in $\overline{C_2}$ along $\partial C_2$ from the face $F_{z_{1,1}}$ intersecting $F_i\cap F_{i_1}$ by the vertex, to the face $F_{z_{1,y_1}}$ preceding $F_{w_{a_i+1,1}}$. Consider the thick path $\mathcal{X}_1=(F_{v_{a_i,1}},\dots, F_{v_{a_i,x_1}})$ with $x_1$ being the first integer with $F_{v_{a_i,x_1}}\cap F_i\ne\varnothing$. Consider the simple curve $\eta\subset\partial P$ consisting of segments connecting the midpoints of the successive edges of intersection of the successive faces of $\mathcal{L}'$.  It divides $\partial P$ into two connected components $\mathcal{E}_1$ and $\mathcal{E}_2$ with $\mathcal{J}_1\setminus(F_{s_{a_i,l}},\dots,F_{s_{a_i,f_i-1}})$ lying in one connected component, say $\mathcal{E}_1$, and $\mathcal{Z}_1$ -- in $\mathcal{E}_2\cup F_{w_{c,h}}$.
Now substitute $\mathcal{X}_1$ for the segment $(F_{s_{a_i,1}},\dots, F_{s_{a_i,f_i-1}})$ of $\mathcal{J}_1$. If $f_i<t_{a_i}$
substitute $\mathcal{Z}_1$ for the segment $(F_{s_{a_i,f_i+1}},\dots, F_{s_{a_i,t_{a_i}}})$ of $\mathcal{I}_2$ to obtain a new loop $(F_i,\mathcal{I}_3,F_j,\mathcal{J}_2)$ with faces in $\mathcal{I}_3$ not intersecting $F_k$. A face $F_{i_a''} $ in $\mathcal{I}_3$ can intersect a face $F_{j_b''}$ in $\mathcal{J}_2$ only if  $F_{i_a''}=F_{w_{c',h'}}$ for some $c',h'$, $F_j=F_{s_{a_j,f_j}}$, and $F_{j_b''}=F_{s_{a_j,l}}$ for $l>f_j$. 
The thick path $\mathcal{Z}_1$ lies in $\mathcal{E}_2\cup F_{w_{c,h}}$ and the segment $(F_j=F_{s_{a_j,f_j}},\dots, F_{s_{a_j,t_{a_j}}})$ lies in $\mathcal{E}_1$; hence intersections of faces in $\mathcal{I}_3$ with faces in $\mathcal{J}_2$ are also intersections of the same faces in $\mathcal{I}_2$ and $\mathcal{J}_1$, and $F_{w_{c',h'}}$ is either $F_{w_{c,h}}$, or lies in $\mathcal{E}_1$. We can apply the same argument for $\mathcal{S}_{a_j}$ as for $\mathcal{S}_{a_i}$ to obtain a new loop $\mathcal{L}_4=(F_i,\mathcal{I}_4,F_j,\mathcal{J}_3)$ with faces in $\mathcal{I}_4$ not intersecting $F_k$ and faces in $\mathcal{J}_3$. Then take the shortest thick path from $F_i$ to $F_j$ in $F_i\cup \mathcal{I}_4\cup F_j$ and the shortest thick path from $F_j$ to $F_i$ in $F_j\cup \mathcal{J}_3\cup F_i$ to obtain a belt we need. 

\begin{figure}[h]
\begin{center}
\includegraphics[height=9cm]{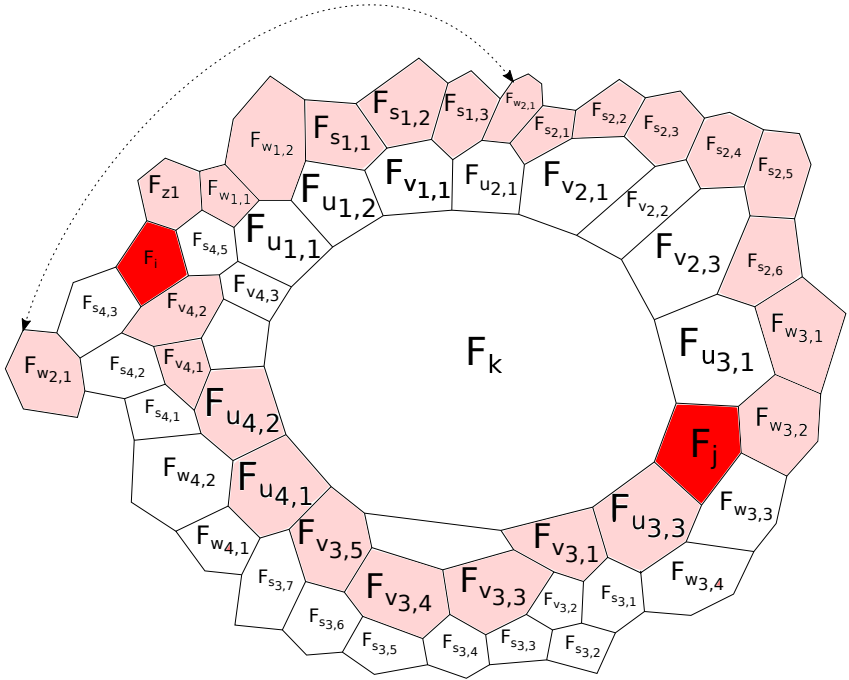}
\end{center}
\caption{Modified belt}\label{bijk-new}
\end{figure}

\end{proof}

\begin{figure}
\begin{center}
\begin{tabular}{cc}
\includegraphics[height=5cm]{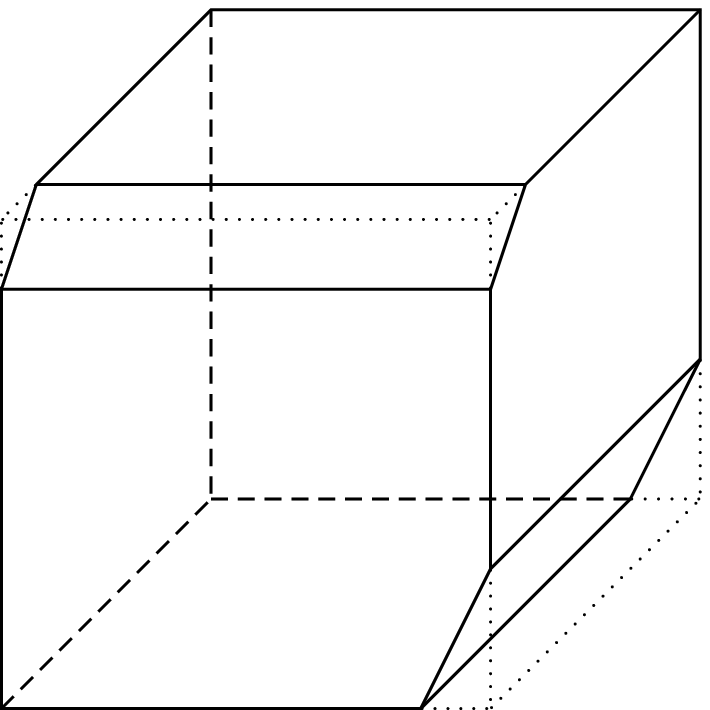}&\includegraphics[height=5cm]{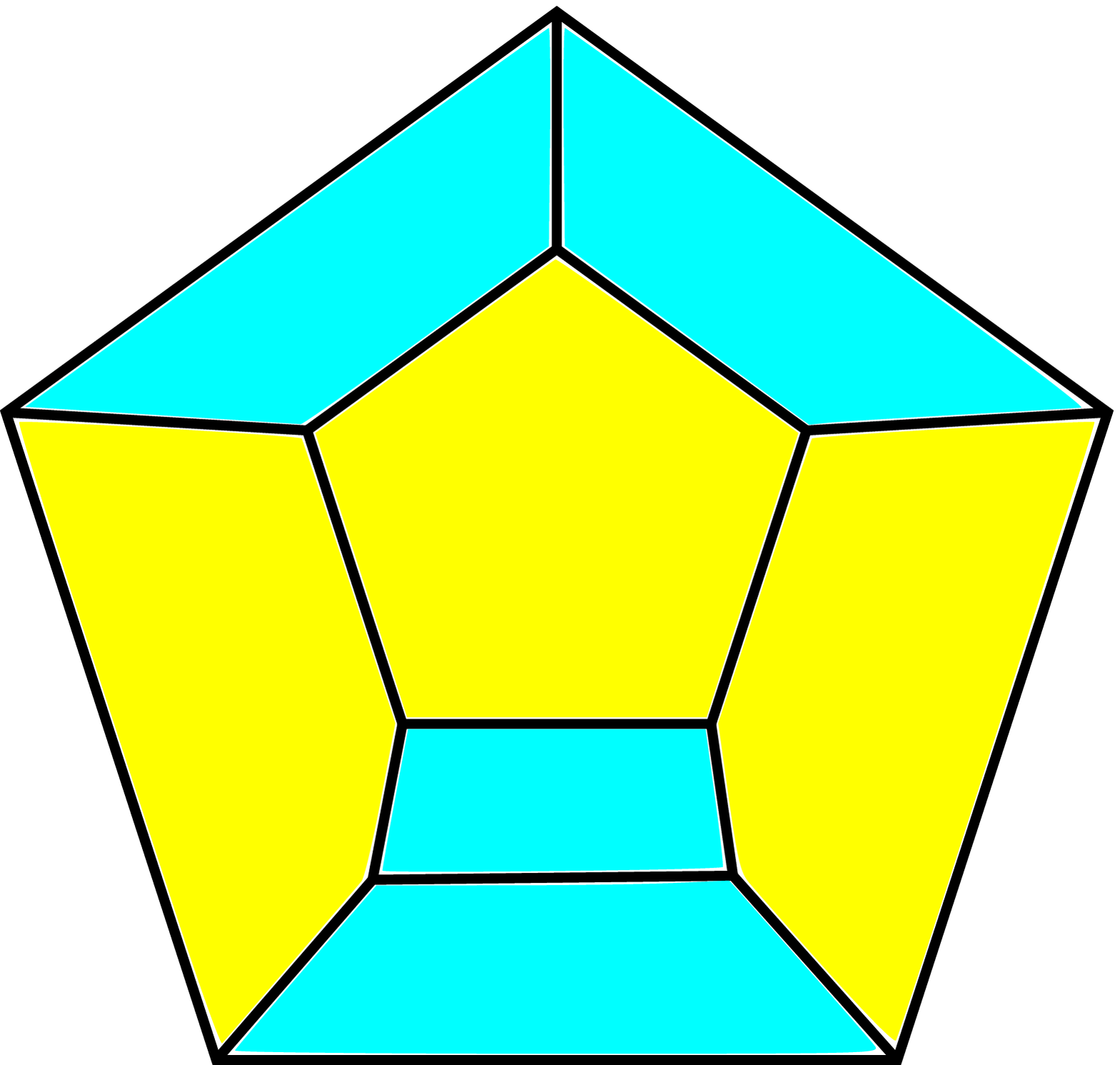}\\
a)&b)
\end{tabular}
\caption{a) A polytope $P_8$; b) its Schlegel diagram}
\label{P4455fig}
\end{center}
\end{figure}

Denote by $P_8$ a polytope obtained from the cube by cutting off two disjoint nonadjacent orthogonal 
edges (Fig. \ref{P4455fig}). It is easy to see, that the polytope $P_8$ is flag. 
\begin{remark}
It is not difficult to show that there
are only two flag $3$-polytopes with $m=8$: the polytope $P_8$, and the hexagonal prism $M_6\times I$. Moreover, $P_8$ 
has five $4$-belts (four trivial and one nontrivial), and $M_6\times I$ has nine $4$-belts (six trivial and three nontrivial). Hence, both of them are $B$-rigid.
\end{remark}

\begin{proposition}\label{apbprop} A flag polytope $P$ is an almost Pogorelov polytope or the polytope $P_8$ if and only if it satisfies 
the condition of Lemma \ref{APb-lemma}. 
\end{proposition}
\begin{proof}
Indeed, let a flag polytope $P$ satisfy the condition of Lemma \ref{APb-lemma}. 
Assume that $P$ has a nontrivial $4$-belt $\mathcal{B}_4=(F_s,F_t,F_u,F_v)$. If  $F_s$
is not adjacent to some faces $F_i\in \overline{C_1}$ and $F_j\in\overline{C_2}$, where  $C_1$ and $C_2$ are 
the connected components  of $\partial P\setminus |\mathcal{B}_4|$, then we can take $F_k=F_s$. 
Any belt $\mathcal{B}$ containing $F_i$ and $F_j$ and not containing $F_k$ should pass through $F_t$ and $F_v$, 
in particular, each component of $|\mathcal{B}|\setminus(F_i\cup F_j)$
intersects $F_k$. A contradiction. Thus, any face of $\mathcal{B}_4$ should be adjacent to all the faces either 
in $\overline{C_1}$, or in $\overline{C_2}$. Without loss of generality assume that $F_s$ is adjacent to all the 
faces in $\overline{C_1}$. Then $\overline{C_1}$ consists of the faces $F_{i_1}$, $\dots$, $F_{i_p}$, where
$(F_t, F_{i_1}, \dots, F_{i_p}, F_v)$ is a part of the belt around $F_s$. We have $p\geqslant 2$, since $\mathcal{B}_4$
is a nontrivial belt. In particular, $F_v$ does not intersect $F_{i_1}$, $\dots$, $F_{i_{p-1}}$, and $F_t$ does not
intersect $F_{i_2},\dots,F_{i_p}$. Then $F_t$ and $F_v$ are adjacent to all the faces in $\overline{C}_2$, and 
$F_u$  -- in $\overline{C}_1$. Since $P$ is flag, $F_t\cap F_{i_1}\cap F_u$, $F_{i_1}\cap F_{i_2}\cap F_u$,
$\dots$, $F_{i_{p-1}}\cap F_{i_p}\cap F_u$, $F_{i_p}\cap F_v\cap F_u$ are vertices. Then  
all the faces $F_{i_1}$, $\dots$, $F_{i_p}$ are quadrangles. The same argument works for $\overline{C_2}$.
Then $P$ has the structure drawn on Fig. \ref{P4Q4}. Let $p\geqslant 3$.
Consider the faces $F_i=F_{i_1}$, $F_j=F_{j_1}$, and  $F_k=F_{i_3}$. Any belt $\mathcal{B}$ containing $F_{i_1}$
contains either  the part $(F_v,F_{i_1},F_{i_2})$, or $(F_s,F_{i_1},F_u)$. In the first case this belt also should contain 
$F_k=F_{i_3}$. In the second case it intersects $F_k=F_{i_3}$ by both connected components of 
$|\mathcal{B}|\setminus(F_i\cup F_j)$. A contradiction. Hence, $p=2$. For the same reason $q=2$. Thus, $P=P_8$.
\begin{figure}
\begin{center}
\includegraphics[height=5cm]{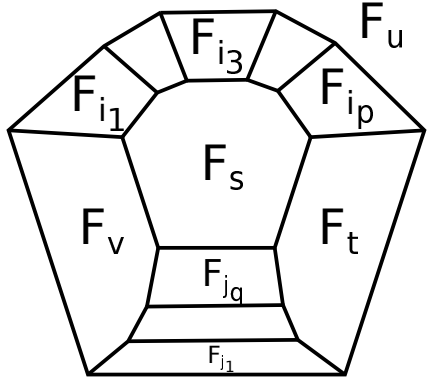}
\caption{A polytope $P$ assumed to satisfy the condition of Lemma \ref{APb-lemma}}
\label{P4Q4}
\end{center}
\end{figure}

At last, let us prove that $P_8$ satisfies the condition of Lemma \ref{APb-lemma}. Its group of combinatorial symmetries 
acts transitively on quadrangles and pentagons. Denote by $\mathcal{B}$ the $4$-belt of pentagons, and by $C_1$
and $C_2$ connected components of its complement in $\partial P$. Each component $\overline{C_1}$ and $\overline{C_2}$
consists of two quadrangles. 

If $F_i$ is a quadrangle adjacent to $F_k$, then any $l$-belt  $\mathcal{B}_l$ containing $F_i$ and not containing $F_k$ in the closure of each connected component of  $|\mathcal{B}_l|\setminus(F_i\cup F_j)$ contains a face from the $4$-belt around $F_i$ 
adjacent to $F_k$. Hence, $F_k$ intersects both connected components.

Now let $F_k$ do not intersect quadrangles among the faces $F_i$ and $F_j$.  
 
If $F_k$ is a quadrangle in $\overline{C_1}$, then
at least one of the faces $F_i$ and $F_j$, say $F_i$, should be a pentagon. 
Since $F_i\cap F_j=\varnothing$, $F_j$ is either the pentagon opposite to $F_i$ in $\mathcal{B}$, 
or a quadrangle in $\overline{C_2}$. In the first case, if $F_i$ intersects both quadrangles in $\overline{C_1}$,
then $\mathcal{B}$ is a belt we need. If $F_i$ intersects only one quadrangle
in $\overline{C}_1$, then  the belt around a quadrangle in $\overline{C_2}$ is a belt we need. 
In the second case $F_i$ is adjacent to both quadrangles in $\overline{C}_1$ and the belt around the quadrangle
in $\overline{C}_2$ different from $F_j$ is a belt we need. 

If $F_k$ is a pentagon,  then it is adjacent to three of four quadrangles of $P$. If both faces $F_i$ and $F_j$
are pentagons, then they are both adjacent to $F_k$, and the belt around the fourth quadrangle is a belt we need. If $F_i$
is the fourth quadrangle, then it is adjacent to all the pentagons different from $F_k$. Then $F_j$ is a quadrangle,
and $F_i$ and $F_j$ lie in different components $\overline{C_1}$ and $\overline{C_2}$. Then $F_j$ is a quadrangle adjacent to $F_k$, which is a contradiction. This finishes the proof.
\end{proof}
\begin{remark}
Thus, we have a characterization of the families of flag, almost Pogorelov and Pogorelov polytopes it terms of the existence of
the belt that passes through $F_i$ and $F_j$, where $F_i\cap F_j=\varnothing$, and does not pass through $F_k$.
It should be mentioned, that there are other characterizations of these families in terms of belts. 
Namely, a simple $3$-polytope is flag 
if and only if any its face is surrounded by a belt \cite{BE17I,BEMPP17}. 
It is almost Pogorelov or the polytope $P_8$ if and only if in addition the loop around any trivial belt
is simple \cite{E19}. A simple $3$-polytope is Pogorelov if and only if any pair of adjacent faces is surrounded 
by a belt \cite{BE17I,BEMPP17}.
Moreover, a simple $3$-polytope is strongly Pogorelov if and only if any its face is surrounded by two belts \cite{B1913}. 
\end{remark}

\section{Generalization of the annihilator lemma}\label{Annsec}
In this section we generalise another crucial tool from \cite{FMW15} -- the annihilator lemma. 
\begin{lemma}\label{h3lemma}
Let  $P$ be a simple $3$-polytope and $\alpha\in H$
$$
\alpha=\sum\limits_{\omega\in N_2(P)}r_{\omega}\widetilde\omega\quad\text{with }|\{\omega\colon r_{\omega}\ne 0\}|\geqslant 2.
$$ 
Then for any $\omega=\{p,q\}$ with  $r_{\omega}\ne 0$ we have 
\begin{enumerate}
\item $\dim {\rm Ann}_H (\alpha)\leqslant \dim {\rm Ann}_H (\widetilde \omega)$; 
\item $\dim {\rm Ann}_H (\alpha)< \dim {\rm Ann}_H (\widetilde \omega)$, if there is $\omega'=\{s,t\}$ 
with $r_{\omega'}\ne\varnothing$ admitting an $l$-belt $\mathcal{B}_l$ containing $F_s$ and $F_t$ such that  
either $F_p$ or $F_q$ does not belong to $\mathcal{B}_l$ and does not intersect at least one of the two connected components of $|\mathcal{B}_l|\setminus(F_s\cup F_t)$.
\end{enumerate}
\end{lemma}
\begin{proof} 
Choose a complementary subspace $C_{\omega}$ to ${\rm Ann}_H (\widetilde \omega)$ in $H$ as a direct sum of complements $C_{\omega,\tau}$ to ${\rm Ann}_H (\widetilde \omega)\cap \widehat{H}_*(P_{\tau},\partial P_{\tau},\mathbb Q)$ in $\widehat{H}_*(P_{\tau},\partial P_{\tau},\mathbb Q)$ for all $\tau\subset[m]\setminus\omega$. Then for any $\beta\in C_{\omega}\setminus\{0\}$ we have $\beta \widetilde{\omega}\ne0$, which is equivalent to the fact that $\beta=\sum\beta_{\tau}$, $\beta_{\tau}\in C_{\omega,\tau}$, $\tau\subset[m]\setminus\omega$, with $\beta_{\tau_\beta}\widetilde\omega\ne0$ for some $\tau_\beta\subset[m]\setminus\omega$. Moreover for any $\omega'\ne\omega$ with $r_{\omega'}\ne 0$ and $\tau\subset[m]\setminus\omega$, $\tau\ne\tau_\beta$, we have $\tau_\beta\sqcup\omega\notin\{\tau\cup \omega',\tau_\beta\cup\omega',\tau\sqcup\omega\}$; hence $(\beta\cdot\alpha)_{\tau_{\beta}\sqcup\omega}= r_{\omega}\beta_{\tau_\beta}\cdot\widetilde{\omega}\ne0$, and $\beta\alpha\ne 0$. Then $C_{\omega}$ forms a direct sum with ${\rm Ann}(\alpha)$, and $\dim {\rm Ann}_H (\alpha)\leqslant \dim {\rm Ann}_H (\widetilde \omega)$.

Now consider some $\omega'\ne\omega$, $|\omega'|=2$, $r_{\omega'}\ne 0$, 
satisfying additional condition of the lemma say for $F_q$. 
Then there is an $l$-belt $\mathcal{B}_l$ such that 
$F_s,F_t\in \mathcal{B}_l$,  $F_q\notin\mathcal{B}_l$, and $F_q$ does not intersect one of the two connected components 
$B_1$ and $B_2$ of $|\mathcal{B}_l|\setminus(F_s\cup F_t)$, say $B_1$. 
Take $\xi=[\sum_{i\colon F_i\subset B_1}F_i]\in \widehat{H}_2(P_{\tau},\partial P_{\tau},\mathbb Q),\,\tau=\{i\colon F_i\in \mathcal{B}_l\setminus\{F_s,F_t\}\}$, and $[F_s]\in \widehat{H}_2(P_{\omega'},\partial P_{\omega'},\mathbb Q)$. 
Then $\xi \cdot [F_s]$ is a generator in $H_1(\mathcal{B}_l,\partial\mathcal{B}_l,\mathbb Q)\simeq \mathbb Q$. 
On the other hand, take $[F_q]\in\widehat{H}_2(P_{\omega},\partial P_{\omega},\mathbb Q)$. 
Then either $F_p\in \mathcal{B}_l\setminus\{F_s,F_t\}$, and $\xi \cdot\widetilde\omega=0$, 
since $\tau\cap \omega\ne\varnothing$,  or $F_p\notin \mathcal{B}_l\setminus\{F_s,F_t\}$, 
and $\pm \xi\cdot\widetilde\omega=\xi\cdot [F_q]=0$, since $F_q$ does not intersect $B_1$. 
In both cases  $\xi\in{\rm Ann}(\widetilde\omega)$ and $\xi\cdot\widetilde{\omega'}\ne 0$. 
Then $\xi\cdot \alpha\ne 0$, since $\tau\sqcup\omega'\ne\tau\sqcup\omega_1$ for $\omega_1\ne\omega'$. 
Consider any $\beta=\sum_{\tau\subset[m]\setminus\omega}\beta_\tau\in C_{\omega}\setminus\{0\}$. 
We have  $(\beta\cdot\alpha)_{\tau_\beta\sqcup\omega}\ne 0$. If $(\xi\cdot\alpha)_{\tau_\beta\sqcup\omega}\ne 0$, 
then since $\xi$ is a homogeneous element, 
$(\xi\cdot\alpha)_{\tau_\beta\sqcup\omega}=r_{\omega_1}\xi\cdot\widetilde{\omega_1}$ 
for $\omega_1=(\tau_\beta\sqcup\omega)\setminus\tau=\{q,r\}$, $r\in[m]$. 
We have $\xi\cdot\widetilde{\omega_1}=\pm\xi\cdot[F_q]=0$, since $F_q$ does not intersect $B_1$. A contradiction. 
Thus, $((\xi+\beta)\cdot\alpha)_{\tau_\beta\sqcup\omega}=(\beta\cdot\alpha)_{\tau_\beta\sqcup\omega}\ne 0$; 
hence $(\xi+\beta)\cdot \alpha\ne0$, and the space $\langle\xi\rangle\oplus C_{\omega}$ 
forms a direct sum with ${\rm Ann}_H(\alpha)$. This finishes the proof. 
\end{proof}
\begin{remark}
In Proposition \ref{anneq} we show that for the $3$-dimensional associahedron, which is the simplest almost
Pogorelov polytope without adjacent quadrangles,  there is an example when $\dim {\rm Ann}_H (\alpha)=\dim {\rm Ann}_H (\widetilde \omega)$.
\end{remark}
\begin{corollary}
Ley $P$ be a simple $3$-polytope, and $\omega\in N_2(P)$ such that
$$ 
\dim {\rm Ann}_H (\alpha)< \dim {\rm Ann}_H (\widetilde \omega)
$$
for any $\alpha=\lambda_\omega\omega+\sum\limits_{\omega'\in N_2(P)\setminus\{\omega\}}\lambda_{\omega'}\widetilde{\omega'}$ with $\lambda_{\omega}\ne0$ and $\lambda_{\omega'}\ne 0$ at least for one $\omega'$, then
for any isomorphism of graded rings $\varphi\colon H^*(\mathcal{Z}_P)\to H^*(\mathcal{Z}_Q)$ for a simple $3$-polytope $Q$
we have 
$\varphi(\widetilde{\omega})=\pm \widetilde{\omega'}$ for some $\omega'\in N_2(Q)$.
\end{corollary}
\begin{proof}
Indeed, we have 
$$
\varphi(\widetilde{\omega})=\sum\limits_{\omega''\in N_2(Q)}\lambda_{\omega''}\widetilde{\omega''}.
$$
Then at least for one $\omega''$ with $\lambda_{\omega''}\ne0$ we have 
$$
\varphi^{-1}(\omega'')=\mu_{\omega}\widetilde{\omega}+\sum\limits_{\omega'\in N_2(P)\setminus\{\omega\}}\mu_{\omega'}
\widetilde{\omega'}\text{ with }\mu_{\omega}\ne 0.
$$
If $\mu_{\omega'}\ne 0$, then by Lemma \ref{h3lemma} we have in cohomology over $\mathbb Q$:
$$
\dim {\rm Ann}_{H(P)} (\widetilde \omega)>\dim {\rm Ann}_{H(P)} (\varphi^{-1}(\widetilde{\omega''}))=
\dim {\rm Ann}_{H(Q)} (\widetilde{\omega''})\geqslant \dim {\rm Ann}_{H(Q)} (\varphi(\widetilde{\omega})).
$$
A contradiction. Hence $\mu_{\omega'}=0$ for all $\omega'$, and $\varphi(\widetilde{\omega})=\pm \widetilde{\omega''}$.
\end{proof}
\begin{definition}\label{good-bad}
Let $P$ be a simple $3$-polytope. We will call an element $\omega'=\{s,t\}\in N_2(P)$ {\it good} for an element 
$\omega=\{p,q\}\in N_2(P)$, if theres is an $l$-belt $\mathcal{B}_l$ containing $F_s$ and $F_t$ such that  
either $F_p$ or $F_q$ does not belong to $\mathcal{B}_l$ and does not intersect at least one of the two connected components of $\mathcal{B}_l\setminus\{F_s,F_t\}$. Otherwise, we call $\omega'$ {\it bad} for $\omega$.
\end{definition}
\begin{corollary}\label{wcor}
Let $P$ be a simple $3$-polytope. If any $\omega'\in N_2(P)\setminus\{\omega\}$ is good for an element $\omega\in N_2(P)$, 
then for any isomorphism of graded rings $\varphi\colon H^*(\mathcal{Z}_P)\to H^*(\mathcal{Z}_Q)$ for a simple $3$-polytope $Q$
we have 
$\varphi(\widetilde{\omega})=\pm \widetilde{\omega'}$ for some $\omega'\in N_2(Q)$.
\end{corollary}
\begin{corollary}\label{gbcor}
Let $P$ be a simple $3$-polytope and $\omega\in N_2(P)$. Then for any isomorphism of graded rings $\varphi\colon H^*(\mathcal{Z}_P)\to H^*(\mathcal{Z}_Q)$ for a simple $3$-polytope $Q$
we have 
$\varphi(\widetilde{\omega})=\sum\limits_{\omega''\in N_2(Q)} \lambda_{\omega''}\widetilde{\omega''}$,
where there is $\omega''\in N_2(Q)$ such that $\lambda_{\omega''}\ne 0$ and all the other elements in $N_2(Q)$
with nonzero coefficients are bad for $\omega''$. Moreover, this is valid for any $\omega''\in N_2(Q)$ 
such that $\lambda_{\omega''}\ne 0$ and $\varphi^{-1}(\widetilde{\omega''})$ has a nonzero coefficient at  
$\widetilde{\omega}$. 
\end{corollary}
\begin{proof}
Indeed, let 
$$
\varphi(\widetilde{\omega})=\sum\limits_{\omega''\in N_2(Q)}\lambda_{\omega''}\widetilde{\omega''}.
$$
Then at least for one $\omega''$ with $\lambda_{\omega''}\ne0$ we have 
$$
\varphi^{-1}(\omega'')=\mu_{\omega}\widetilde{\omega}+\sum\limits_{\omega'\in N_2(P)\setminus\{\omega\}}\mu_{\omega'}
\widetilde{\omega'}\text{ with }\mu_{\omega}\ne 0.
$$
If $\lambda_{\omega''_1}\ne 0$, and $\omega''_1$  is good for $\omega''$, then by Lemma \ref{h3lemma} we have in cohomology over $\mathbb Q$:
$$
\dim {\rm Ann}_{H(P)} (\widetilde \omega)\geqslant \dim {\rm Ann}_{H(P)} (\varphi^{-1}(\widetilde{\omega''}))=
\dim {\rm Ann}_{H(Q)} (\widetilde{\omega''})> \dim {\rm Ann}_{H(Q)} (\varphi(\widetilde{\omega})).
$$
A contradiction. Hence $\lambda_{\omega''_1}=0$ for all elements $\omega''_1\in N_2(Q)$, which are good for $\omega''$.
Moreover, by a similar argument, $\mu_{\omega'}=0$ for all elements, which are good for $\omega$.
\end{proof}

\begin{corollary}
Let $P$ be an almost Pogorelov polytope different from the cube and the pentagonal prism, and $\omega=\{p,q\}\in N_2(P)$. 
Then an element $\omega'=\{s,t\}\in N_2(P)\setminus\{\omega\}$ is bad for $\omega$ 
if and only if at least one of the faces $F_s$ and $F_t$ is a quadrangle, and either $F_p$ and $F_q$ are opposite 
faces of the $4$-belt around this face, or exactly one of the faces $F_p$ and $F_q$ belongs to this belt, and the other
face among $F_s$ and $F_t$ is either the other face among $F_p$ and $F_q$, or is a quadrangle adjacent to it.  In particular,
\begin{enumerate}
\item If $F_p$ and $F_q$ are not adjacent to quadrangles, then any element $\omega'\in N_2(P)\setminus\{\omega\}$ is good
for $\omega$. In particular, this is valid, if both $F_p$ and $F_q$ are quadrangles.
\item If $F_p$ is a quadrangle and $F_q$ is adjacent to quadrangles, then
$\omega'\in N_2(P)\setminus\{\omega'\}$ is bad for $\omega$ if and only if $\omega'=\{p,r\}$, 
where $F_r$ is a quadrangle adjacent to $F_q$.
\item If each of the faces $F_p$ and $F_q$ is adjacent to quadrangles,  and $F_p$ and $F_q$ do not belong to any $4$-belt, 
then  $\omega'=\{s,t\}\in N_2(P)\setminus\{\omega\}$ 
is bad for $\omega$ if and only if either one of the faces $F_s$ and $F_t$ is $F_p$ or $F_q$ 
and the other face is a quadrangle adjacent to the other face, or both $F_s$ and $F_t$ are quadrangles, 
and one of them is adjacent to exactly $F_p$, and the other -- to exactly $F_q$.  
\item If the faces $F_p$ and $F_q$  belong to a $4$-belt around the face $F_r$, then 
$\omega'=\{s,t\}\in N_2(P)\setminus\{\omega\}$ is bad for $\omega$ if and only if either $\omega'=\{r,s\}$, or
one of the faces $F_s$ and $F_t$ is $F_p$ or $F_q$ and the other face is a 
quadrangle adjacent to the other face,
or both $F_s$ and $F_t$ are quadrangles, and one of them is adjacent to exactly $F_p$, 
and the other -- to exactly $F_q$.  
\end{enumerate}
\end{corollary}
\begin{proof}
Without loss of generality assume that $q\notin \omega'$. If $\omega'$ is bad for $\omega$, then by Lemma \ref{APb-lemma} one of the faces $F_s$ and 
$F_t$, say $F_s$, is a quadrangle adjacent to $F_q$. Then $F_p\ne F_s$. If $F_p\ne F_t$, then either $F_p$ is adjacent to $F_s$, or $F_t$ is a quadrangle adjacent to $F_p$. 

On the other hand, let $F_s$ be a quadrangle adjacent to $F_q$. Then $F_p\ne F_s$. If either $F_p=F_t$, or  
$F_p\ne F_t$ and $F_p$ is adjacent to $F_s$, or $F_t$ is a quadrangle adjacent to  $F_p$, then
$\omega'$ is bad for $\omega$. 
\end{proof}

\begin{corollary}\label{n4cor}
Let $P$ be an almost Pogorelov polytope different from the cube and the pentagonal prism, and $\omega=\{p,q\}\in N_2(P)$
such that $F_p$ and $F_q$ are not adjacent to quadrangles. Then for any isomorphism of graded rings $\varphi\colon H^*(\mathcal{Z}_P)\to H^*(\mathcal{Z}_Q)$ for a simple $3$-polytope $Q$
we have 
$\varphi(\widetilde{\omega})=\pm \widetilde{\omega'}$ for some $\omega'\in N_2(Q)$
\end{corollary}
\begin{proof}
This follows directly from Corollary \ref{wcor}.
\end{proof}

\begin{corollary}\label{Bcor}
Let $P$ be an almost Pogorelov polytope different from the cube and the pentagonal prism, and $\mathcal{B}_k$ be a
$k$-belt such that the set $|\mathcal{B}_k|$ does not have common points with quadrangles. 
Then for any isomorphism of graded rings $\varphi\colon H^*(\mathcal{Z}_P)\to H^*(\mathcal{Z}_Q)$ 
for a simple $3$-polytope $Q$ we have 
$\varphi(\widetilde{\mathcal{B}_k})=\pm \widetilde{\mathcal{B}_k'}$ for some $k$-belt $\mathcal{B}'$ of $Q$.
\end{corollary}
\begin{proof}
We know from Section \ref{Bsec} that 
$\varphi(\widetilde{\mathcal{B}_k})=\sum_j\mu_j\widetilde{\mathcal{B}_{k,j}'}$ for $k$-belts $\mathcal{B}_{k,j}'$ of $Q$. 

\begin{lemma}\label{Hdivlemma}
Let $P$ be a simple $3$-polytope and $\omega=\{p,q\}\in N_2(P)$. Then an element
$$
x=\sum_\eta \alpha_{\eta}+\sum_\tau \beta_{\tau}\in \bigoplus\limits_{\eta}\widetilde{H}^0(P_{\eta})\bigoplus\limits_{\tau}
\widetilde{H}^1(P_{\tau})
$$ 
is divisible by $\widetilde{\omega}$ if and only if $\alpha_\eta=0$ for all $\eta$, and each nonzero $\beta_{\tau}$ is
divisible by $\widetilde{\omega}$ (in particular, $\omega\subset\tau$ and $\widetilde{H}^0(P_{\tau\setminus\omega})\ne 0$).
\end{lemma}
\begin{proof}
If $x$ is divisible by $\widetilde{\omega}$, then $x=\widetilde{\omega}\sum_{\zeta}\xi_{\zeta}$ for $\xi_{\zeta}\in \widetilde{H}^0(P_{\zeta})$. Then $x=\sum\limits_{\omega\cap \zeta=\varnothing}\widetilde{\omega}\xi_{\zeta}$, where 
$\widetilde{\omega}\xi_{\zeta}\in \widetilde{H}^1(P_{\omega\sqcup\zeta})$, and $\omega\sqcup \zeta\ne\omega\sqcup\zeta'$ for $\zeta\ne\zeta'$. Thus, each nonzero summand $\beta_\tau$ has the form $\widetilde{\omega}\xi_{\zeta}$, in particular, $\omega\subset \tau$, and $\widetilde{H}^0(P_{\tau\setminus\omega})=\widetilde{H}^0(P_{\zeta})\ne 0$.
\end{proof}

\begin{corollary}\label{Bdivcor}
Let $P$ be a simple $3$-polytope and $\omega=\{p,q\}\in N_2(P)$. Then an element
$x=\sum_j\mu_j\widetilde{\mathcal{B}_{k,j}}\in\boldsymbol{B}_k$ is divisible by $\widetilde{\omega}$
if and only if each $\widetilde{\mathcal{B}_{k,j}}$ with $\mu_j\ne 0$ is divisible by $\widetilde{\omega}$,
and if and only if each belt $\mathcal{B}_{k,j}$ with $\mu_j\ne 0$ contains $F_p$ and $F_q$.
\end{corollary}
\begin{proof}
This follows directly from Lemma \ref{Hdivlemma}, since the element $\widetilde{\mathcal{B}_{k,j}}$ is divisible by $\widetilde{\omega}$ if and only if $\omega\subset \omega(\mathcal{B}_{k,j})$, that is $\mathcal{B}_{k,j}$ contains $F_p$ and $F_q$.
\end{proof}
Take any $\omega=\{p,q\}\in N_2(P)$, $\omega\subset \omega(\mathcal{B}_k)$.  
Since both faces $F_p$ and $F_q$ do not intersect quadrangles, Corollary \ref{n4cor} implies that 
$\varphi(\widetilde{\omega})=\pm\widetilde{\omega'}$ for some
$\omega'\in N_2(Q)$. Then  $\omega'\subset\omega(\mathcal{B}_{k,j}')$ for each $\mu_j\ne 0$.
Thus, the isomorphism $\varphi$ maps the set $\{\pm \widetilde{\omega}\colon\omega\subset\omega(\mathcal{B}_k)\}$ bijectively to the corresponding set of any $\mathcal{B}_{k,j}'$ with $\mu_j\ne 0$.  But such a set defines uniquely the $k$-belt; hence we have only one nonzero $\mu_j$, which should be equal to $\pm 1$. This finishes the proof.
\end{proof}

\section{$B$-rigid subsets for almost Pogorelov polytopes}\label{Brssec}
In this section we study $B$-rigid subsets for almost Pogorelov polytopes.
\begin{lemma}\label{2dimlemma}
The set of subgroups $G(\mathcal{B}_4)\subset \boldsymbol{H}_3$ of rank $2$ generated by elements $[\widetilde{\omega}]=\widetilde{\omega}+A_3$ and $[\widetilde{\omega'}]=\widetilde{\omega'}+A_3$, where $\omega\sqcup \omega'=\omega(\mathcal{B}_4)$ is $B$-rigid in the class of flag $3$-polytopes with $2\,\rk \boldsymbol{B}_4= \rk\boldsymbol{H}_3$.
\end{lemma}
\begin{proof} 
The elements $\{[\widetilde{\omega}]\}$, where $\omega\in N_2(P)$ are subsets of $\omega(\mathcal{B}_4)$ for $4$-belts, form a basis in $\boldsymbol{H}_3$. Take an element 
$$
\alpha=\lambda_1[\widetilde{\omega_1}]+\lambda_1'[\widetilde{\omega_1'}]+\dots+\lambda_k[\widetilde{\omega_k}]+\lambda_k'[\widetilde{\omega_k'}],
$$
where $\mathcal{B}_{4,1}$,\dots, $\mathcal{B}_{4,k}$ are all $4$-belts, and 
$\omega_i\sqcup \omega'_i=\omega(\mathcal{B}_{4,i})$. We have $[\widetilde{\omega_i}]\cdot[\widetilde{\omega_i'}]=\widetilde{\mathcal{B}_{4,i}}$.

As mentioned above we have an embedding 
$H^*(\mathcal{Z}_P)\subset H^*(\mathcal{Z}_P)\otimes \mathbb Q=H^*(\mathcal{Z}_P,\mathbb Q)$.
Images of these subgroups are $2$-dimensional linear subspaces in $\boldsymbol{H}_3\otimes\mathbb Q$.
For the element $\alpha\in \boldsymbol{H}_3\otimes\mathbb Q$ define a linear subspace
$$
A(\alpha)=\{\beta\in \boldsymbol{H}_3\otimes\mathbb Q \colon \alpha\cdot\beta=0\} 
$$
We have
$$
\beta=\mu_1[\widetilde{\omega_1}]+\mu_1'[\widetilde{\omega_1'}]+\dots+\mu_k[\widetilde{\omega_k}]+\mu_k'[\widetilde{\omega_k'}],
$$
and 
$$
\alpha\cdot\beta=\sum\limits_{i=1}^k(\lambda_i\mu_i'-\lambda_i'\mu_i)\widetilde{\mathcal{B}_{4,i}}.
$$
Then $\beta\in A(\alpha)$ if and only if $\lambda_i\mu_i'-\lambda_i'\mu_i=0$ for all $i=1,\dots, k$.
Thus, $\dim A(\alpha)=\dim \boldsymbol{H}_3\otimes \mathbb Q-l$, where $l$ is the number of $i$
with $(\lambda_i,\lambda_i')\ne (0,0)$. Hence, maximal dimension is for $\alpha\in\langle[\widetilde{\omega_i}],[\widetilde{\omega_i'}]\rangle$ for some 
$i$. Therefore, the collection of these subspaces is $B$-rigid over $\mathbb Q$, and the collection of 
the corresponding subgroups is $B$-rigid.
\end{proof}
\begin{corollary}\label{4bbr}
The set of elements
$$
\{\pm \widetilde{\mathcal{B}_4}\colon \mathcal{B}_4\text{ is a $4$-belt}\}\subset H^6(\mathcal{Z}_P)
$$
is $B$-rigid in the class of flag $3$-polytopes with $2\,\rk \boldsymbol{B}_4= \rk\boldsymbol{H}_3$.
\end{corollary}
\begin{proof}
Indeed, $\widetilde{\mathcal{B}_{4,i}}$ is a generator of the image of the bilinear mapping 
$$
\langle[\widetilde{\omega_i}],[\widetilde{\omega_i'}]\rangle\times \langle[\widetilde{\omega_i}],[\widetilde{\omega_i'}]\rangle\to H^6(\mathcal{Z}_P).
$$
\end{proof}

\begin{corollary}\label{bijcor}
Let $P,Q\in \mathcal{P}_{aPog}\setminus\{I^3,M_5\times I\}$. Then any isomorphism of graded rings 
$\varphi\colon H^*(\mathcal{Z}_P)\to H^*(\mathcal{Z}_Q)$ induces a bijection $\varphi_4$ between the sets of quadrangles of 
$P$ and $Q$ by the rule
$$
\varphi_4(F_i)=F_{i'}', \text{ where $\varphi(\widetilde{\mathcal{B}_i})=\pm \widetilde{\mathcal{B}_{i'}'}$ for 
$4$-belts  $\mathcal{B}_i$ and $\mathcal{B}_{i'}'$ around $F_i$ and $F_{i'}'$.}
$$
\end{corollary}
\begin{proof}
Indeed, any $4$-belt around a quadrangle does not surround a quadrangle on the other side, for otherwise $P=I^3$.
\end{proof}
\begin{notation}\label{f4not}
In what follows we will use notations from Corollary \ref{bijcor}, namely  $\varphi_4$ for the isomorphism
between the sets of quadrangles of polytopes $P,Q\in \mathcal{P}_{aPog}\setminus\{I^3,M_5\times I\}$ induced by
an isomorphism of graded rings 
$\varphi\colon H^*(\mathcal{Z}_P)\to H^*(\mathcal{Z}_Q)$, and
$\mathcal{B}_i$ and $\mathcal{B}_{i'}'$ for the $4$-belts around quadrangles $F_i$ and $F_{i'}'$ such that 
$\varphi_4(F_i)=F_{i'}'$.
\end{notation}

\begin{lemma}\label{wb4lemma}
The set of elements:
$$
\{\pm[\widetilde{\omega}]\colon \omega\in N_2(P),\omega\subset \omega(\mathcal{B}_4)\text{ for some $4$-belt } \mathcal{B}_4\}\subset \boldsymbol{H}_3
$$
is $B$-rigid in the class of flag $3$-polytopes with $2\,\rk \boldsymbol{B}_4= \rk\boldsymbol{H}_3$.
\end{lemma}
\begin{proof}
Consider the $4$-belt $\mathcal{B}_{4,i}$ corresponding to the pair $\omega_i,\omega_i'\in N_2(P)$.
Let $\omega_i=\{p,q\}$, $\omega_i'=\{s,t\}$. Consider the belt $\mathcal{B}_l$ around the face $F_p$. It contains both
$F_s$ and $F_t$. Let $B_1$ and $B_2$ be the connected components of $\mathcal{B}_l\setminus\{F_s,F_t\}$. They lie 
in different connected components of $\partial P\setminus\{\mathcal{B}_{4,i}\}$. 
Let  $F_q$ intersect  some face $F_a\in B_1$  and some face $F_b\in B_2$. Then $(F_p,F_a,F_q,F_b)$ is a $4$-belt. It
is different from $\mathcal{B}_{4,i}$ and contains $F_p$ and $F_q$. A contradiction to  Lemma \ref{2unlemma}.

Now let $\varphi \colon H^*(\mathcal{Z}_P)\to H^*(\mathcal{Z}_Q)$ be an isomorphism of graded rings for 
flag polytopes $P$ and $Q$ with $2\,\rk \boldsymbol{B}_4= \rk\boldsymbol{H}_3$. 
By Lemma \ref{2dimlemma} 
$$
\varphi(\widetilde{\omega_i})=\lambda_j\widetilde{\omega_j}+\lambda_j'\widetilde{\omega_j'}+\sum\limits_{\widetilde{\omega}\in A_3(Q)}\mu_{\omega}\widetilde{\omega}
$$
for some $j\in\{1,\dots,k\}$. Since $\varphi^{-1}(A_3(Q))\subset A_3(P)$, for at least one of the elements  $\widetilde{\omega_j}$ and $\widetilde{\omega_j'}$ with nonzero coefficient, say for $\widetilde{\omega_j}$, its preimage contains $\widetilde{\omega_i}$ with a  nonzero coefficient: 
$$
\varphi^{-1}(\widetilde{\omega_j})=\tau_i\widetilde{\omega_i}+\tau_i'\widetilde{\omega_i'}+\sum\limits_{\widetilde{\omega}\in A_3(P)}\eta_{\omega}\widetilde{\omega}, \qquad \tau_i\ne0.
$$
If $\tau_i'\ne 0$, then by Lemma \ref{h3lemma} we have in cohomology over $\mathbb Q$:
$$
\dim {\rm Ann}_{H(P)} (\widetilde \omega_i)>\dim {\rm Ann}_{H(P)} (\varphi^{-1}(\widetilde{\omega_j}))=
\dim {\rm Ann}_{H(Q)} (\widetilde{\omega_j})\geqslant \dim {\rm Ann}_{H(Q)} (\varphi(\widetilde{\omega_i})).
$$
A contradiction. Hence $\tau_i'=0$ and  for the mapping $\widehat{\varphi}\colon \boldsymbol{H}_3(P)\to \boldsymbol{H}_3(Q)$
we have: $\widehat{\varphi}^{-1}([\widetilde\omega_j])=\tau_i[\widetilde\omega_i]$.
Then $\tau_i=\pm 1$ and  $\widehat{\varphi}([\widetilde\omega_i])=\pm[\widetilde\omega_j]$. This finishes the proof.
\end{proof}

\begin{lemma}
The sets of elements:
$$
\{\pm[\widetilde{\omega}]\colon \omega\in N_2(P),\omega\subset \omega(\mathcal{B}_4)\text{ for a trivial $4$-belt } \mathcal{B}_4\}\subset \boldsymbol{H}_3
$$
and
$$
\{\pm \widetilde{\mathcal{B}_4}\colon \mathcal{B}_4\text{ is a trivial $4$-belt}\}\subset H^6(\mathcal{Z}_P)
$$
are $B$-rigid in the class of flag $3$-polytopes with $2\,\rk \boldsymbol{B}_4= \rk\boldsymbol{H}_3$.
\end{lemma}
\begin{proof}
Indeed, any generator of $\widetilde{H}^1(P_{\omega})=\mathbb Z$ for $P_{\omega}$ from Fig. \ref{SetsH7} a)-c) is divisible
by exactly $\widetilde\omega$ and $\widetilde{\omega'}$ for the corresponding belt $\mathcal{B}_4$ with
$\widetilde\omega\cdot\widetilde\omega'=\widetilde{\mathcal{B}_4}$ among all the elements in $H^3(\mathcal{Z}_P)$
corresponding to sets in $N_2(P)$. Therefore, 
the set $\omega\in N_2(P)\setminus N_2^0(P)$ corresponds to a trivial $4$-belt if and only if
the free abelian subgroup $[\widetilde{\omega}]\cdot H^4(\mathcal{Z}_P)\subset\boldsymbol{I}_7$ has rank $(m-5)$,
and the element $\widetilde{\mathcal{B}_4}$ corresponds to a trivial belt if and only if it is divisible by such an element
$[\widetilde{\omega}]\in\boldsymbol{H}_3$. 
\end{proof}

\begin{corollary}\label{4wBcor}
Let $P,Q\in \mathcal{P}_{aPog}\setminus\{I^3,M_5\times I\}$, and let $\mathcal{B}_k$ be a $k$-belt 
passing through a quadrangle $F_i$ and its adjacent faces $F_p$ and $F_q$. 
Then for any isomorphism of graded rings $\varphi\colon H^*(\mathcal{Z}_P)\to H^*(\mathcal{Z}_Q)$ we have
$$
\varphi(\widetilde{\mathcal{B}_k})=\sum_j\mu_j\widetilde{\mathcal{B}_{k,j}'}
$$ 
for $k$-belts $\mathcal{B}_{k,j}'$ of $Q$ such that  for any $\mu_j\ne 0$ the belt $\widetilde{\mathcal{B}_{k,j}'}$ 
passes through the nonadjacent faces $F_{p'}'$ and $F_{q'}'$ of $Q$,
where $\varphi(\widetilde{\{p,q\}}+A_3(P))=\pm \widetilde{\{p',q'\}}+A_3(Q)$. Moreover, $F_{p'}'$ and $F_{q'}'$ are adjacent to 
the quadrangle $F_{i'}'=\varphi_4(F_i)$ of $Q$ (see Notation \ref{f4not}).
\end{corollary}
\begin{proof}
We know from Section \ref{Bsec} that 
$\varphi(\widetilde{\mathcal{B}_k})=\sum_j\mu_j\widetilde{\mathcal{B}_{k,j}'}$ for $k$-belts $\mathcal{B}_{k,j}'$ of $Q$.
\begin{lemma}\label{cosetlemma}
Let $P\in \mathcal{P}_{aPog}\setminus\{I^3,M_5\times I\}$, and let $\mathcal{B}_k$ be a $k$-belt passing through a quadrangle 
$F_i$ and its adjacent faces $F_p$ an $F_q$. Let $x$ be a generator of $\widetilde{H}^1(P_{\tau})=\mathbb Z$
for $\tau=\omega(\mathcal{B}_k)$ (in this case $x=\pm\widetilde{\mathcal{B}_k}$), or $\tau=\omega(\mathcal{B}_k)\sqcup \{r\}$,
where $F_r$ either is not adjacent to faces in $\mathcal{B}_k$, or is adjacent to exactly one face in $\mathcal{B}_k$.
Then $x$ is divisible by any element in the coset $\widetilde{\{p,q\}}+A_3(P)$.
\end{lemma}
\begin{proof}
Indeed, consider the set $P_{\tau\setminus\{p,q\}}$. One of its connected components is $F_i$, since $F_i$ can not be
adjacent to $F_r$ (for otherwise, $F_r$ is also adjacent to both $F_p$ and $F_q$, which is a contradiction). Then
for $[F_i]\in \widehat{H}_2(P_{\tau\setminus\{p,q\}},\partial P_{\tau\setminus\{p,q\}})$ and 
$[F_p]\in \widehat{H}_2(P_{\{p,q\}},\partial P_{\{p,q\}})$ we have $[F_i\cap F_p]$ is a generator of 
$H_1(P_{\tau},\partial P_{\tau})\simeq \widetilde{H}^1(P_{\tau})$.
On the other hand, for any $\{s,t\}\in N_2^0(P)$ at least one of the faces $F_s$ and $F_t$ does
not belong to the belt around $F_i$,
say $F_s$. Then either $F_s=F_i$, or $F_s\cap F_i=\varnothing$. Therefore $\xi\cdot\widetilde{\{s,t\}}=0$, 
where $\xi$ is the element in $H^*(\mathcal{Z}_P)$ corresponding to 
$[F_i]\in \widehat{H}_2(P_{\tau\setminus\{p,q\}},\partial P_{\tau\setminus\{p,q\}})$. Thus, 
$
\pm \xi\cdot(\widetilde{\{p,q\}}+\sum\limits_{\omega\in N_2^0(P)}\lambda_{\omega}\widetilde{\omega})=x.
$
\end{proof}

Then $\varphi(\widetilde{\mathcal{B}_k})$ is divisible by any element in the coset $\pm\widetilde{\{p',q'\}}+A_3(Q)$, 
in particular, by $\widetilde{\{p',q'\}}$. By Corollary \ref{Bdivcor}  each belt $\mathcal{B}_{k,j}'$ passes through $F_{p'}'$ 
and $F_{q'}'$. The second part of the Corollary follows from the fact that $\widetilde{\mathcal{B}_{i'}'}$ is also divisible by any 
element in the coset $\pm\widetilde{\{p',q'\}}+A_3(Q)$ and passes through $F_{p'}'$ and $F_{q'}'$.
\end{proof}

\begin{proposition}\label{BwProp}
Let $P,Q\in \mathcal{P}_{aPog}\setminus\{I^3,M_5\times I\}$, and let $F_p$ be a quadrangle of $P$ not adjacent to a face $F_q$.
Assume that for an isomorphism of graded rings $\varphi\colon H^*(\mathcal{Z}_P)\to H^*(\mathcal{Z}_Q)$ we have  
$\varphi(\widetilde{\{p,q\}})=\pm \widetilde{\{s,t\}}$ for some $\{s,t\}\in N_2(Q)$. Then $p'\in \{s,t\}$ for
$F_{p'}'=\varphi_4(F_p)$ (see Notation \ref{f4not}). In particular, $\varphi(\widetilde{\{p,q\}})=\pm \widetilde{\{p',q'\}}$ for quadrangles $F_p$ and $F_q$, and the set 
$$
\{\pm\widetilde{\{p,q\}}\colon \text{ $F_p$ and $F_q$ are quadrangles}\}\subset H^3(\mathcal{Z}_P)
$$
is $B$-rigid in the class $\mathcal{P}_{aPog}\setminus\{I^3,M_5\times I\}$.
\end{proposition}
\begin{proof}
We start with technical results. 
\begin{lemma}\label{ijaPoglemma}
Let $F_i$ and $F_j$ be two adjacent faces of a polytope $\mathcal{P}_{aPog}\setminus\{I^3,M_5\times I\}$,
and $F_u$, $F_v$ be the faces intersecting $F_i\cap F_j$ by vertices.
Then the loop $\mathcal{L}=(F_u, F_{i_1}, \dots, F_{i_k}, F_v, F_{j_1}, \dots, F_{j_r})$ around $F_i$ and $F_j$,
where $F_{i_l}\cap F_i$ are successive edges of $F_i$, and $F_{j_l}\cap F_j$ are successive edges of $F_j$
is simple. It is a belt if and only if both $F_u$ and $F_v$ are not quadrangles. Moreover, if we delete from  $\mathcal{L}$
the quadrangles among $F_u$ and $F_v$, we obtain a belt. 
\end{lemma}
\begin{proof}
Any face of a flag polytope is surrounded by a belt (see, for example, \cite[Proposition 2.9.2.]{BE17I}). Hence, 
$\mathcal{B}_1=(F_u, F_{i_1}, \dots, F_{i_k}, F_v, F_j)$ and $\mathcal{B}_2=(F_v, F_{j_1}, \dots, F_{j_r}, F_u, F_i)$ are belts. 
Then $\mathcal{L}$ is not simple if and only if $F_{i_p}=F_{j_q}$ for some $p$, $q$. 
But then $(F_i,F_j,F_{i_p})$ is a $3$-belt, which  is a contradiction. Thus, $\mathcal{L}$ is a simple loop. 
If it is not a belt, then two non-successive faces are adjacent.
This can be possible only if they do not belong to $\mathcal{B}_1$ and $\mathcal{B}_2$, that is one face is $F_{i_p}$, 
and the other is $F_{j_q}$. Then $(F_i,F_j,F_{i_p}, F_{j_q})$ is a $4$-belt. It should surround a quadrangle 
adjacent to both $F_i$ and $F_j$. It can be only $F_u$ and $F_v$. If both of these faces are not
quadrangles, then $\mathcal{L}$ is a belt. Else delete these faces to obtain a new simple loop $\mathcal{L}'$. 
If two non-successive faces of this loop are adjacent, they have the form $F_{i_p}$ and $F_{j_q}$. By the previous argument 
these faces should be adjacent to a quadrangle $F_u$ or $F_v$. Then $(i_p,j_q)\in\{(i_k,j_1),(i_1,j_r)\}$ and $F_{i_p}$ and 
$F_{j_q}$ are successive in $\mathcal{L}'$. A contradiction.
\end{proof}
\begin{lemma}\label{ij4lemma}
Let $P\in \mathcal{P}_{aPog}\setminus\{I^3,M_5\times I\}$. 
Then for any three pairwise different faces $\{F_i,F_j,F_k\}$ such that $F_i\cap F_j=\varnothing$, 
$F_k$ is a quadrangle, and at least one of the faces $F_i$ and $F_j$ is not adjacent to
$F_k$, there exists an $l$-belt ($l\geqslant 4$) $\mathcal{B}_l$ such that $F_i,F_j\in \mathcal{B}_l$, 
$F_k\notin \mathcal{B}_l$,  and $\mathcal{B}_l$ does not contain any of the two pairs of opposite faces of 
the $4$-belt around $F_k$.
\end{lemma}
\begin{proof}
Let $(F_{i_1},F_{i_2},F_{i_3},F_{i_4})$ be the $4$-belt around $F_k$, and $F_{j_1}$, $F_{j_2}$, $F_{j_3}$, $F_{j_4}$
be the faces intersecting the edges $F_{i_l}\cap F_{i_{l+1}}$ by vertices different from $F_{i_l}\cap F_{i_{l+1}}\cap F_k$ 
(for convenience we consider $l\in \mathbb Z_4=\mathbb Z/4\mathbb Z$ for $i_l$ and $j_l$). By Lemma \ref{ijaPoglemma}
for any $l$ we have $F_{j_l}\ne F_{j_{l+1}}$  and $F_{j_l}\ne F_{j_{l+2}}$, since  $F_{i_{l+1}}$ and 
$F_{i_{l+2}}$ are not quadrangles. Since there are at most one face among $F_i$ and $F_j$ in the $4$-belt
around $F_k$, there are at least two pairs $(F_{i_l},F_{i_{l+1}})$ not containing both $F_i$ and $F_j$. For them 
there is at least one $F_{j_l}\notin \{F_i,F_j\}$. Thus, we obtain a triple $\{F_{i_l},F_{i_{l+1}},F_{j_l}\}$ not
containing faces in $\{F_i,F_j\}$. By Lemma \ref{ijaPoglemma} there is a belt corresponding to $(F_{i_l},F_{i_{l+1}})$.
It contains $F_{i_{l+2}}$, $F_{i_{l+3}}$, and contains $F_{j_l}$ if and only if it is not a quadrangle. 
Connect the midpoints of edges of intersection of successive faces of this belt to obtain a piecewise linear closed
cycle containing in one component of the complement $F_{i_l}$ and $F_{i_{l+1}}$. Substitute a disk for this component
to obtain a spherical graph. By \cite[Lemma 5.1.1]{BE17I} (also \cite[Lemma 3.11]{BE17S}) this graph corresponds to a flag simple polytope.
By the analog of SCC for flag polytopes (see Section \ref{Bsec}) there is a belt in this polytope, which contains 
the faces arisen from $F_i$ and $F_j$ and does not contain the face corresponding to the disk. In $P$ this belt corresponds
to a belt $\mathcal{B}_l$ we need, since subdivision of the disk does not add or remove
edges of intersection of faces of the belt. The belt $\mathcal{B}_l$ contains $F_i$ and $F_j$, but does not contain $F_{i_l}$ and $F_{i_{l+1}}$. This finishes the proof.
\end{proof}

Now we are ready to proof Proposition \ref{BwProp}.  Let $p'\notin \{s,t\}$. 
The faces $F_s'$ and $F_t'$ are not adjacent simultaneously
to any quadrangle, for otherwise $\{s,t\}\notin N_2^0(Q)$, $\{p,q\}\notin N_2^0(P)$, and $F_p$
is adjacent to a quadrangle. In particular, they are not adjacent simultaneously to $F_{p'}'$.
By Lemma \ref{ij4lemma} there is a belt $\mathcal{B}_l'$ in $Q$, which contains $F_s'$ and $F_t'$,
and does not contain any of the pairs of opposite faces of the belt $\mathcal{B}_{p'}$ around $F_{p'}'$. 
Consider the preimage
$$
\varphi^{-1}(\widetilde{\mathcal{B}_l'})=\sum_s\lambda_s\widetilde{\mathcal{B}_{l,s}}.
$$
By Corollary \ref{Bdivcor} any of the belts with $\lambda_s\ne0$ contains $F_p$ and $F_q$. In particular,
it passes through a pair of opposite faces in the belt $\mathcal{B}_p$ around $F_p$. 
By Corollary \ref{4wBcor}
$$
\widetilde{\mathcal{B}_l'}=\varphi\left(\sum_s\lambda_s\widetilde{\mathcal{B}_{l,s}}\right)=
\sum_s\lambda_s\varphi(\widetilde{\mathcal{B}_{l,s}})=\sum_s\lambda_s\sum_j\mu_{s,j}\widetilde{\mathcal{B}_{l,j}'}=
\sum_j\left(\sum_s\lambda_s\mu_{s,j}\right)\widetilde{\mathcal{B}_{l,j}'}
$$
where each belt $\mathcal{B}_{l,j}'$, such that $\lambda_s\ne0\ne\mu_{s,j}$ for some $s$, 
contains a pair of opposite faces in the belt
$\mathcal{B}_{p'}$. But $\mathcal{B}_l'$ is equal to one of this belts,
hence also contains a pair of opposite faces of the belt $\mathcal{B}_{p'}$ around $F_{p'}'$. A contradiction. Thus, $p'\in \{s,t\}$.

If $F_p$ and $F_q$ are both quadrangles, then Corollary \ref{n4cor} implies
that $\varphi(\widetilde{\{p,q\}})=\pm \widetilde{\{s,t\}}$ for some 
$\{s,t\}\in N_2^0(Q)$. From the above argument $q'\in \{s,t\}$ for $F_{q'}'=\varphi_4(F_q)$. Thus, $\{s,t\}=\{p',q'\}$.
This finishes the proof.

\end{proof}

\begin{lemma}\label{2k4klemma}
Let $P,Q\in \mathcal{P}_{aPog}\setminus\{I^3,M_5\times I\}$, and let $\mathcal{B}_{2k}$ be a $(2k)$-belt of $P$ containing 
$k$ quadrangles. Then for any isomorphism of graded rings $\varphi\colon H^*(\mathcal{Z}_P)\to H^*(\mathcal{Z}_Q)$  
we have  $\varphi(\widetilde{\mathcal{B}_{2k}})=\pm \widetilde{\mathcal{B}_{2k}'}$ for a $(2k)$-belt $\mathcal{B}_{2k}'$ in $Q$
containing $k$ quadrangles. In particular, the set of elements 
$$
\{\pm \widetilde{\mathcal{B}_{2k}}\colon \mathcal{B}_{2k} - \text{a $(2k)$-belt containing $k$ quadrangles}\}\subset H^{2k+2}(\mathcal{Z}_P)
$$ 
is $B$-rigid in the class $\mathcal{P}_{aPog}\setminus\{I^3,M_5\times I\}$.  Moreover, $\varphi_4$ (see Notation \ref{f4not})
induces a bijection between quadrangles of $\mathcal{B}_{2k}$ and quadrangles of $\mathcal{B}_{2k}'$, 
and  for any two faces of $F_p$ and $F_q$ 
of $\mathcal{B}_{2k}$ adjacent to a quadrangle $F_i$ in $\mathcal{B}_{2k}$ we have 
$$
\varphi(\widetilde{\{p,q\}}+A_3(P))=\pm \widetilde{\{p',q'\}}+A_3(Q),
$$ 
where the faces $F_{p'}'$ and $F_{q'}'$ are adjacent to $F_{i'}'=\varphi_4(F_i)$.
\end{lemma}
\begin{proof}
Since $P$ has no adjacent quadrangles, 
$\mathcal{B}_{2k}=(F_{i_1},F_{j_1},F_{i_2},F_{j_2},\dots,F_{i_k},F_{j_k})$, where 
$F_{i_1}$, $\dots$, $F_{i_k}$ are quadrangles, and $F_{j_1}$, $\dots$, $F_{j_k}$ are not.
Also $k\geqslant 3$, since each $4$-belt surrounds a quadrangle and for this 
reason can not contain quadrangles.

We have
$$
\varphi(\widetilde{\mathcal{B}_{2k}})=\sum_s\mu_s\widetilde{\mathcal{B}_{2k,s}'},
$$
where by Corollary \ref{Bdivcor} and Proposition \ref{BwProp} 
for each $\mu_s\ne0$ the belt $\mathcal{B}_{2k,s}'$ contains all the pairs of 
quadrangles $\{F_{i_a'}'=\varphi_4(F_{i_a}),F_{i_b'}'=\varphi_4(F_{i_b})\}$, $a\ne b$. 
There are $\frac{k(k-1)}{2}$ such pairs. The belt $\mathcal{B}_{2k,s}'$
may contain at most this number of pairs of quadrangles, and it is achieved if and only if
$\mathcal{B}_{2k,s}'$ contains exactly $k$ quadrangles, and
$$
\mathcal{B}_{2k,s}'=(F_{u_1}',F_{v_1}',F_{u_2}',F_{v_2}',\dots,F_{u_k}',F_{v_k}')
$$
for some quadrangles $F_{u_1}'$, $\dots$, $F_{u_k}'$. Then  $\varphi_4$ induces a bijection between the sets 
$\{F_{i_1},\dots, F_{i_k}\}$ and $\{F_{u_1}', \dots, F_{u_k}'\}$: $F_{i_a}\to F_{u_{a'}}'=\varphi_4(F_{i_a})$. 
At this moment we do not know if this bijection preserves the cyclic order. 

For each pair $\{j_{a-1},j_a\}$ (where we assume that $a,a-1\in \mathbb Z_k=\mathbb Z/k\mathbb Z$) we have
$\varphi(\widetilde{\{j_{a-1},j_a\}}+A_3(P))=\pm \widetilde{\{z,w\}}+A_3(Q)$ for faces $F_z$ and $F_w$ adjacent to
$F_{u_{a'}}$. Moreover,  by Corollary \ref{4wBcor} the belt $\mathcal{B}_{2k,s}'$ passes through $F_z$ and $F_w$. 
Thus, $\{z,w\}=\{v_{a'-1},v_{a'}\}$.  Since the belt $\mathcal{B}_{2k,s}'$ is covered by segments $(F_{v_{a'-1}}',F_{u_{a'}}',F_{v_{a'}}')$, there is only one belt $\mathcal{B}_{2k,s}'$ with $\mu_s\ne0$, and $\varphi(\widetilde{\mathcal{B}_{2k}})=\pm\widetilde{\mathcal{B}_{2k,s}'}$. 
\end{proof}

\begin{lemma}\label{2k4ktrivlemma}
Let $P,Q\in \mathcal{P}_{aPog}\setminus\{I^3,M_5\times I\}$, and let $\mathcal{B}_{2k}$ be a trivial $(2k)$-belt of $P$ containing 
$k$ quadrangles. Then for any isomorphism of graded rings $\varphi\colon H^*(\mathcal{Z}_P)\to H^*(\mathcal{Z}_Q)$  
we have  $\varphi(\widetilde{\mathcal{B}_{2k}})=\pm \widetilde{\mathcal{B}_{2k}'}$ for a trivial $(2k)$-belt 
$\mathcal{B}_{2k}'$ in $Q$ containing $k$ quadrangles. In particular, the set of elements 
$$
\{\pm \widetilde{\mathcal{B}_{2k}}\colon \mathcal{B}_{2k} - \text{a trivial $(2k)$-belt containing $k$ quadrangles}\}\subset H^{2k+2}(\mathcal{Z}_P)
$$ 
is $B$-rigid in the class $\mathcal{P}_{aPog}\setminus\{I^3,M_5\times I\}$.
\end{lemma}
\begin{proof}
As before, we see that $k\geqslant 3$, since each $4$-belt surrounds a quadrangle and for this 
reason can not contain quadrangles.  

By Lemma \ref{2k4ktrivlemma} we have $\varphi(\widetilde{\mathcal{B}_{2k}})=\pm \widetilde{\mathcal{B}_{2k}'}$ for a 
$(2k)$-belt  $\mathcal{B}_{2k}'$ in $Q$ containing $k$ quadrangles. We will use notations from the proof of this lemma.

Assume that the belt $\mathcal{B}_{2k}$ surrounds a face $F_i$, and the belt $\mathcal{B}_{2k}'$ is not trivial.

Consider a subgroup  $G(\mathcal{B}_{2k})\subset H^{2k+3}(\mathcal{Z}_P)$ 
consisting of all the elements divisible by $\widetilde{\{i_a,i_b\}}$ for all $a\ne b$, and by each member of a coset
$\widetilde{\{j_{a-1},j_a\}}+A_3(P)$ for all $a$. We have $2k+3\leqslant m$, since outside the belt there are at least two faces, for otherwise  $P$ should be a prism and have adjacent quadrangles. 

Since $\omega(\mathcal{B}_{2k})=\bigcup_{a\ne b}\{i_a,i_b\}\bigcup_a\{j_{a-1},j_a\}$ Lemma \ref{Hdivlemma} implies that 
any element in $G(\mathcal{B}_{2k})$ has the form
$$
\sum\limits_{\tau\supset\omega(\mathcal{B}_l)}\beta_\tau,\text{ where }\beta_\tau\in \widetilde{H}^1(P_{\tau}), |\tau|=l+1,
$$
and each nonzero $\beta_\tau$ is divisible by all $\widetilde{\{i_a,i_b\}}$ and  $\widetilde{\{j_{a-1},j_a\}}$.
Then $\tau=\omega(\mathcal{B}_{2k})\sqcup\{r\}$. We have $F_r\ne F_i$, for otherwise $P_{\tau}$ is contractible. 
Let $F_r$ be adjacent to two non-successive faces of $\mathcal{B}_{2k}$, say $F_p$ and $F_q$.
Then $(F_r,F_p,F_i,F_q)$ is a $4$-belt surrounding a quadrangle $F_{i_a}$ in $\mathcal{B}_{2k}$, and 
$\{F_p,F_q\}= \{F_{j_{a-1}}, F_{j_a}\}$. There are $k$ such quadrangles.
If $F_r$ is adjacent to some face in $\mathcal{B}_{2k}$ different from $F_{j_{a-1}}$, $F_{i_a}$, and $F_{j_a}$, then 
this face is not successive  with $F_{i_a}$, and by the previous argument $F_{i_a}$ belongs
a $4$-belt around a quadrangle, which is a contradiction. Thus, $F_r$ is not adjacent to other faces in
$\mathcal{B}_{2k}$. There are $k$ faces of this type. For them 
$P_{\tau\setminus\{i_b\}}$ is connected for any $i_b\ne i_a$, hence $\beta_\tau$ is not divisible by
$\widetilde{\{i_a,i_b\}}$. If $F_r$ is adjacent to two successive faces of $\mathcal{B}_{2k}$, then one of this faces is
a quadrangle, and $F_r$ is of previous type. If $F_r$ is adjacent to exactly one face of $\mathcal{B}_{2k}$, then this
face is not a quadrangle, and $P_{\tau}$ is homeomorphic to $|\mathcal{B}_{2k}|$. If $F_r$ is not adjacent to faces in 
$\mathcal{B}_{2k}$, then $P_{\tau}$ is a disjoint union of $|\mathcal{B}_{2k}|$ and $F_r$. In both cases 
$\widetilde{H}^1(P_{\tau})=\mathbb Z$, and its generator is divisible by all the elements 
$\widetilde{\{i_a,i_b\}}$ and  $\widetilde{\{j_{a-1},j_a\}}$. Moreover, by Lemma \ref{cosetlemma} this generator is also
divisible by any element in the coset $\widetilde{\{j_{a-1},j_a\}}+A_3(P)$. 
Thus, $G(\mathcal{B}_{2k})$ is a free abelian group of rank $m-2k-1-k=m-3k-1$.

Now consider the belt $\mathcal{B}_{2k}'$. Since $\varphi(\widetilde{\{i_a,i_b\}})=\pm\widetilde{\{u_{a'},u_{b'}\}}$ 
for all $a\ne b$, and 
$\varphi(\widetilde{\{j_{a-1},j_a\}}+A_3(P))=\pm \widetilde{\{v_{a'-1},v_{a'}\}}+A_3(Q)$ for all $a$, we have 
$\varphi(G(\mathcal{B}_{2k}))=G(\mathcal{B}_{2k}')$. 
By the above argument each element in $G(\mathcal{B}_{2k}')$ has the form 
$$
\sum\limits_{\tau\supset\omega(\mathcal{B}_{2k}')}\beta_\tau',\text{ where }\beta_{\tau}'\in \widetilde{H}^1(Q_{\tau}), |\tau|=l+1,
$$
and each nonzero $\beta_{\tau}'$ is divisible by all $\widetilde{\{u_a,u_b\}}$ and  $\widetilde{\{v_{a-1},v_a\}}$. Also
$\tau=\omega(\mathcal{B}_{2k}')\sqcup r$. Consider a face $F_r'$ of $Q$. Let $F_r'$ be adjacent to two non-successive 
faces of $\mathcal{B}_{2k}'$, say $F_p'$ and $F_q'$. Since $\mathcal{B}_{2k}'$ is a nontrivial belt and
$Q$ is a flag polytope, $F_r'$ is not adjacent to some
face of $\mathcal{B}_{2k}'$. We can assume that for the segment $T=(F_p',F_{w_1}',\dots,F_{w_t}',F_q')$ of $\mathcal{B}_{2k}'$
the face $F_r'$ is not adjacent to  $F_{w_1}'$, $\dots$, $F_{w_t}'$. Then $F_p'$ and $F_q'$ are not quadrangles, 
for otherwise either $F_{w_1}'$ or $F_{w_t}'$ is adjacent to $F_r'$. Hence,  
$\mathcal{B}_{t+3}'=(F_r',F_p',F_{w_1}',\dots,F_{w_t}',F_q')$
is a $(t+3)$-belt, where $4\leqslant t+3\leqslant 2k$, and $t+3=2k$ only if $F_p'$ and $F_q'$ are faces adjacent to one face 
$F_v'$ in $\mathcal{B}_{2k}'$. This face $F_v'$ is a quadrangle, since $F_p'$ and $F_q'$ are not quadrangles.
Then $F_v'\cap F_r'\ne\varnothing$, for otherwise $(F_v',F_p',F_r', F_q')$ is a $4$-belt around a quadrangle adjacent to 
the quadrangle $F_v'$. Thus, if $t+3=2k$, then $F_r'$ is adjacent only to the three faces $F_{v_{a-1}}'$, $F_{u_a}'$, 
and $F_{v_a}'$ in $\mathcal{B}_{2k}'$ for some $a$.

Now assume that $t+3<2k$. If $t+3=4$, then $\mathcal{B}_{t+3}'$ surrounds a quadrangle adjacent to three successive faces of 
$\mathcal{B}_{2k}'$, which is a contradiction. Thus, $t+3\geqslant 5$. Moreover, $t+3\geqslant 6$ is an even number, 
for otherwise either $F_p'$ or $F_q'$ is a quadrangle. In particular, there are at least two quadrangles in $T$. Consider the preimage 
$$
\varphi^{-1}(\widetilde{\mathcal{B}_{t+3}'})=\sum_s\lambda_s \widetilde{\mathcal{B}_{t+3,s}}
$$
Each belt $\mathcal{B}_{t+3,s}$ is divisible by any $\widetilde{\{i_a,i_b\}}$ for 
$\{u_{a'},u_{b'}\}\subset \omega(\mathcal{B}_{t+3}')$. Corollary \ref{4wBcor}  implies that it is also divisible 
by any $\widetilde{\{j_{a-1},j_{a}\}}$ with  $\{v_{a'-1},v_{a'}\}\subset \omega(\mathcal{B}_{t+3}')$ 
(since in this case $F_{u_{a'}}\in T$). In particular,
it contains $F_{j_{a-1}}$, $F_{i_a}$, $F_{j_a}$, $F_{j_{b-1}}$, $F_{i_b}$, and $F_{j_b}$.
The mapping $\varphi_4^{-1}$ sends the set of quadrangles in $T$ injectively to some set of quadrangles in  
$\mathcal{B}_{t+3,s}\cap \mathcal{B}_{2k}$. Moreover, the image of each quadrangle in contained in 
$\mathcal{B}_{t+3,s}\cap \mathcal{B}_{2k}$ with two adjacent faces. Denote by 
$M\subset \mathcal{B}_{t+3,s}\cap \mathcal{B}_{2k}$ the union of the images of quadrangles and the faces in 
$\mathcal{B}_{t+3,s}\cap \mathcal{B}_{2k}$ adjacent to this images.  
The set $M$ is a disjoint union of segments of $\mathcal{B}_{2k}$. Enumerate them in the order they appear in the belt. If 
we identify the last face of each segment with the first face of the next segment (both of them are not quadrangles), 
we obtain one segment consisting of $|T|=t+2$ faces. If $M$ consists of more than one connected component, then
it has more than $t+2$ faces. But  $M$ has at most $t+3$ faces, since 
$M\subset \mathcal{B}_{t+3,s}$. Therefore,  $M=\mathcal{B}_{t+3,s}\subset \mathcal{B}_{2k}$.
This is impossible, since $t+3<2k$. Thus, $M$ is a segment of $\mathcal{B}_{2k}$ consisting of $t+2$ faces.
Then the $(t+3)$-th face in $\mathcal{B}_{t+3,s}$ is some face $F_x\notin \mathcal{B}_{2k}$.
Since $\mathcal{B}_{t+3,s}$ is a belt, $F_x$ is different from $F_i$. 
Let $F_{j_a}$ and $F_{j_b}$ be the first and the last faces of $M$ (we know that these faces are not quadrangles, 
in particular, they are not adjacent). Then $(F_{j_a},F_i,F_{j_b},F_x)$ is a $4$-belt surrounding some quadrangle 
$F_{i_c}\in\mathcal{B}_{2k}$.  Then $\{F_{j_a},F_{j_b}\}=\{F_{j_{c-1}},F_{j_c}\}$, and $F_x$ is adjacent to $F_{i_c}$.  
Therefore $M=\mathcal{B}_{2k}\setminus\{F_{i_c}\}$, and  $t+2=2k-1$, which is a contradiction.
Hence, the case $t+3<2k$ is impossible.  

Thus, if the face $F_r'$ is adjacent in $\mathcal{B}_{2k'}$
to non-successive faces, then it is adjacent exactly to some quadrangle $F_{u_a}'$ 
and its adjacent faces $F_{v_{a-1}}',F_{v_{a}}'$. There are $2k$ faces of this type. If $F_r'$ 
is adjacent to two successive faces, then one
of them is a quadrangle, and we come to the previous case. Also $F_r'$ can be adjacent to exactly one face $F_{v_a}'$,
which is not a quadrangle, or to no faces in $\mathcal{B}_{2k}'$.
In each of these cases $\widetilde{H}^1(Q_{\tau})=\mathbb Z$. In the first case the generator is not divisible by 
$\widetilde{\{u_a,u_b\}}$ for any $b\ne a$, since $Q_{\tau\setminus\{u_a,u_b\}}$ is connected. In 
the second and the third cases the generator is divisible  by all the elements $\widetilde{\{u_a,u_b\}}$ and  
$\widetilde{\{v_{a-1},v_a\}}$. Moreover, by Lemma \ref{cosetlemma} it is also
divisible by any element in the coset $\widetilde{\{v_{a-1},v_a\}}+A_3(Q)$. 
Thus, $G(\mathcal{B}_{2k}')$ is a free abelian group of rank $m-2k-2k=m-4k$.
Since $m-4k<m-3k-1$ for $k>1$ we obtain a contradiction. Thus,  $\mathcal{B}_{2k}'$ is a trivial belt.
This finishes the proof.
\end{proof}

\begin{lemma}\label{42klemma}
Let $P\in \mathcal{P}_{aPog}\setminus\{I^3,M_5\times I\}$, $F_i$ be its quadrangle, and let 
$\mathcal{B}_{2k}$ be the $(2k)$-belt containing $k$ quadrangles and surrounding a face $F_j$. 
Then $F_i$ and $F_j$ are adjacent if and only if the element $\widetilde{\mathcal{B}_{2k}}$ is divisible by any 
element in the coset $\widetilde{\{p,q\}}+A_3(P)$ for one of the two pairs of opposite  faces $\{F_p, F_q\}$ of the $4$-belt
$\mathcal{B}_i$ around $F_i$. In particular,  if $F_i$ and $F_j$ are adjacent, then for any isomorphism of graded rings 
$\varphi\colon H^*(\mathcal{Z}_P)\to H^*(\mathcal{Z}_Q)$  for a polytope $Q\in \mathcal{P}_{aPog}\setminus\{I^3,M_5\times I\}$ 
we have  $\varphi(\widetilde{\mathcal{B}_{2k}})=\pm \widetilde{\mathcal{B}_{2k}'}$ for a $(2k)$-belt 
$\mathcal{B}_{2k}'$ containing $k$ quadrangles and surrounding a face $F_{j'}'$ adjacent to 
$F_{i'}'=\varphi_4(F_i)$.
\end{lemma}
\begin{proof}
If $F_j$ is adjacent to $F_i$, then Lemma \ref{cosetlemma} implies
that $\widetilde{\mathcal{B}_2}$ is divisible by any element in the 
coset $\widetilde{\{p,q\}}+A_3(P)$ for faces $F_p$, $F_q$ adjacent to $F_i$ and lying in $\mathcal{B}_{2k}$. 
On the other hand, if the element $\widetilde{\mathcal{B}_{2k}}$ is divisible by any element in the 
coset $\widetilde{\{p,q\}}+A_3(P)$ for faces $F_p$, $F_q$ adjacent to $F_i$, then it is divisible by $\widetilde{\{p,q\}}$. 
Hence, $\mathcal{B}_{2k}$ contains $F_p$ and $F_q$. If it does not contain $F_i$, then $(F_i,F_p,F_j,F_q)$ is a $4$-belt
around some quadrangle adjacent to $F_i$. A contradiction. 

By Lemma \ref{2k4ktrivlemma} we have $\varphi(\widetilde{\mathcal{B}_{2k}})=\pm\widetilde{\mathcal{B}_{2k}'}$
for a belt  $\mathcal{B}_{2k}'$ containing $k$ quadrangles and surrounding a unique face $F_{j'}'$.
The element $\widetilde{\mathcal{B}_{2k}'}$ is divisible by any element in the 
coset $\widetilde{\{p',q'\}}+A_3(Q)$ for faces $F_{p'}'$, $F_{q'}'$ adjacent to $F_{i'}'$. Hence, 
the face $F_{j'}'$ is adjacent to $F_{i'}'$.
\end{proof}

\begin{lemma}\label{2k2klemma}
Let $P\in \mathcal{P}_{aPog}\setminus\{I^3,M_5\times I\}$, and let 
$\mathcal{B}_{2k}$ be the $(2k)$-belt containing $k$ quadrangles and surrounding a face $F_i$, and
$\mathcal{B}_{2l}$ be the $(2l)$-belt containing $l$ quadrangles and surrounding a face $F_j$. 
Then $F_i$ and $F_j$ are adjacent if and only if the elements $\widetilde{\mathcal{B}_{2k}}$  and
$\widetilde{\mathcal{B}_{2l}}$ have exactly one common divisor among elements
$\widetilde{\{p,q\}}$ corresponding to pairs of quadrangles. 
In particular,  if $F_i$ and $F_j$ are adjacent, then for any isomorphism of graded rings 
$\varphi\colon H^*(\mathcal{Z}_P)\to H^*(\mathcal{Z}_Q)$  for a polytope $Q\in \mathcal{P}_{aPog}\setminus\{I^3,M_5\times I\}$ 
we have  $\varphi(\widetilde{\mathcal{B}_{2k}})=\pm \widetilde{\mathcal{B}_{2k}'}$ for a $(2k)$-belt 
$\mathcal{B}_{2k}'$ containing $k$ quadrangles and surrounding a face $F_{i'}'$,  and 
$\varphi(\widetilde{\mathcal{B}_{2l}})=\pm \widetilde{\mathcal{B}_{2l}'}$ for a $(2l)$-belt 
$\mathcal{B}_{2l}'$ containing $l$ quadrangles and surrounding a face $F_{j'}'$ 
adjacent to $F_{i'}'$.
\end{lemma}
\begin{proof}
We have $k,l\geqslant3$, since a $4$-belt surrounds a quadrangle and can not contain quadrangles.
If $F_i$ and $F_j$ are adjacent, then, since $P$ is flag, 
$\mathcal{B}_1\cap \mathcal{B}_2$ consists of two non-adjacent faces intersecting the edge $F_i\cap F_j$ by vertices, 
and both of them are quadrangles.  On the other hand, let $F_i\cap F_j=\varnothing$, 
and  $\{u,v\}\subset\omega(\mathcal{B}_1)\cap\omega(\mathcal{B}_2)$ with $F_u\cap F_v=\varnothing$. 
Then $(F_u,F_i,F_v,F_j)$ is a $4$-belt around a quadrangle. If $F_u$ and $F_v$ are quadrangles, 
we have a contradiction. If $F_i=F_j$, then the elements $\widetilde{\mathcal{B}_{2k}}$ and $\widetilde{\mathcal{B}_{2l}}$
have $\frac{k(k-1)}{2}\geqslant 3$ common divisors corresponding to pairs of quadrangles.
The second part of the lemma follows from Lemma \ref{2k4ktrivlemma}.
\end{proof}

\begin{remark}\label{BIR-rem}
Lemmas \ref{2k4ktrivlemma}, \ref{42klemma}, and \ref{2k2klemma} imply that any ideal almost Pogorelov polytope is $B$-rigid.
The details see in \cite{E20b}.
\end{remark}

\section{Cohomology ring $H^*(\mathcal{Z}_{As^3})$}\label{Exas}
\begin{figure}
\begin{center}
\begin{tabular}{ccc}
\includegraphics[height=5cm]{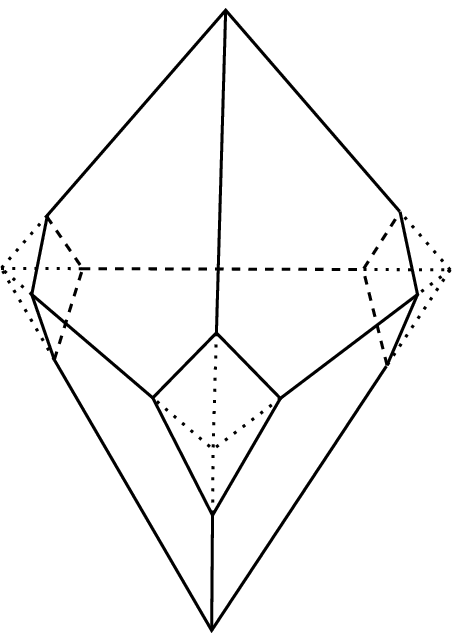}&\includegraphics[height=5cm]{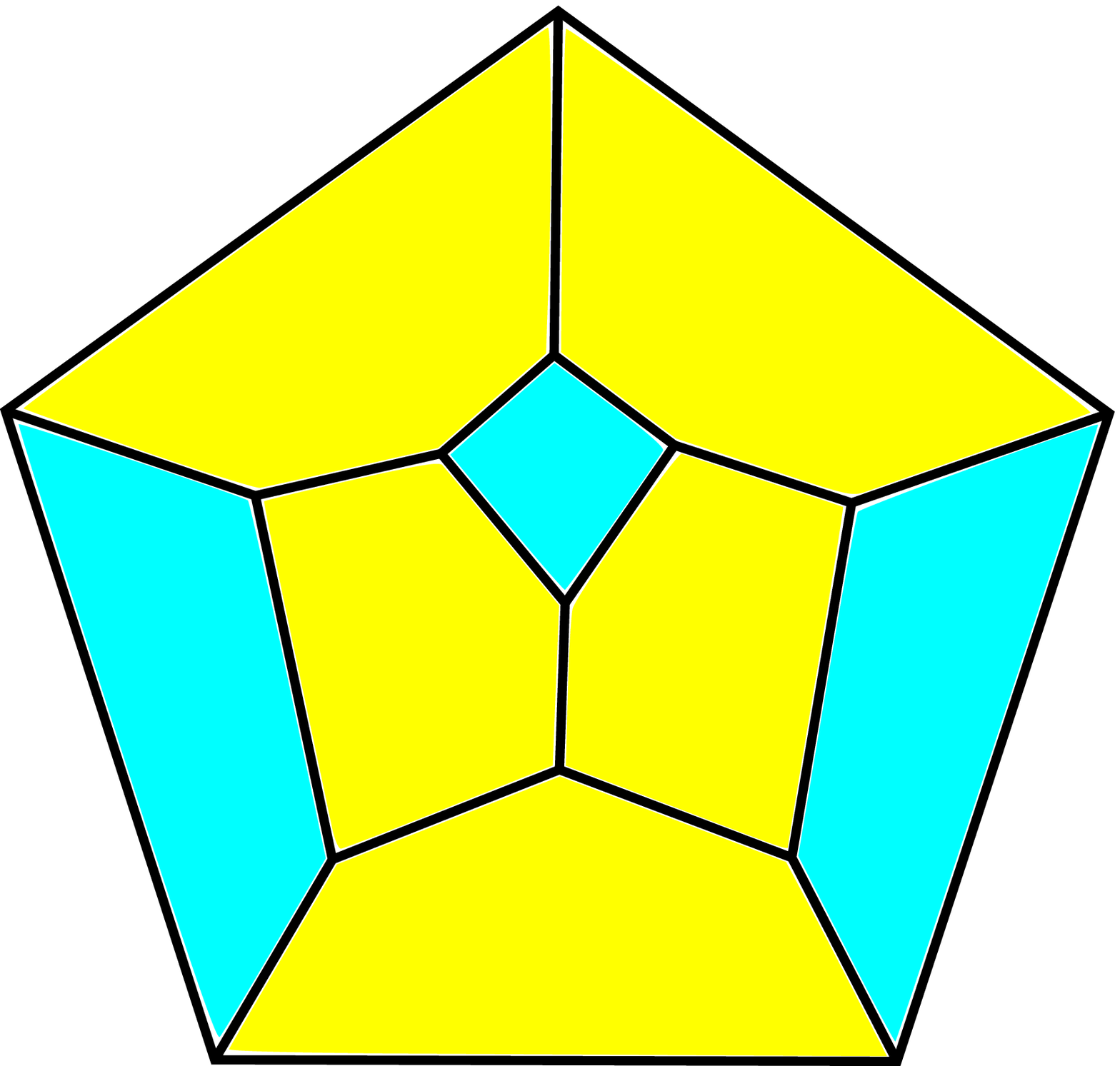}&\includegraphics[height=5cm]{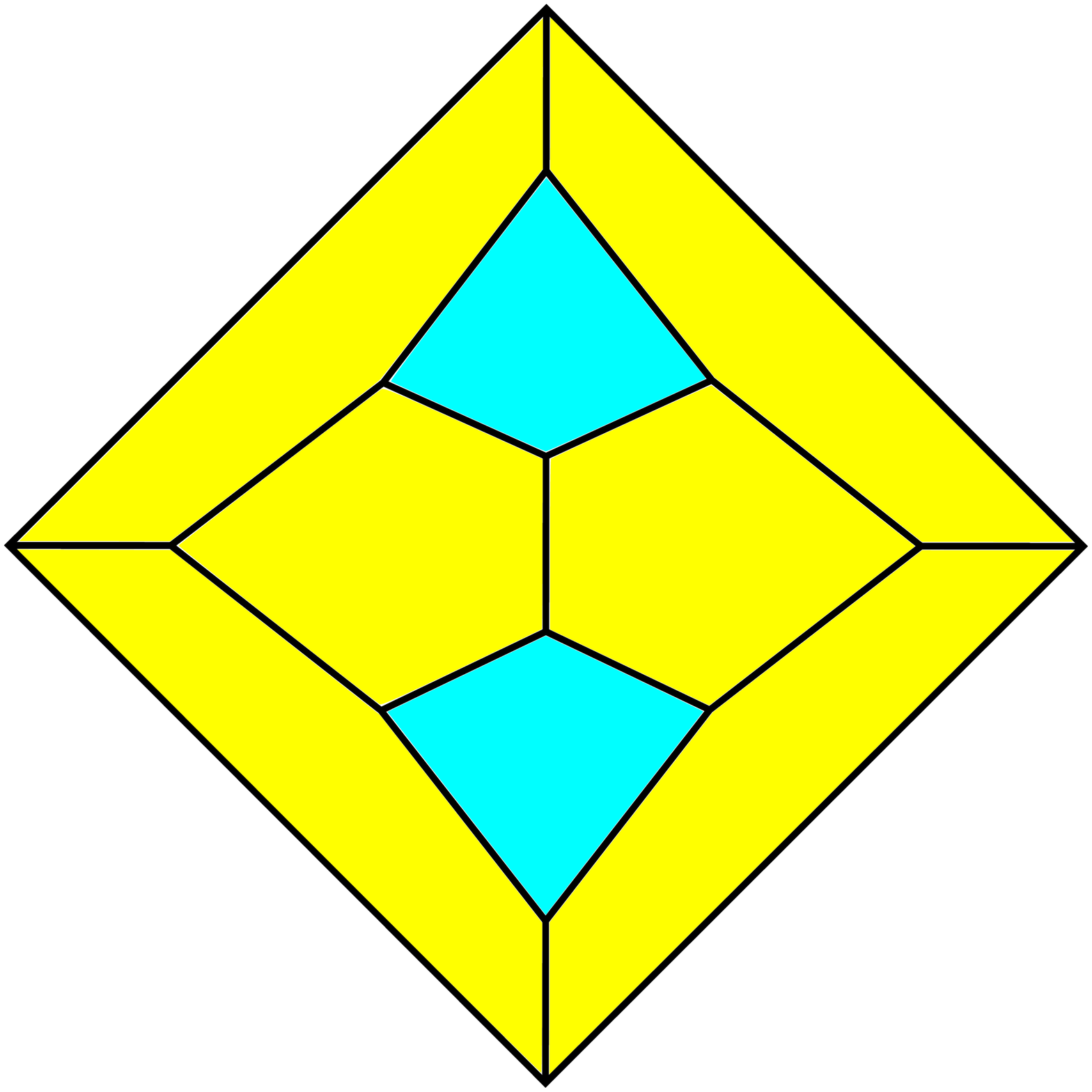}\\
a)&b)&c)
\end{tabular}
\caption{a) a $3$-dimensional associahedron  (Stasheff polytope) $As^3$ as a truncated bipyramid; b) its Schlegel diagram based at a pentagon; c) its Schlegel diagram based at a quadrangle}\label{As-sh}
\end{center}
\end{figure}

For the associahedron $As^3$ we have the following picture.
The group of symmetries of $As^3$ acts transitively on quadrangles ($3$ quadrangles) and pentagons ($6$ pentagons).

There are three equivalence classes of edges:
\begin{enumerate}
\item a common edge of a quadrangle and a pentagon ($12$ edges);
\item a common edge of pentagons connecting two vertices of quadrangles ($3$ edges);
\item a common edge of pentagons containing a common vertex of $3$ pentagons ($6$ edges).
\end{enumerate}
There are two equivalence classed of vertices:
\begin{enumerate}
\item vertices of quadrangles ($12$ vertices);
\item common vertices of $3$ pentagons ($2$ vertices).
\end{enumerate}

We have the following subsets $P_{\omega}$ of faces with $\widetilde{H}^0(P_{\omega})\ne0$ (see Fig. \ref{Cycles}). The complementary subsets $P_{[m]\setminus\omega}$ are exactly subsets with $\widetilde{H}^1(P_{[m]\setminus\omega})\ne0$. It should be $|\omega|\geqslant 2$.
\begin{enumerate}
\item For $|\omega|=2$ the sets $\omega=\{i,j\}$ correspond to pairs of nonadjacent faces $F_i$, $F_j$, 
$F_i\cap F_j=\varnothing$.  We have $3$ types of pairs encoded by numbers $k_1$, $k_2$, of edges of faces: $4-4$ 
($3$ pairs, all equivalent); $4-5$ ($6$ pairs, all equivalent); $5-5$ ($6$ pairs, all equivalent); In total we have 
$\mathbb Z^{15}=H^3(\mathcal{Z}_P)$. By the Poincare duality $H^9(\mathcal{Z}_P)\simeq \mathbb Z^{15}$. 
The generators correspond to generators in $\widetilde{H}^1(P_{[m]\setminus\omega})$ for the complementary subsets.
\item For $|\omega|=3$ the sets $\omega=\{i,j,k\}$ correspond to either triples of disjoint faces or to pairs of disjoint sets: 
a face  and a pair of adjacent faces; For the first type we have only one set $4-4-4$. This corresponds to 
$\mathbb Z^2\subset H^4(\mathcal{Z}_P)$. For the second type we have cases: $4-45$ 
(for each quadrangle $4$ pairs $45$, in total $12$ pairs, all equivalent); 
$4-55$ (for each quadrangle one pair, in total $3$ pairs, all equivalent); 
$5-45$ (for each pentagon two pairs, in total $12$ pairs, all equivalent); 
$5-55$ (for each pentagon one pair, in total $6$ pairs, all equivalent). 
This corresponds to $\mathbb Z^{33}\subset H^4(\mathcal{Z}_P)$. 
In total, since no two faces can give nonzero contribution to $H^4(\mathcal{Z}_P)$, we have 
$H^4(\mathcal{Z}_P)=\mathbb Z^2\oplus\mathbb Z^{33}=\mathbb Z^{35}$.  
By the Poincare duality $H^8(\mathcal{Z}_P)\simeq \mathbb Z^2\oplus\mathbb Z^{33}=\mathbb Z^{35}$. 
Here the summand $\mathbb Z^2$ corresponds to the complement to three disjoint quadrangles, 
which is a sphere with two holes. The generators in $\mathbb Z^{33}$ correspond to generators in 
$\widetilde{H}^1(P_{[m]\setminus\omega})$ for the complementary subsets. 
\item Let $|\omega|=4$. Four can be expressed as a sum of positive integers in the following ways: 
$4=1+1+1+1$; $4=1+1+1+2$; $4=1+3$; $4=2+2$. First let us mention that if one of the connected components of 
$P_{\omega}$ consists of one face, then either there are only two components, or there are three components 
each consisting of a single quadrangle (it can be extracted from Fig. \ref{As-sh}).
Thus, we can have only $1+3$ or $2+2$. For the first case we may have 
$4-445$ (two for each quadrangle, all equivalent, in total $6$), 
$4-455$ (two for each quadrangle, all equivalent, in total $6$), or 
$5-455$ (one for each pentagon, all equivalent, in total $6$). 
For the second case we may have only $45-45$ (all equivalent, in total $6$). 
Since $P$ has no $3$-belts, we have $H^5(\mathcal{Z}_P)\simeq \mathbb Z^{24}$. 
By the Poincare duality $H^7(\mathcal{Z}_P)\simeq \mathbb Z^{24}$. 
The generators correspond to generators in $\widetilde{H}^1(P_{[m]\setminus\omega})$ for the complementary subsets.
\item For $\omega=5$ we have $5=1+4=2+3$. For the first case we have 
$4-4455$ (one for each quadrangle, all equivalent, in total $3$). 
The case $2+3$ is impossible, since for any pair of adjacent faces the set of faces not adjacent 
to both of them has cardinality at most $2$.  We have $\mathbb Z^3\subset H^6(\mathcal{Z}_P)$. 
The complementary subsets are $4$-belts and correspond to generators in $\widetilde{H}^1(P_{[m]\setminus\omega})$, 
and we obtain a complementary summand $\mathbb Z^3\subset H^6(\mathcal{Z}_P)$ such that 
$H^6(\mathcal{Z}_P)=\mathbb Z^3\oplus\mathbb Z^3$. 
\end{enumerate}

\begin{figure}
\begin{center}
\includegraphics[height=10cm]{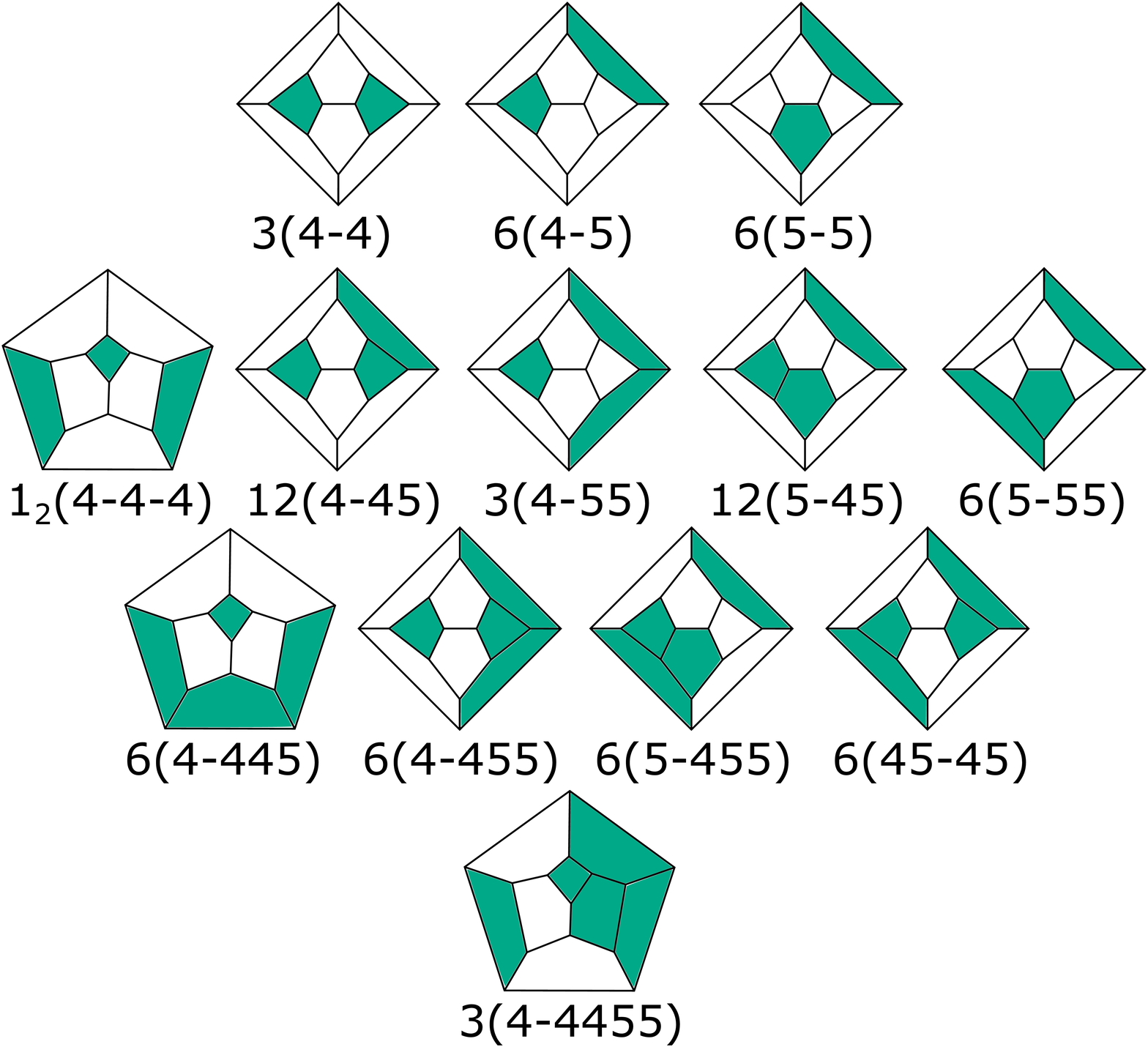}
\caption{Collections $P_{\omega}$ giving classes in $\widetilde{H}^0(P_{\omega})$}\label{Cycles}
\end{center}
\end{figure}

Additive structure of $H^*(\mathcal{Z}_P)$ is represented in Table \ref{As-T}.
\begin{table}
\setlength{\tabcolsep}{4pt}
\centering
\caption{Additive structure of $H^*(\mathcal{Z}_{As^3})$}\label{As-T}
\label{tab1}
\begin{tabular}{|c|c|c|c|c|c|c|c|c|c|c|c|c|c|}
\hline
$k$&0&1&2&3&4&5&6&7&8&9&10&11&12\\
\hline
$H^k(\mathcal{Z}_P)$&$\mathbb Z$&0&0&$\mathbb Z^{15}$&$\mathbb Z^{35}$&$\mathbb Z^{24}$&$\mathbb Z^6$&
$\mathbb Z^{24}$&$\mathbb Z^{35}$&$\mathbb Z^{15}$&0&0&$\mathbb Z$\\
\hline
\end{tabular}
\end{table}

Choose in each group $\widetilde{H}^0(P_{\omega})=\mathbb Z$ a generator $\beta_{\omega}$. We will call these generators 
{\it canonical generators}.  There is only one set $P_{\omega}$ with $\rk \widetilde{H}^0(P_{\omega})>1$. This is the union of three
quadrangles with $\widetilde{H}^0(P_{\omega})=\mathbb Z^2$. 
Choose in this group two generators $\beta_{\omega,1}$ and $\beta_{\omega,2}$ corresponding to two quadrangles. 
We will also call these generators {\it additional}. Canonical and additional generators together form the set of 
multiplicative generators $\{\beta\}$ of $H^*(\mathcal{Z}_{As^3})$, since $As^3$ is a flag polytope. For short
we will call them {\it generators}. 

To describe the multiplication in $H^*(\mathcal{Z}_P)$ we need to describe how the product of each two
generators can be expressed in terms of the elements $\{\beta^*\}$ dual to generators, where $\beta\cdot \beta^*=[\mathcal{Z}_P]$
-- a fundamental class in cohomology, and $\beta\cdot (\beta')^*=0$ for $\beta\ne\beta'$. 

For each set $P_{\omega}$ with $\widetilde{H}^0(P_{\omega})\ne 0$ on Fig. \ref{Cycles-D-mono} we describe all the ways how
$P_{[m]\setminus\omega}$ up to combinatorial symmetries 
can be represented as a union of $P_{\omega_1}\cup P_{\omega_2}$ with 
$[m]\setminus\omega=\omega_1\sqcup\omega_2$ and $\widetilde{H}^0(P_{\omega_1})\ne 0$, 
$\widetilde{H}^0(P_{\omega_2})\ne 0$. Also we give the number of combinatorially symmetric decompositions. 

\begin{figure}
\begin{center}
\includegraphics[height=12cm]{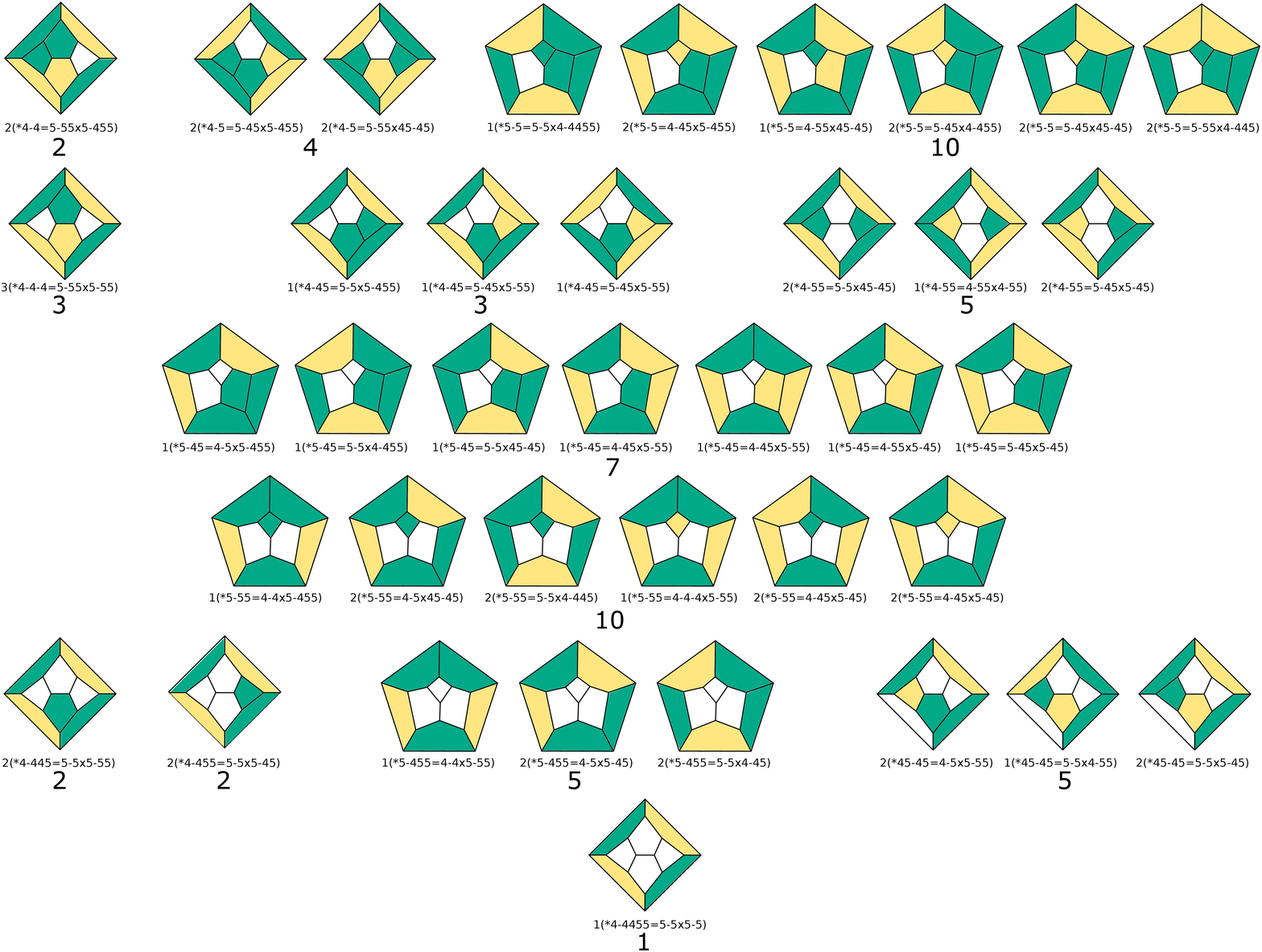}
\caption{Decompositions of the sets $P_{[m]\setminus\omega}$ with $\widetilde{H}^0(P_{\omega})\ne 0$}\label{Cycles-D-mono}
\end{center}
\end{figure}

On Fig. \ref{Cycles-D-Rel} we describe all the ways 
how the boundary $\partial P$ of the polytope $P$ up to combinatorial symmetries can be decomposed 
as a union of three sets $P_{\omega_1}$, $P_{\omega_2}$, and $P_{\omega_3}$ with 
$[m]=\omega_1\sqcup\omega_2\sqcup\omega_3$, and $\widetilde{H}^0(P_{\omega_i})\ne 0$ for $i=1,2,3$.
We present the number of combinatorially symmetric decompositions and for each set we give
the number of decompositions where it arises.

\begin{figure}
\begin{center}
\includegraphics[height=11cm]{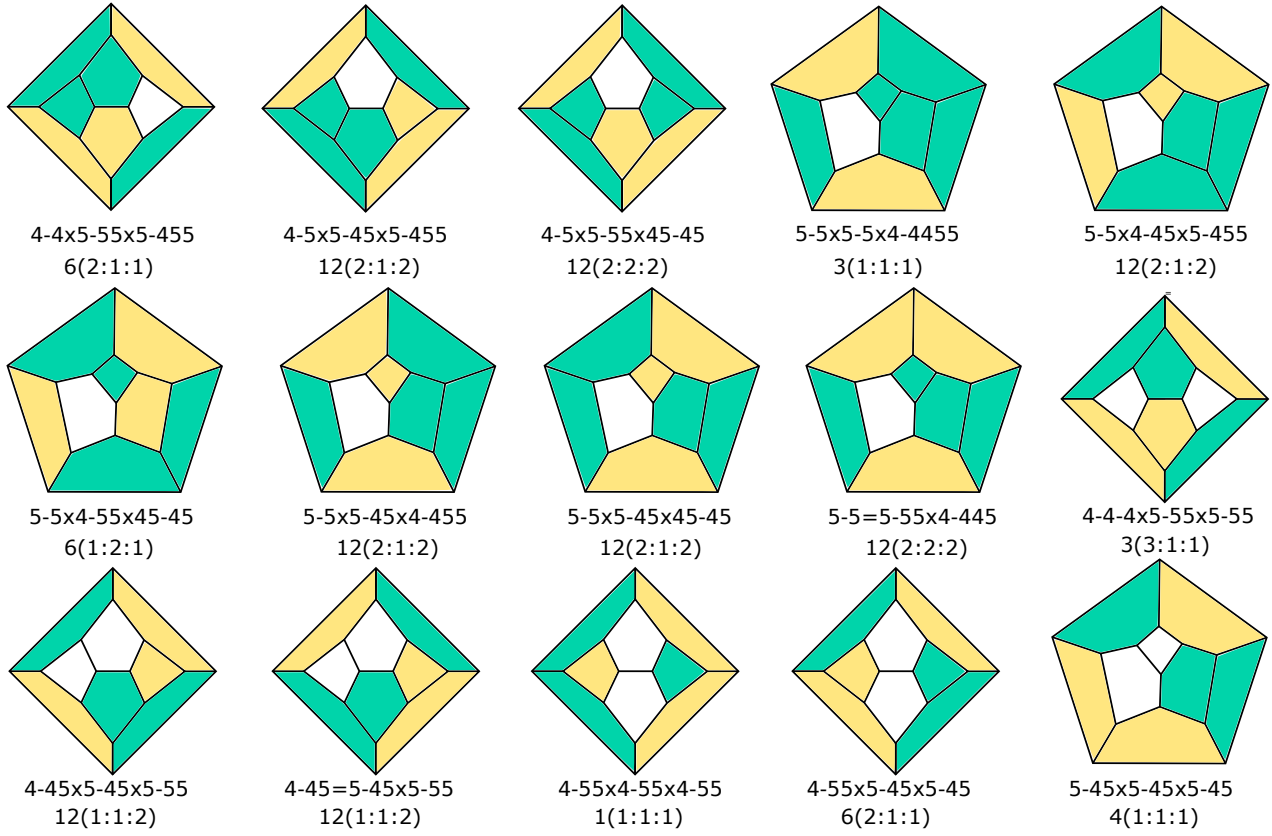}
\caption{Decompositions $\partial P=P_{\omega_1}\cup P_{\omega_2}\cup P_{\omega_3}$ with 
$[m]=\omega_1\sqcup\omega_2\sqcup\omega_3$, and $\widetilde{H}^0(P_{\omega_i})\ne 0$ for $i=1,2,3$}\label{Cycles-D-Rel}
\end{center}
\end{figure}

For generators $\beta_1\in \widetilde{H}^0(P_{\omega_1})$, $\beta_2\in \widetilde{H}^0(P_{\omega_2})$ their 
product $\beta_1\cdot\beta_2$ can be expressed uniquely in terms of elements 
$\beta^*\in \widetilde{H}^1(P_{\omega_1\sqcup\omega_2})$ dual to generators 
$\beta\in\widetilde{H}^0(P_{\omega_3})$, $\omega_3=[m]\setminus(\omega_1\sqcup\omega_2)$.
If $\widetilde{H}^0(P_{\omega_3})=\mathbb Z$, then  $\beta_1\cdot\beta_2$ is necessarily $\pm \beta_{\omega_3}$.
If $\widetilde{H}^0(P_{\omega_3})=\mathbb Z^2$, then $P_{\omega_3}$ is the set of three disjoint quadrangles, and 
due to Fig. \ref{Cycles-D-mono} $P_{\omega_1}$ and $P_{\omega_2}$ both are disjoint unions of a pentagon and two adjacent pentagons.  There are $3$ such decompositions. The dual classes $\beta_{\omega_3,1}^*$ and $\beta_{\omega_3,2}^*$ in 
$H_1(P_{[m]\setminus\omega_3},\partial P_{[m]\setminus\omega_3})$ are represented by the edges connecting 
a vertex of the quadrangle corresponding to the generator to a vertex of the third quadrangle. The edge connecting 
vertices of the two quadrangles corresponding to generators
gives the element $\pm \beta_{\omega_3,1}^*\pm \beta_{\omega_3,2}^*$. Thus, the three
products $\beta_1\cdot\beta_2$ are equal to $\pm \beta_{\omega_3,1}^*$, $\pm\beta_{\omega_3,2}^*$, and 
$\pm \beta_{\omega_3,1}^*\pm \beta_{\omega_3,2}^*$. 

Up to signs this gives a full description of the ring $H^*(\mathcal{Z}_{As^3},\mathbb Z)$.

For an element $\widetilde{\omega}$, $\omega=\{F_p,F_q\}$, $F_p\cap F_q=\varnothing$,
the complementary subspace in $H^*(\mathcal{Z}_P,\mathbb Q)$ to ${\rm Ann}_H(\widetilde\omega)$ is generated 
by the element  $\widetilde\omega^*\in H^m(\mathcal{Z}_P,\mathbb Q)$ dual to $\widetilde\omega$, and the canonical generators dividing 
$\widetilde\omega^*$. In particular, 
\begin{enumerate}
\item If $F_p$ and $F_q$ are both quadrangles, then $\codim {\rm Ann}_H(\widetilde\omega)=5$;
\item If $F_p$ is a quadrangle, and $F_q$ is a pentagon, then  $\codim {\rm Ann}_H(\widetilde\omega)=9$;
\item If $F_p$ and $F_q$ are both pentagons, then  $\codim {\rm Ann}_H(\widetilde\omega)=21$.
\end{enumerate}   
Also 
\begin{enumerate}
\item If $F_p$ and $F_q$ are both quadrangles, then any element $\omega'\in N_2(P)\setminus\{\omega\}$ is good for $\omega$;
\item If $F_p$ is a quadrangle, and $F_q$ is a pentagon, then an element $\omega'\in N_2(P)\setminus\{\omega\}$ is bad for 
$\omega$ if and only if $\omega'=\{p,r\}$, where $F_r$ is a quadrangle adjacent to $F_q$. There are two such pairs;
\item If $F_p$ and $F_q$ are both pentagons, then they are both adjacent to a unique quadrangle $F_r$. Then 
an element $\omega'=\{s,t\}\in N_2(P)\setminus\{\omega\}$ is bad for $\omega$ if and only if either
$\omega'=\{F_r,F_u\}$, where $F_u\cap F_r=\varnothing$, or $\omega'=\{p,v\}$, where $F_v\ne F_r$ is a quadrangle
adjacent to $F_q$, or $\omega'=\{q,w\}$, where $F_w\ne F_r$ is a quadrangle adjacent to $F_p$,
or $\omega'=\{v,w\}$. There are $7$ bad elements in total.
\end{enumerate}   

\begin{corollary}
The set of elements $\{\pm \widetilde\omega\}$ corresponding to pairs of quadrangles in $As^3$ is mapped
to itself under each automorphism of a graded ring $H^*(\mathcal{Z}_P)$.
\end{corollary}
\begin{corollary}
Under each automorphism of a graded ring $H^*(\mathcal{Z}_P)$ for $P=As^3$ each element $\widetilde\omega$, where $\omega=\{p,q\}$, $F_p$ is a quadrangle, and $F_q$ is a pentagon, is mapped
to the element $\pm\widetilde\omega'+\lambda_1\widetilde\omega_1'+\lambda_2\widetilde\omega_2'$, where $\omega'$ has the same type as $\omega$, and $\omega_1'$, $\omega_2'$ are bad for $\omega'$.
\end{corollary}

\begin{proposition} \label{anneq}
Let $\omega=\{p,q\}$, where $F_p$ is a quadrangle, and $F_q$ is a pentagon, and let $F_u$ and $F_v$ be quadrangles different from $F_p$. Then  for $\omega_1=\{q,u\}$, $\omega_2=\{q,v\}$ and any $\lambda_1$, $\lambda_2$ we have
$$
\dim{\rm Ann}_H(\widetilde\omega)=\dim {\rm Ann}_H(\widetilde\omega+\lambda_1\widetilde\omega_1+\lambda_2\widetilde\omega_2).
$$
\begin{proof}
Denote $4\text{-}5=\widetilde{\omega}$, $4\text{-}4_1=\widetilde{\omega_1}$, $4\text{-}4_2=\widetilde{\omega_2}$, 
$\alpha=4\text{-}5+\lambda_1 4\text{-}4_1+\lambda_2 4\text{-}4_2$.
Then due to Fig. \ref{Cycles-D-mono} 
\begin{enumerate}
\item $4\text{-}4_1^*$ is divided by $5\text{-}55_{11}$, $5\text{-}55_{12}$, $5\text{-}455_{11}$, and $5\text{-}455_{12}$.
Moreover, choosing appropriate signs of generators, we can assume that
$$
4\text{-}4_1^*=5\text{-}55_{11}\cdot 5\text{-}455_{11}=5\text{-}55_{12}\cdot 5\text{-}455_{12}.
$$
\item $4\text{-}4_2^*$ is divided by $5\text{-}55_{21}$, $5\text{-}55_{22}$, $5\text{-}455_{21}$, and $5\text{-}455_{22}$. 
Here all the elements listed above are different. 
Moreover, we can assume that 
$$
4\text{-}4_2^*=5\text{-}55_{21}\cdot 5\text{-}455_{21}=5\text{-}55_{22}\cdot 5\text{-}455_{22}.
$$

\item $4\text{-}5^*$ is divided by $5\text{-}45_1$, $5\text{-}45_2$, $5\text{-}455_{11}$, $5\text{-}455_{21}$, 
$5\text{-}55_{11}$, $5\text{-}55_{21}$, $45\text{-}45_1$ and $45\text{-}45_2$. 
Moreover, choosing appropriate signs of generators, we have 
\begin{gather*}
4\text{-}5^*=5\text{-}45_1\cdot 5\text{-}455_{11}=5\text{-}45_2\cdot 5\text{-}455_{21}=5\text{-}55_{12}\cdot45\text{-}45_1=
5\text{-}55_{22}\cdot45\text{-}45_2.
\end{gather*}
\end{enumerate}
An element belongs to ${\rm Ann}_H(\widetilde\omega)$ if and only if in its expression in terms of the canonical generators and dual elements coefficients at $4\text{-}5^*$ and $5\text{-}45_1$, $5\text{-}45_2$, $5\text{-}455_{11}$, $5\text{-}455_{21}$, $5\text{-}55_{12}$, $5\text{-}55_{22}$, $45\text{-}45_1$ and $45\text{-}45_2$ are zero.

Let $x=\sum_\beta(\varphi(\beta)\beta+\varphi(\beta^*) \beta^*)$ be the expression of an element  $x\in H^*(\mathcal{Z}_P,\mathbb Q)$ in terms of generators and dual elements. We have
\begin{multline*}
x=x_1+\varphi(4\text{-}4_1^*)4\text{-}4_1^*+\varphi(5\text{-}55_{11})5\text{-}55_{11}
+\varphi(5\text{-}55_{12})5\text{-}55_{12}
+\varphi(5\text{-}455_{11})5\text{-}455_{11}+\varphi(5\text{-}455_{12})5\text{-}455_{12}+\\
+\varphi(4\text{-}4_2^*)4\text{-}4_2^*+\varphi(5\text{-}55_{21})5\text{-}55_{21}+\varphi(5\text{-}55_{22})5\text{-}55_{22}+
\varphi(5\text{-}455_{21})5\text{-}455_{21}+\varphi(5\text{-}455_{22})5\text{-}455_{22}+\\
+\varphi(4\text{-}5^*)4\text{-}5^*+\varphi(5\text{-}45_1)5\text{-}45_1+\varphi(5\text{-}45_2)5\text{-}45_2
+\varphi(45\text{-}45_1)45\text{-}45_1+\varphi(45\text{-}45_2)45\text{-}45_2,
\end{multline*}
where $x_1$ is a linear combination of generators and dual elements having zero product with 
$4\text{-}4_1$, $4\text{-}4_2$, and $4\text{-}5$. Also $\alpha=4\text{-}5+\lambda_1 4\text{-}4_1+\lambda_2 4\text{-}4_2$.
We have 
\begin{gather*}
5\text{-}455\cdot 4\text{-}4=45\text{-}45\cdot 4\text{-}5=-5\text{-}55^*;\\
5\text{-}55\cdot 4\text{-}4=5\text{-}45\cdot4\text{-}5=-5\text{-}455^*;\\
5\text{-}455\cdot 4\text{-}5=-5\text{-}45^*;\\
5\text{-}55\cdot 4\text{-}5=-45\text{-}45^*.
\end{gather*}
Therefore, 
\begin{multline*}
-x\cdot\alpha =\left[\varphi(4\text{-}5^*)+\varphi(4\text{-}4_1^*)\lambda_1+\varphi(4\text{-}4_2^*)\lambda_2\right][\mathcal{Z}_P]+\\
+\varphi(5\text{-}455_{11})\lambda_15\text{-}55_{11}^*+
\left[\varphi(5\text{-}455_{12})\lambda_1+\varphi(45\text{-}45_1)\right]5\text{-}55_{12}^*+\\
+\left[\varphi(5\text{-}55_{11})\lambda_1+\varphi(5\text{-}45_1)\right]5\text{-}455_{11}^*+
\varphi(5\text{-}55_{12})\lambda_15\text{-}455_{12}^*+\\
+\varphi(5\text{-}455_{21})\lambda_25\text{-}55_{21}^*+
\left[\varphi(5\text{-}455_{22})\lambda_2+\varphi(45\text{-}45_2)\right]5\text{-}55_{22}^*+\\
+\left[\varphi(5\text{-}55_{21})\lambda_2+\varphi(5\text{-}45_2)\right]5\text{-}455_{21}^*+
\varphi(5\text{-}55_{22})\lambda_25\text{-}455_{22}^*+\\
+\varphi(5\text{-}455_{11})5\text{-}45_1^*+
\varphi(5\text{-}455_{21})5\text{-}45_2^*+
\varphi(5\text{-}55_{12})45\text{-}45_1^*+
\varphi(5\text{-}55_{22})45\text{-}45_2^*.
\end{multline*}
Thus,  $x\cdot\alpha=0$ if and only if 
\begin{multline*}
0=\varphi(5\text{-}455_{11})=\varphi(5\text{-}455_{21})=\varphi(5\text{-}55_{12})=\varphi(5\text{-}55_{22})=\\
=\varphi(5\text{-}55_{12})\lambda_1=\varphi(5\text{-}455_{11})\lambda_1=\varphi(5\text{-}55_{22})\lambda_2=
\varphi(5\text{-}455_{21})\lambda_2=\\
=\varphi(5\text{-}455_{12})\lambda_1+\varphi(45\text{-}45_1)=
\varphi(5\text{-}55_{11})\lambda_1+\varphi(5\text{-}45_1)=\\
=\varphi(5\text{-}455_{22})\lambda_2+\varphi(45\text{-}45_2)=
\varphi(5\text{-}55_{21})\lambda_2+\varphi(5\text{-}45_2)=\\
=\varphi(4\text{-}5^*)+\varphi(4\text{-}4_1^*)\lambda_1+\varphi(4\text{-}4_2^*)\lambda_2.
\end{multline*}
This is equivalent to $9$ linearly independent equations:
\begin{gather*}
\varphi(5\text{-}455_{11})=0;\quad\varphi(5\text{-}455_{21})=0;\quad\varphi(5\text{-}55_{12})=0;\quad\varphi(5\text{-}55_{22})=0;\\
\varphi(45\text{-}45_1)=-\varphi(5\text{-}455_{12})\lambda_1;\quad \varphi(5\text{-}45_1)=-\varphi(5\text{-}55_{11})\lambda_1;\\
\varphi(45\text{-}45_2)=-\varphi(5\text{-}455_{22})\lambda_2;\quad\varphi(5\text{-}45_2)=-\varphi(5\text{-}55_{21})\lambda_2;\\
\varphi(4\text{-}5^*)+\varphi(4\text{-}4_1^*)\lambda_1+\varphi(4\text{-}4_2^*)\lambda_2=0.
\end{gather*}
Thus, $\codim {\rm Ann}_H(\alpha)=9=\codim {\rm Ann}_H(\widetilde\omega)$.
\end{proof}
\end{proposition} 

\section{Acknowledgements}
The  author is grateful to Victor Buchstaber for his encouraging support and attention to this work, 
and Donald Stanley for a fruitful discussion. 
He is grateful to the Fields Institute for Research in Mathematical Sciences (Canada) for providing excellent research 
conditions and support while working on this paper at the Thematic Program on Toric Topology and Polyhedral Products.

\end{document}